\documentclass[10pt,leqno]{amsart}
\usepackage{enumerate}
\usepackage{hyperref}
\usepackage[T1]{fontenc}
\usepackage{cite}
\usepackage{amsmath,amsthm, amsfonts,amssymb}
\usepackage{color}

\theoremstyle{definition}

\theoremstyle{Remark}

\newcommand{\mJ}{\mathcal{J}}

\newcommand{\st}{\cH}

\newcommand{\cinvmB}{[\check{\bc}_1^{-1}m_B]}
\newcommand{\cdinvmB}{[\check{\bc}_2^{-1}m_B]}
\newcommand{\cinvderladjdmB}{[\check{\bc}_1^{-1}\partial_{\la_2}^{j_2}m_B]}
\newcommand{\cdinvderlaujumB}{[\check{\bc}_2^{-1}\partial_{\la_1}^{j_1}m_B]}
\newcommand{\cdinvderzetaujumB}{[\check{\bc}_2^{-1}\partial_{\zeta_1}^{j_1}m_B]}
\newcommand{\mbc}{[\check\bc^{-1}m]}

\newcommand{\mbcuu}{\mathbf{m}^{1,1}}

\newcommand{\MAd}{M_{A_2}}

\newcommand{\cinvMAd}{[\check{\bc}_1^{-1}M_{A_2}]}

\newcommand{\cinvMRd}{[\check\bc^{-1}M_{R_2}]}
\newcommand{\MRd}{M_{R_2}}
\newcommand{\Mpd}{M_{\vp_2}}
\newcommand{\Nu}{L_{1}}
\newcommand{\Nou}{L_{\om_1}}
\newcommand{\Nod}{N_{\om_2}}

\newcommand{\phipuA}{\phi_{1A_2}^p}
\newcommand{\taupuuA}{\tau_{1A_2}^{p,1}}
\newcommand{\taupduA}{\tau_{1A_2}^{p,2}}
\newcommand{\tauptuA}{\tau_{1A_2}^{p,3}}
\newcommand{\phipuu}{\phi_{11}^p}
\newcommand{\taupuuu}{\tau_{11}^{p,1}}
\newcommand{\taupduu}{\tau_{11}^{p,2}}
\newcommand{\tauptuu}{\tau_{11}^{p,3}}
\newcommand{\taupquu}{\tau_{11}^{p,4}}

\renewcommand{\textwidth}{6.0in}                        
\usepackage{amssymb,amsmath}
\usepackage{color}
\usepackage{esint}

\usepackage{ulem}
\renewcommand{\emph}{\normalem}

\theoremstyle{plain}
\newtheorem{theorem}{Theorem}[section]
\newtheorem*{theorem*}{Theorem}

\newtheorem{lemma}[theorem]{Lemma}
\newtheorem{corollary}[theorem]{Corollary}
\newtheorem{proposition}[theorem]{Proposition}
\theoremstyle{remark}
\newtheorem{remark}[theorem]{Remark}
\numberwithin{equation}{section}
\theoremstyle{definition}
\newtheorem{definition}[theorem]{Definition}
\numberwithin{equation}{section}
\theoremstyle{Basic assumptions}
\newtheorem{Basic assumptions}[theorem]{Basic assumptions}
\theoremstyle{notation}

\newcount\quantno
\everydisplay{\quantno=0}\everycr{\quantno=0}
\newcommand\quant{\advance\quantno by1
                      \ifnum\quantno=1\qquad\else\quad\fi\forall }
\newcommand\itemno[1]{(\romannumeral #1)}

\renewcommand\Re{\operatorname{\mathrm{Re}}}
\renewcommand\Im{\operatorname{\mathrm{Im}}}

\newcommand\rest[1]{\kern-.1em
          \lower.5ex\hbox{$\scriptstyle #1$}\kern.05em}

\renewcommand\mod[1]{\vert{#1}\vert}
\newcommand\bigmod[1]{\bigl\vert{#1}\bigr|}
\newcommand\Bigmod[1]{\Bigl\vert{#1}\Bigr|}

\newcommand\norm[2]{{\Vert{#1}\Vert_{#2}}}

\newcommand\bignorm[2]{\,{\big\Vert{#1}\big\Vert_{#2}}}
\newcommand\bignormto[3]{{\big\Vert{#1}\big\Vert^{#3}_{#2}}}
\newcommand\Bignorm[2]{\left.{\Big\Vert{#1}\Big\Vert_{#2}}\right.}
\newcommand\Bignormto[3]{{\Big\Vert{#1}\Big\Vert^{#3}_{#2}}}

\newcommand\bigopnorm[2]{\bigl|\!\bigl|\!\bigl| {#1}
\bigr|\!\bigr|\!\bigr|_{#2}}

\newcommand\wrt{\,\text{\rm d}}

\newcommand\bC{\mathbf{C}}

\newcommand\bW{\mathbf{W}}

\newcommand\BA{\mathbb{A}}

\newcommand\BC{\mathbb{C}}

\newcommand\BG{\mathbb{G}}

\newcommand\BK{\mathbb{K}}
\newcommand\BN{\mathbb{N}} \newcommand\OBN{\overline{\mathbb{N}}}
\newcommand\BR{\mathbb{R}}

\newcommand\BX{\mathbb{X}}

\newcommand\cA{\mathcal{A}}  
\newcommand\fra{\mathfrak{a}}
\newcommand\cB{\mathcal{B}}

\newcommand\cD{\mathcal{D}} 

\newcommand\cE{\mathcal{E}}

                               \newcommand\frg
{\mathfrak{g}}
\newcommand\cH{\mathcal{H}}

\newcommand\cJ{\mathcal{J}}
  
\newcommand\frk{\mathfrak{k}}
\newcommand\cL{\mathcal{L}}  

\newcommand\cM{\mathcal{M}}

\newcommand\frn{\mathfrak{n}}
\newcommand\cO{\mathcal{O}}  

\newcommand\cP{\mathcal{P}}  
\newcommand\frp{\mathfrak{p}}

\newcommand\cU{\mathcal{U}}

\newcommand\dist{{\rm dist}}
\newcommand\sign{{\rm sign}}

\newcommand\al{\alpha}
\newcommand\be{\beta}
    \newcommand\Ga{\Gamma}
\newcommand\de{\delta}
  \newcommand\vep{\varepsilon}

\newcommand\la{\lambda}   
\newcommand\om{\omega}    \newcommand\Om{\Omega}

\newcommand\vp{\varphi}

\newcommand\OV{\overline}
\newcommand\funnyk{k\hbox to 0pt{\hss\phantom{g}}}

\newcommand\lu[1]{L^1(#1)}
\newcommand\lp[1]{L^p(#1)}
\newcommand\laq[1]{L^q(#1)}
\newcommand\ld[1]{L^2(#1)}

\newcommand\lr[1]{L^r(#1)}
\newcommand\ly[1]{L^\infty(#1)}

\newcommand\Mar[2]{\cM ({#1};{#2})}
\newcommand\Marinfty[2]{\cM _\infty({#1};{#2})}
\newcommand\Horm[2]{\cH({#1};{#2})}
\newcommand\Horminfty[2]{\cH_\infty({#1};{#2})}
\newcommand\Cvp[1]{Cv_p(#1)}

\newcommand\bc{\mathbf{c}}

\newcommand\wt{\widetilde}
\newcommand\whH{\widehat{\phantom{G}}\hbox to 0pt{\hss $H$}}

\renewcommand\square{\rlap{\hbox{$\sqcup$}}{\hbox{$\sqcap$}}}
\renewcommand\endproof{\quad\phantom{a}\hfill\square\medskip\goodbreak}
\newcommand\emspace{\hbox to 6pt{\hss}}
\newcommand\ds{\displaystyle}

{\mathsurround=0pt}

\newcommand\rmi{\hbox{\rm (i)}}
\newcommand\rmii{\hbox{\rm (ii)}}
\newcommand\rmiii{\hbox{\rm (iii)}}
\newcommand\rmiv{\hbox{\rm (iv)}}
\newcommand\rmv{\hbox{\rm (v)}}

\newcommand\dtt[1]{\,\frac{\mathrm {d} #1}{ #1}}

\newcommand\One{{\mathbf{1}}}

\newcommand\e{\mathrm{e}}

\newcommand\sft[1]{\wt{#1}}

\newcommand\Spp{\Sigma_0^{+}}
\newcommand\Sppp{\Sigma_s}

\newcommand\TWp{T_p}
\newcommand\TWu{T_{\mathbf{W}_1}}

\newcommand\kA{\kappa_A}
\newcommand\kR{\kappa_R}
\newcommand\ku{\kappa_1}
\newcommand\ko{\kappa_\om}
\newcommand\kAA{\kappa_{A_1A_2}}
\newcommand\kAR{\kappa_{A_1R_2}}
\newcommand\kRA{\kappa_{R_1A_2}}
\newcommand\kRR{\kappa_{R_1R_2}}
\newcommand\kuA{\kappa_{1A_2}}
\newcommand\kuR{\kappa_{1R_2}}

\newcommand\kuu{\kappa_{11}}
\newcommand\kou{\kappa_{\om_11}}
\newcommand\kuo{\kappa_{1\om_2}}
\newcommand\koo{\kappa_{\om_1\om_2}}
\newcommand\kop{\kappa_{\om_1\vp_2}}

\newcommand\alza[1]{\kern+.13em
          \raise.35ex\hbox{$\scriptstyle #1$}}
\newcommand\alzaprima[1]{\kern-.25em
          \raise.2ex\hbox{$\scriptstyle #1$}}

\newcommand\tre{\mathbf{3}}

\DeclareSymbolFont{EUEX}{U}{euex}{m}{n}

\DeclareSymbolFont{euexlargesymbols}{U}{euex}{m}{n}
\DeclareMathSymbol{\intop}{\mathop}{euexlargesymbols}{"52}
     \def\int{\intop\nolimits}

\DeclareSymbolFont{euexsymbols}     {U}{euex}{m}{n}

\title[Marcinkiewicz multipliers on products of symmetric spaces]{Marcinkiewicz-type
multipliers \\
on products of noncompact symmetric spaces}
\author{Stefano Meda and B\l a\.{z}ej Wr\'obel}
\address{Dipartimento di Matematica e Applicazioni, Universit\`{a} di Milano-Bicocca,
via R. Cozzi 53 I-20125, Milano, Italy}
\address{Instytut Matematyczny, Uniwersytet Wroc\l awski, pl. Grunwaldzki 2/4, 50-384 Wroc\l aw, Poland}
\email{stefano.meda@unimib.it}
\email{blazej.wrobel@math.uni.wroc.pl}



\begin{document}

\begin{abstract}
In this paper we prove a Marcinkiewicz-type multiplier result for the
spherical Fourier transform on products of rank one noncompact symmetric spaces.
\end{abstract}

\maketitle
\numberwithin{equation}{section}
\section{Introduction}

A celebrated result of L.~H\"ormander \cite{Ho} states that if
$B$ is a bounded translation invariant operator on $\ld{\BR^n}$
and the Fourier transform $m_B$ of its convolution kernel
satisfies the following Mihlin type conditions
\begin{equation}
\mod{D^I m_B (\xi)}
\leq C \, \mod{\xi}^{-\mod{I}}
\quant \xi \in \BR^n \setminus \{0\}
\end{equation}
for all multi-indices $I$ of length $\leq [\!\![n/2]\!\!]+1$,
then $B$ extends to an operator bounded on $\lp{\BR^n}$ for all
$p$ in $(1,\infty)$, and of weak type~$1$.
The operator $B$ is usually referred to as the \emph{Fourier
multiplier operator} associated to the \emph{multiplier} $m_B$.

Typical operators to which H\"ormander's result applies are Calder\'on--Zygmund
singular integral operators, e.g. Riesz transforms, and spectral
multipliers of $-\Delta$ (the positive Laplacian on $\BR^n$) such as
its purely imaginary powers $(-\Delta)^{iu}$, for $u$ in $\BR$.
However, there are interesting operators to which H\"ormander's theorem does not apply.
A paradigmatic example is the multiplier associated to the double Hilbert transform
in $\BR^2$.  Indeed, such multiplier is ``singular'' on the co-ordinate axes, whereas
a multiplier satisfying H\"ormander conditions may be ``singular'' only at the origin.
Another interesting example is the multiplier
$
m(\xi_1,\xi_2)
= \bigmod{\xi_1}^{2iu} \, \bigmod{\xi_2}^{2iv}
$
in $\BR^{n_1}\times \BR^{n_2}$, where $u$ and $v$ are real numbers,
associated to the operator $(-\Delta_1)^{iu} \otimes (-\Delta_2)^{iv}$,
where $\Delta_1$ and $\Delta_2$ are the standard Laplacians on $\BR^{n_1}$ and $\BR^{n_2}$, respectively. 
Of course, a straightforward argument
shows that
$$
\bigopnorm{(-\Delta_1)^{iu} \otimes (-\Delta_2)^{iv}}{\lp{\BR^{n_1}\times \BR^{n_2}}}
\leq  \bigopnorm{(-\Delta_1)^{iu}}{\lp{\BR^{n_1}}} \,
\bigopnorm{(-\Delta_2)^{iv}}{\lp{\BR^{n_2}}}.
$$
Thus, though the multiplier $m$ above does not fall under the scope of
H\"ormander's multiplier theorem on the product space $\BR^{n_1}\times \BR^{n_2}$,
the operator norms on the right hand side may be estimated by using H\"ormander's theorem
on each of the factor spaces $\BR^{n_1}$ and $\BR^{n_2}$.  However, this argument
does not apply, for instance, to the slightly different multiplier
$
m(\xi_1,\xi_2)
= \big(\mod{\xi_1}^2+\mod{\xi_2}^2\big)^{iu} \, \bigmod{\xi_2}^{iv},
$
which, fortunately, falls under the scope of the celebrated Marcinkiewicz multiplier theorem.

The problem of extending the classical H\"ormander multiplier theorem to the setting
of symmetric spaces of the noncompact type has been considered by various authors
\cite{CS,ST,AL,A1,GMM,I1,I2,MV}, and may be phrased as follows (we refer the reader to Section~\ref{s: Background}
and the references therein for all unexplained notation and terminology).
Suppose that $\BX= \BG/\BK$ is a symmetric space of the noncompact type.
It is well known that if $B$ is a $\BG$-invariant bounded linear operator on $\ld{\BX}$, 
then there exists a $\BK$--bi-invariant tempered distribution $k_B$ on $\BG$ such that 
$Bf = f*k_B$ for all $f$ in $\ld{\BX}$ (see \cite[Prop.~1.7.1 and Ch.~6.1]{GV} for details).
We call $k_B$ the \emph{kernel} of $B$.
We denote its spherical Fourier transform~$\wt k_B$
by $m_B$ and call it the \emph{spherical multiplier} associated to $B$.
Clearly $\wt k_B$ is a bounded Weyl-invariant function on $\fra^*$. 

A well known result of J.L.~Clerc and E.M.~Stein \cite{CS} states that 
if $B$ is $\BG$-invariant bounded linear operator on $\lp{\BX}$ for some $p$ in $(1,\infty)$, then $m_B$
continues analytically to a bounded Weyl-invariant function in a tube $T_p$ over a suitable polyhedron in $\BC^\ell$, where 
$\ell$ is the real rank of $\BX$.
The best sufficient condition available in the literature in the case where $\BX$ has real rank one 
is due to A.D.~Ionescu \cite{I1}, who proved that if 
$p$ is in $(1,\infty)\setminus\{2\}$, $m$ is a bounded Weyl-invariant holomorphic function in $\TWp$ 
(in this case  $\TWp := \{\zeta \in \BC:  \bigmod{\Im \zeta} < \de(p) \, \mod{\rho}\}$,
where $\de(p) = \mod{2/p-1}$ and $\rho$ is half the sum of all positive roots with multiplicity)
and satisfies the following inequalities
$$
\bigmod{\partial^jm(\zeta)}
\leq C\, \Big(\min \big[\mod{\la-i\de(p)\rho},\mod{\la+i\de(p)\rho}\big]\Big)^{-j}
\quant \zeta \in \TWp,
$$
for $j=0,1,\ldots,N$, with $N$ large enough,
then the associated multiplier operator is bounded on $\lp{\BX}$ and, by interpolation
with the trivial $L^2$ result, on $\lr{\BX}$ for all $r$ such that $\bigmod{1/r-1/2}\leq \bigmod{1/p-1/2}$. 

Versions of this result in the higher rank case may be found in \cite{A1,I2,MV}.  In particular, they apply 
to the case where $\BX = \BX_1\times \BX_2$, and $\BX_1$ and $\BX_2$ have real rank one.
However, none of these results applies to the simple operator $\cL_1^{it}\, (\cL_1+\cL_2)^{iu} \cL_2^{iv}$, 
where $\cL_1$ and $\cL_2$ denote the Laplace--Beltrami operators on~$\BX_1$ and $\BX_2$, respectively, and 
$u$, $v$ and $t$ are non-zero real numbers.  
To the best of our knowledge, the problem of establishing Marcinkiewicz type
multiplier theorems for the spherical Fourier transform on noncompact
symmetric spaces has not been considered yet.
The purpose of this paper is to fill this gap and prove the following.

\begin{theorem} \label{t: main} 
Suppose that $\BX_1$ and $\BX_2$ are rank one symmetric spaces of the noncompact type
of dimensions $n_1$ and $n_2$, respectively.
Suppose that $p$ is in $(1,\infty)\setminus \{2\}$, $N_1>(n_1+3)/2$, 
$N_2>(n_2+3)/2$, and that $B$ is a $\BG$-invariant operator such that  the estimate
\begin{align*}
&   \bigmod{\partial_{\la_1}^{j_1}\partial_{\la_2}^{j_2} m_B(\la_1,\la_2)}\\
&   \le C \, \Big(\min \big[\mod{\la_1-i\de(p)\rho_1},\mod{\la_1+i\de(p)\rho_1}\big]\Big)^{-j_1}
	\cdot  \Big(\min \big[\mod{\la_2-i\de(p)\rho_2},\mod{\la_2+i\de(p)\rho_2}\big]\Big)^{-j_2} 
\end{align*}
holds for all $j_1 \leq N_1$ and $j_2\leq N_2$, uniformly in $(\la_1,\la_2) \in T_p^{(1)} \times T_p^{(2)}$.
Then $B$ extends to a bounded operator on $\lp{\BX_1\times \BX_2}$.
\end{theorem}

\noindent
A comparison between the condition for $m_B$ in Theorem~\ref{t: main} and the condition considered in 
\cite[Theorem~4.1]{I3} (restricted to the reduced case) may be found in Remark~\ref{rem: comparison MW Ion}.
We emphasize that the part of the kernel $k_B\circ \exp$ (corresponding to operators $B$ satisfying the 
assumptions of the theorem above) may be much more singular near the walls of the Weyl chamber 
than the kernels of operators satisfying H\"ormander type conditions.  
In particular, Ionescu proved that the latter can be estimated using Herz's majorizing principle.  
In our case, we need a more sophisticated transference principle that takes into account the oscillations of the kernel.

\noindent
Our result extends to reducible symmetric spaces
of the form $\BX_1\times \cdots \times \BX_m$, where $m\geq 3$ and each of the factors
is a rank one symmetric space of the noncompact type.  The modifications of the proofs
needed to cover this more general case are straightforward, though lengthy, and are omitted.  

Our main result suggests that conditions of Marcinkiewicz type are worth considering in the setting
of higher rank noncompact symmetric spaces.  We shall come back to this problem in the near future.

Our paper is organized as follows.  Section~\ref{s: Background} contains some notation
concerning symmetric spaces, and a few preliminary results concerning
the spherical Fourier analysis on rank one symmetric spaces.
Section~\ref{s: Boundedness results for convolution operators} contains some
criteria for the boundedness of various convolution operators on $L^p$ that
are frequently used in the sequel.  This includes some consequences of a transference result
for left invariant operators on semidirect products of groups \cite{CMW}, which generalizes
previous results of Ionescu \cite{I1}.
Section~\ref{s: Statement} contains the statement of our main result, Theorem~\ref{t: main},
together with an outline of its proof.  Here we introduce the splitting of the
operator $B$ as a sum of three operators $B_0$, $B_1$ and $B_2$.
Sections~\ref{s: loc-loc}, \ref{s: loc-glob} and \ref{s: glob-glob}
contain all the details concerning the analysis of the operators $B_0$, $B_1$ and $B_2$, respectively.

We will use the ``variable constant convention'', and denote by $C,$
possibly with sub- or superscripts, a constant that may vary from place to
place and may depend on any factor quantified (implicitly or explicitly)
before its occurrence, but not on factors quantified afterwards.

\section{Background and preliminary results} \label{s: Background}

\subsection{Preliminaries on symmetric spaces}
Suppose that $\BG$ is a noncompact semisimple Lie groups with finite centre.
Denote by $\BK$ a maximal compact subgroup of $\BG$
and consider the associated Riemannian symmetric space of the noncompact type $\BX:=\BG/\BK$.
We briefly summarize the main features of spherical harmonic analysis on $\BX$.
The books of S.~Helgason \cite{H1,H2} and Gangolli and Varadarajan \cite{GV}
are basic references concerning the analysis on symmetric spaces.

Denote by $\theta$ a Cartan involution of the Lie algebra $\frg$ of $\BG$,
and write $\frg = \frk \oplus \frp$ for the corresponding
Cartan decomposition. Let $\fra$ be a maximal
abelian subspace of $\frp$, and denote by $\fra^*$ its dual space,
and by $\fra_\BC^*$ the complexification of $\fra^*$.
Denote by $\Sigma$ the set of (restricted) roots
of $(\frg,\fra)$; a choice for the set of positive roots
is written $\Sigma^+$, and $\fra^+$ denotes the corresponding Weyl chamber.
The vector $\rho$ denotes $(1/2)\sum_{\al \in \Sigma^+} m_{\al} \,\al$,
where $m_{\al}$ is the multiplicity of $\al$.
We denote by $\Sppp$ the set of simple
roots in $\Sigma^+$, and by $\Spp$ the set of indivisible
positive roots.
Denote by $W$ the Weyl group of $(\BG,\BK)$,
and by $\bW$ the interior of the convex hull of the points $\{w\cdot \rho: w \in W\}$.
For each $p$ in $(1,\infty) \setminus \{2\}$, we let
 \begin{equation*}\de(p)=\mod{2/p-1},\end{equation*}
then we denote by $\bW_p$ the dilate of $\bW$ by $\de(p)$,
and by $\TWp$ the corresponding tube $\fra^*+i\bW_p$ in $\fra_\BC^*$.

The spherical Fourier transform of an integrable function $g$ on $\BG$
is the function $\st g$, defined by
$$
\st g (\la)
= \int_{\BG} g(x) \, \vp_{-\la}(x) \wrt x
\quant \la \in \fra_\BC^*
$$
where $\vp_{\la}$ denote the spherical functions on $\BG$.
We shall write also $\sft g$ in place of $\st g$.  Recall that $\st g$ is Weyl-invariant.
For ``nice'' $\BK$--bi-invariant functions $g$, the inversion formula is given by
\begin{equation*}
g(x)
= \int_{\fra^*} \sft g(\la)\, \vp_{\la}(x) \, \wrt \nu(\la),
\end{equation*}
where the Plancherel measure $\nu$ is given by
$
\wrt \nu(\la)
= \bigmod{\bc(\la)}^{-2} \wrt\la,
$
and $\bc$ is the Harish-Chandra function.
In the rank one case it is well known that there exists a constant $C$ such that
\begin{equation} 
\label{f: HCest}
\bigmod{\partial^{j} \check\bc^{-1}(\la)}
\leq  C \,  \big(1+|\la|\big)^{(n-1)/2-j}
\quant \la: 0\leq \Im \la \leq \rho
\end{equation}
where $n$ denotes the dimension of $\BX$ (see, for instance, \cite[Lemma 4.2]{ST} and \cite[Appendix A]{I1}).
The spherical Fourier transform extends to $\BK$--bi-invariant
tempered distributions on $\BG$ (see, for instance, \cite[Ch.~6.1]{GV}).

\subsection{Background on symmetric spaces of rank one} \label{s: BX rank one}
In this section we consider the case where $\BX$ has real rank one, i.e. the algebra
$\fra$ is one dimensional.  We denote by $\BA$ the multiplicative group $\exp\fra$,
which is obviously isomorphic to the additive group of the vector space $\fra$.
It is convenient to choose a particular isomorphism between $\BA$ and $\fra$,
which we now describe.  Denote by $\al$ the unique simple positive root of the pair $(\frg,\fra)$.  
Denote by $H_0$ the unique vector in $\fra$ such that
$\al(H_0) = 1$, and normalize the Killing form of $\frg$ so that $\bigmod{H_0}^2 = 1$.
Every vector in $\fra$ is of the form $tH_0$, with $t$ in~$\BR$, and every element
of $\BA$ is then of the form $\exp(tH_0)$.  The Weyl chamber is $\fra^+ := \{tH_0: t>0\}$, and we
set $\BA^+ := \exp (\fra^+)$ and $\BA^- := \exp (-\fra^+)$.  We shall often write $a$ for $\exp(tH_0)$ and $\log a$ for $tH_0$.

The root system is either of the form $\{-\al,\al\}$ or of the form
$\{-2\al,-\al, \al,2 \al\}$.  We denote by~$m_\al$ and $m_{2\al}$ the
multiplicities of $\al$ and $2 \al$, respectively.  Observe that $m_\al+m_{2\al}+1=n$
and $m_{2\al} = 0$ in the case where $2\al$ is not a root.  
Clearly $2\rho = (m_\al+2m_{2\al}) \, \al$, and $2\mod{\rho} = m_\al+2m_{2\al}$.  
We consider the Lie algebras $\frn := \frg_{\al} + \frg_{2\al}$ and $\OV\frn := \frg_{-\al} + \frg_{-2\al}$,
where $\frg_\be$ denotes the root space associated to the root $\be$, and the corresponding connected and 
simply connected nilpotent Lie groups $\BN$ and $\OBN$, respectively.  
If $a$ is in $\BA$ and $\la$ belongs to $\fra^*$, we 
write $a^\la$ instead of $\e^{\la(\log a)}$.  Define
$$
\de(a)
:= 2^{-2\mod{\rho}} \, \big[a^\al-a^{-\al}\big]^{m_\al}  \,  \big[a^{2\al}-a^{-2\al}\big]^{m_{2\al}}
\quant a \in \BA^+
$$
 and note that  $\delta(a)$ is of order $(\log a)^{n-1}$ when $a$ is close to $1$ and of order $a^{2\rho}$ when $a$ is large.
 
We recall the following integration formula in Cartan co-ordinates:
\begin{equation*} 
\int_\BG f(x) \wrt x
	=  c_\BG\, \int_\BK\int_{\BA^+} \int_\BK f(kak') \, \de(a) \, \wrt k \wrt a \wrt k',
\end{equation*}
where the Haar measure on $\BK$ is normalized so as to have total mass $1$, $\wrt a$ denotes
the Lebesgue measure on $\BA$ and $c_\BG$ is a constant depending on the group $\BG$.

\subsubsection{Expansions of spherical functions}
\textsl{Near the origin} the spherical function $\vp_\la$ admits the following expansion
\cite[Theorem 2.1]{ST}. There exists a positive real number $r_0$
such that if $0< t\leq r_0$ and $L$ is any positive integer
\begin{equation} \label{f: SpherDecSmall}
\varphi_{\la}(a)
= A(\la,a) + R(\la,a),
\end{equation}
where
\begin{equation} \label{f: SpherDecSmall I}
\begin{aligned}
A(\la,a)
& = c_0 \, w(a) \, \cJ_{n/2-1}(\la t) \\ 
R(\la,a)
& = c_0 \, w(a) \, \sum_{\ell=1}^{L}\, t^{2\ell}\, a_{\ell}(t)\, \cJ_{n/2+\ell-1}(\la t)+E_{L+1}(\la t);
\end{aligned}
\end{equation}
here $a = \exp(tH_0)$, $\ds w(a) := \Big[\frac{\log^{n-1}a}{\de(a)}\Big]^{1/2}\!\!,$
$\cJ_{\mu}(x)=C_{\mu}z^{-\mu}J_{\mu}(z),$ with $J_\mu$ being the Bessel function of the first kind and order $\mu$, and
\begin{equation*} 
|a_{\ell}(t)|
          \leq C_{L}  
\qquad\hbox{and}\qquad
\begin{cases}
|E_{L+1}(\la t)| \leq C_L \, t^{2(L+1)} 
	& \hbox{if $|\la t|\leq 1$}\\
|E_{L+1}(\la t)| \leq C_L \, t^{2(L+1)}\, (\la t)^{-(n-1)/2-L-1} 
	& \hbox{if $|\la t|> 1$}.
\end{cases}
\end{equation*}

\noindent
\textsl{At infinity}, the celebrated Harish-Chandra expansion states that for $a$ in $\BA^+$ 
\begin{equation} \label{f: HC expansion}
\vp_{\la}(a )
= a^{-\rho+i\la}\, \bc(\la)\, [1+a^{-2\al}\, \om(\la,a)]
      +a^{-\rho -i\la}\, \bc(-\la)\, [1+a^{-2\al}\, \om(-\la,a)],
\end{equation}
where $\om(\la,a) = \sum_{\ell=1}^{\infty}\, \Ga_{\ell}(\la)\, a^{-2\ell\al}$
and there exist constants $b_\al$ such that the coefficients
$\Ga_\ell$ satisfy the following estimates \cite[Appendix A]{I1}
\begin{equation*} 
\bigmod{\partial^j_{\la}\Ga_{\ell}(\la)}
\leq C\,  \ell^{b_j}\,  \big(1+|\Re \la|\big)^{-j},
\end{equation*}
for $\ell\geq 1$ and $0\leq \Im \la \leq \rho$. As a consequence Ionescu obtained also the bound
\begin{equation} \label{f: estIonescu}
\bigmod{\partial^j_{\la}\om(\la,a)}+\bigmod{a\partial_a\partial^j_{\la}\om(\la,a)}
\leq C_j \, (1+|\Re \la|)^{-j}
\asymp (1+|\la|)^{-j},
\end{equation}
for all integers $j=0,1,\ldots$ whenever  $\al(\log a)\geq 1/2$ and $0\leq \Im \la\leq \rho.$

\begin{remark}  \label{rem: spherical inversion}
Suppose that $m$ is a Weyl-invariant function on $\fra^*$.  
By \eqref{f: HC expansion}, we may write
$$
\begin{aligned}
\cH^{-1}m (a)
	=  \int_{\fra^*} \!\! a^{i\la-\rho}\, \big[1+a^{-2\al}\om(\la,a) \big] \, \mbc (\la)  \wrt \la 
	 +\int_{\fra^*} \!\! a^{-i\la-\rho}\, \big[1+a^{-2\al}\om(-\la,a) \big] \, \bc^{-1} (\la) \wrt \la
\end{aligned}
$$
for every $a$ in $\BA^+$. 
Changing variables ($-\la = \la'$) in the second integral above and using the Weyl invariance
of $m$ shows that the second integral is equal to the first.   Therefore
\begin{equation} \label{f: inversion infinity}
\cH^{-1}m (a)
= { 2} \int_{\fra^*} a^{i\la-\rho}\, \big[1+a^{-2\al}\om(\la,a) \big] \, \mbc (\la)  \wrt \la 
\quant a \in \BA^+.
\end{equation}
\end{remark}

\subsubsection{Estimates for multiplier operators}
In this subsection we collect several results, most of which already known,
concerning spherical Fourier multipliers on rank one symmetric spaces.  
These results will be used later to estimate various error
terms arising in the study of multiplier operators satisfying Marcinkiewicz--Mihlin
conditions on reducible symmetric spaces.
First, we need the following definition.  For each $p$ in $(1,\infty) \setminus\{2\}$
we denote by $\Theta_p: \fra_\BC^*\to [0,\infty)$ the function, defined by
\begin{equation} \label{f: Theta}
\Theta_p(\la)
= \min \Big[\mod{\la-i\de(p)\rho},\mod{\la+i\de(p)\rho}\Big].
\end{equation}
Notice that $\Theta_p(\la)^2 \asymp (\Re\la)^2 + \dist \big(\Im\la, \bW_p^c\big)^2$ on $\OV T_p$, and that  
$\Theta_p(\la) \asymp \mod{\la}$ for $\mod{\la}$ large.  Recall that in this case 
$T_p = \{\la \in \BC: \bigmod{\Im \la} < \de(p) \, \mod{\rho} \}$.  

\begin{definition} \label{def: Horm}
Suppose that $p$ is in $(1,\infty)\setminus \{2\}$, that $N$ is a positive integer
and that $m$ is a bounded Weyl-invariant holomorphic function on the strip $\TWp$.
We say that $m$ satisfies \emph{H\"ormander condition of order $N$ on $\TWp$} [resp.
\emph{of order $N$ at infinity on $\TWp$}] if 
$$
\bignorm{m}{\Horm{\TWp}{N}}
:= \sup_{\zeta\in T_p} \Theta_p(\zeta)^j\bigmod{\partial^jm(\zeta)}
<\infty 
$$
[resp. $\bignorm{m}{\Horminfty{\TWp}{N}}
:= \sup_{\zeta\in T_p} \big(1+\bigmod{\zeta}\big)^j\bigmod{\partial^jm(\zeta)} <\infty$]
for $j=0,1,\ldots,N$. Then $\Horm{\TWp}{N}$ [resp.\ $\Horminfty{\TWp}{N}$] is defined as the spaces of those 
functions $m$ for which $\bignorm{m}{\Horm{\TWp}{N}}$ [resp. $\bignorm{m}{\Horminfty{\TWp}{N}}$] is finite.
\end{definition}

\begin{definition} \label{def: Horminfty}
Suppose that $N$ is a positive integer and that $m$ is a bounded function on $\fra^*$.
We say that $m$ satisfies a  \emph{H\"ormander condition of order $N$ on $\fra^*$} [resp. a 
\emph{of order $N$ at infinity on $\fra^*$}] if 
$$
\bignorm{m}{\Horm{\fra^*}{N}}
:= \sup_{\zeta\in \fra^*} \bigmod{\zeta}^j\bigmod{\partial^jm(\zeta)}
<\infty 
$$
[resp. $\bignorm{m}{\Horminfty{\fra^*}{N}}
:= \sup_{\zeta\in \fra^*} \big(1+\bigmod{\zeta}\big)^j\bigmod{\partial^jm(\zeta)}
<\infty$]
for $j=0,1,\ldots,N$.  The space $\Horm{\fra^*}{N}$ [resp. 
$\Horminfty{\fra^*}{N}$] is defined as the space of those 
functions $m$ for which $\bignorm{m}{\Horm{\fra^*}{N}}$ [resp. $\bignorm{m}{\Horminfty{\fra^*}{N}}$] is finite.
\end{definition}

\noindent
Suppose that $m$ is holomorphic in $T_p$.   
For each $v$ in $\big(-\de(p), \de(p)\big)$, set $m_v(\la) := m(\la+iv\rho)$, i.e.,
$m_v$ is the restriction of $m$ to $\fra^*+iv\rho$.  
The following observation is presumably known to the experts.

\begin{proposition} \label{p: extension}
Suppose that $m$ is in $\Horm{\TWp}{N}$.  Then $m$ extends to a function, still denoted $m$, in
$C^N \big(\OV{T_p}\setminus\{\pm i \de(p) \rho\}\big)$, and the boundary values
$m(\cdot\pm i\de(p)\rho)$ 
satisfy a Mihlin--H\"ormander condition of order $N$ on $\fra^*$. 
\end{proposition}

\begin{proof}
Note that if $0 < u < v < \de(p)$, then 
$
\ds m(\la+iv\rho) - m(\la+iu\rho) = i\rho \int_u^v m'(\la+is\rho) \wrt s,
$
so that for each $\de >0$
$$
\sup_{\mod{\la} \geq \de}\, \bigmod{m_v(\la)-m_u(\la)}
\leq \mod{\rho} \bignorm{m}{\Horm{T_p}{N}} \, \int_u^v \dist\big(\la+is\rho, i\de(p)\rho\big)^{-1}  \wrt s
\leq \frac{\mod{\rho}}{\de}\, \bignorm{m}{\Horm{T_p}{N}} \, (v-u).
$$
Denote by $m_{\de(p)}$ the uniform limit on $\fra^*\setminus (-\de,\de)$ 
of $\{m_{u_j}\}$, where $\{u_j\}$ is any sequence such that $u_j\uparrow \de(p) \rho$ (it is straightforward to check
that the limit does not depend on the particular sequence $\{u_j\}$ chosen).  Clearly $m_{\de(p)}$
is continuous on $\fra^*\setminus(-\de,\de)$.  Since this argument works for every $\de>0$,
$m_{\de(p)}$ is continuous on $\fra^*\setminus\{0\}$.
By arguing similarly we may define $m_{-\de(p)}$, which is also continuous on $\fra^*\setminus\{0\}$. 
With a slight abuse of notation, denote by $m$ the function on $\OV{T_p}\setminus{\pm i\de(p)\rho}$
that agrees with the given function $m$ on $T_p$, and is equal to $m_{\de(p)}$
on $\fra^*+i\de(p)\rho$ and to $m_{-\de(p)}$ on $\fra^*-i\de(p)\rho$.  A classical argument
now shows that $m$ is, in fact, a continuous function on $\OV{T_p}\setminus\{\pm i \de(p) \rho\}$.

By iterating the argument above, we see that for each $j\in \{1,\ldots, N\}$ the derivative $m^{(j)}$ 
extends to a continuous function, still denoted by $m^{(j)}$, on $\OV{T_p}\setminus\{\pm i \de(p) \rho\}$. 
It remains to prove that $m_{\de(p)}^{(j)}$ is the $j^{\textrm{th}}$ derivative
of $m_{\de(p)}$.  It is straightforward to do so in the sense of distributions,
hence in the classical sense, in $\fra^*\setminus\{0\}$,  because of the continuity of 
the functions $m_{\de(p)}^{(j)}$.  
\end{proof}

\noindent
Throughout the paper $\bC$ denotes a smooth even function on $\BR$
that is equal to~$1$ on $[-1,1]$, vanishes outside of the interval $[-2,2]$ and satisfies $0\leq \bC \leq 1$. 
Define a $\BK$--bi-invariant function $\Phi$ on $\BG$ by 
$$
\Phi(a) 
:= \bC \big(\al(\log a) \big)
\quant a \in \BA.  
$$
For $m$ in $\Horm{\TWp}{N}$, define the $\BK$--bi-invariant functions $\kA$, $\kR$, $\ku$, 
and $\ko$ on $\BG$ by
\begin{equation} \label{f: kernels rank one}
\begin{aligned} 
\kA(a)
& := \Phi(a) \, \int_{\fra^*} A(\la,a)\, m(\la)\, \wrt \nu(\la)  \\
\kR(a)
& := \Phi(a) \, \int_{\fra^*} R(\la,a) \, m(\la) \wrt\nu(\la)  \\
\ku(a)
& := [1-\Phi(a)] \,  \int_{\fra^*} a^{i\la-\rho} \, \mbc(\la) \wrt \la  \\
\ko(a)
& := [1-\Phi(a)] \,  \int_{\fra^*} a^{i\la-\rho-2\al} \, \om(\la,a)\,\mbc(\la)\wrt \la
\end{aligned} 
\quant a \in \BA^+. 
\end{equation}
Note that, by \eqref{f: SpherDecSmall} and \eqref{f: inversion infinity},
\begin{equation*}
\Phi \,\cH^{-1}m 
=  \kA+\kR
\qquad\hbox{and}\qquad 
\big(1-\Phi\big) \,\cH^{-1}m 
=  2\ku+2\ko.
\end{equation*}

\begin{lemma} \label{l: one dim multipliers loc}
Suppose that $N> (n+3)/2$ and that $1<p<2$.  Then there exists
a constant $C$ such that for every Weyl-invariant function $m$ in $\Horminfty{\fra^*}{N}$ the following hold:
\begin{enumerate}
\item[\itemno1]
$
\bignorm{\kA}{\Cvp{\BX}}
\leq C \,  \bignorm{m}{\Horminfty{\fra^*}{N}};
$
\item[\itemno2]
$\ds \bigmod{\kR(a)} \leq C\, \bignorm{m}{\Horminfty{\fra^*}{N}} \, \frac{\Phi(a)}{(\log a)^{n-1}}$, whence 
$
\bignorm{\kR}{\lu{\BX}}
\leq  C\,  \bignorm{m}{\Horminfty{\fra^*}{N}}.
$
\end{enumerate}
\end{lemma}

\begin{proof}
The proof is implicit in \cite[Sections~4 and~5]{ST}.  
In fact, the stronger assumption that $m$ is in $\Horminfty{T_p}{N}$ is made therein.  
However, $\kA$ and $\kR$ are supported near the origin, and 
it is straightforward to check that the arguments in \cite{ST} carry over to the case where 
$m$ is in $\Horminfty{\fra^*}{N}$, because there is no need to shift the contour of integration to
obtain local estimates.
\end{proof}

\noindent
The group $\BG$ also admits the Iwasawa decomposition $\BG= \OBN\BA\BK$.  The corresponding integration 
formula reads
\begin{equation*} 
\int_\BG f(x) \wrt x
	=  c_\BG\, \int_{\OBN} \int_{\BA} \int_{\BK} f(v ak) \, a^{2\rho} \, \wrt v \wrt a\wrt k.
\end{equation*}
For each $g$ in $\BG$ we denote by $[g]_+$ and $\exp\big[H(g)\big]$ the middle components of $g$ in the 
Cartan $\BG=\BK \BA^+ \BK$ and the Iwasawa $\BG= \OBN\BA\BK$ decompositions, respectively. 
Recall that in the rank one case $H(v)$ is in $\OV{\fra^+}$ for every $v$ in $\OBN$ \cite[Corollary~6.6]{H1}. 
By \cite[Lemma 3]{I2}, 
\begin{equation} \label{f: decom Iwas}
[vb]_+
= b  \,\exp\big[H(v)\big] \, \exp\big[E(v,b)H_0\big]  
\qquad\hbox{and}\quad 0\leq E(v,b)\leq 2b^{-2}
\end{equation}
for every $b\in \BA^+$ and every $v\in\OBN$. For further reference we remark that \eqref{f: decom Iwas} implies that $[vb]_+\ge b,$ for $b\in \BA^+, v\in\OBN$. We set $P(v) := \e^{-\rho H(v)}$.  
It is well known that for any $q>1$ both $P(v)^q$ and $H(v)P(v)^q$ are in $\lu{\OBN}$.  This follows 
by an explicit computation starting from \cite[Theorem~6.1~\rmii]{H3}.  
Now we outline the strategy of the main result in \cite{I1}: the reader will find all the details therein.
We warn the reader that our notation is different from that employed in \cite{I1}.  This change is motivated
by the need of keeping the formulae as compact as possible, in view of the application of Ionescu's strategy 
to the product case.  

We need the following notation.  For $q\in [1,\infty],$ we let $Cv_q({\BA})$ be the space of bounded operators on $\laq{\BA}$
	that commute with translations (equivalently the space of convolutors), 
	equipped with the operator norm on $\laq{\BA}$.
Denote by $\chi_{\BA^+}$ and $\chi_{\BA^-}$ the characteristic functions of $\BA^+$ and $\BA^-$, respectively; 
we extend them to $\OBN\BA$
by requiring that $\chi_{\BA^+}(va) = \chi_{\BA^+}(a)$ and $\chi_{\BA^-}(va) = \chi_{\BA^-}(a)$.  
Recall that any function on $\BX$
may be identified with a function on $\OBN\BA$; in particular, $\ku(vb) = \ku([vb]_+)$.
It will be clear from the context whether we think of $\ku$ as a function on $\OBN\BA$ or
as a $\BK$-invariant function on $\BX$.  
For $m$ in $\Horm{T_p}{N}$ define, at least formally, 
\begin{equation} \label{f: phip}
\phi_p(a)
:= 
\big[1-\Phi(a)\big]\, a^{\de(p)\rho} \int_{\fra^*} a^{i\la} \, \mbc(\la)\wrt\la
\quant a \in \BA.   
\end{equation}
Notice that if $m$ is bounded and rapidly decreasing at infinity, then the integral above is
convergent. 
Notice also that $\phi_p$ may not be Weyl-invariant on $\BA$. 
We shall often use the following formulae for $\phi_p$, obtained from \eqref{f: phip} 
by moving the contour of integration from $\fra^*$
to $\fra^*+i \de(p)\rho$ and
to $\fra^*+i \big[\de(p)\rho-\rho^{p,\vep(a)}\big]$, respectively: 
\begin{equation} \label{f: phipI}
\begin{aligned}
\phi_p(a)
& = \big[1-\Phi(a)\big] \, \int_{\fra^*} a^{i\la} \, \mbc_{\de(p)\rho}(\la)\wrt\la \\
\phi_p(a)
& = \big[1-\Phi(a)\big]\, \e^{\mod{\rho} \vep(a) \sign(\log a)}
	\int_{\fra^*} a^{i\la} \, \mbc_{\rho^{p,\vep(a)}}(\la)\wrt\la
\end{aligned}
\quant a \in \BA,   
\end{equation}
where $\rho^{p,\vep(a)} = \big(\de(p)-\vep/ \mod{\log a}\big) \rho$.  The following lemma
is due to Ionescu \cite{I1}.  Our proof is slightly different from the original one.  

\begin{lemma} \label{l: phip basic}
Suppose that $m$ is in $\Horm{T_p}{N}$ and it is rapidly decreasing at infinity.  Then 
$$
\begin{aligned}
\bigmod{\phi_p(a)}
\leq C \bignorm{m}{\Horm{\TWp}{N}}\, \frac{1-\Phi(a)}{\log a} 
\quad\hbox{and}\quad
\bigmod{a \partial_a\phi_p(a)}
& \leq C \bignorm{m}{\Horm{\TWp}{N}}\, \frac{1-\Phi(a)}{\log^2 a}
\quant a \in~\BA^+.
\end{aligned}
$$
Furthermore $\chi_{\BA^-}\phi_p$ is in $\lu{\BA}$, and $\chi_{\BA^+}\phi_p$ 
and $\phi_p$ are in $Cv_q(\BA)$ for all $q$ in $(1,\infty)$.  The convolution norm of these kernels
is controlled by $C \bignorm{m}{\Horm{\TWp}{N}}$.
\end{lemma}

\begin{proof}
The pointwise estimates, which are versions on $\BA$ of Calder\'on--Zygmund inequalities,
follow from repeated integration by parts as in \cite[p.~114--115]{I1}.  

According to the classical theory of singular integral operators, 
in order to conclude that $\phi_p$ is in $\Cvp{\BA}$,
one has to prove that $\phi_p$ is in $Cv_2(\BA)$.  This is implicit in \cite[p.~114-115]{I1}.
We provide a different proof, which we shall generalize later to the reduced case.
By \eqref{f: phipI}, $\phi_p (a) = [1-\Phi(a)]\, \Xi (a)$, where 
$
\ds\Xi(a)
:= \int_{\fra^*} a^{i\la} \, \mbc_{\de(p)\rho}(\la)\wrt\la.
$ 
Set $\Psi(\la) := \bC\bigl(\la(H_0)\bigr)$, where $\bC$ is the function defined slightly 
above formulae \eqref{f: kernels rank one}.
Write $(1-\Psi) + \Psi$ instead of $1$ inside the integral above.  Correspondingly, we have the 
decomposition $\Xi = \Xi^{1-\Psi} + \Xi^\Psi$.  
By integrating by parts $N$ times, we see that 
$$
\big[1-\Phi(a)\big] \, \bigmod{\Xi^{1-\Psi}(a)} 
\leq C \bignorm{m}{\Horm{T_p}{N}} \frac{1-\Phi(a)}{\mod{\log a}^N},
$$ 
which belongs to $\lu{\BA}$, hence to $Cv_q(\BA)$ for all $q$ in $(1,\infty)$.  

Since the Mellin transform of $\Xi^\Psi$ is in $\Horm{\fra^*}{N}$, $\Xi^{\Psi}$ is in $Cv_q(\BA)$
for all $q$ in $(1,\infty)$, by the classical Mihlin--H\"ormander theorem.  
It follows from \cite[Theorem~3.4]{Co} that $\Phi$ is a pointwise multiplier of $Cv_2(\BA)$, whence so is 
$1-\Phi$.  
Thus, we may conclude that $(1-\Phi)\, \Xi^{\Psi}$ is in $Cv_q(\BA)$ for all $q$ in $(1,\infty)$.    

Altogether, we have proved that $\phi_p$ is in $Cv_q(\BA)$ for all $q$ in $(1,\infty)$.  
By integrating by parts $N$ times in the integral in \eqref{f: phip}, we see that 
$$
\phi_p(a)
= \frac{1-\Phi(a)}{(i\log a)^N}\, a^{\de(p)\rho} \int_{\fra^*} a^{i\la} \, \partial_\la^N\mbc(\la)\wrt\la.
$$
Thus, 
$$
\begin{aligned}
\bigmod{\phi_p(a)}
& \leq C \bignorm{m}{\Horm{\TWp}{N}}\, \frac{1-\Phi(a)}{\mod{\log a}^N}\, a^{\de(p)\rho} \int_{\fra^*} 
	 \big(1+\mod{\la}\big)^{(n-1)/2-N} \wrt\la \\
& \leq C \bignorm{m}{\Horm{\TWp}{N}}\, [1-\Phi(a)] \, a^{\de(p)\rho}, 
\end{aligned}
$$
and $\ds \int_{\BA^-} \bigmod{\phi_p(a)} \dtt{a} C \bignorm{m}{\Horm{\TWp}{N}}$ is convergent, i.e.,  
$\phi_p\chi_{\BA^-}$ is in $\lu{\BA}$.  Since $\phi_p\chi_{\BA^+}$ is the difference of 
$\phi_p$ and $\phi_p\chi_{\BA^-}$, which are in $Cv_q(\BA)$ for all $q$ in $(1,\infty)$, so is $\phi_p\chi_{\BA^+}$.   
\end{proof}

\noindent
Also, it is convenient to defined the functions
$$
\begin{aligned}
\tau^{p,1} (vb)
	& := \chi_{\BA^+}(b) \, P(v)^{2/p} \, \big[\exp\big(E(v,b)H_0\big)-1\big] \, \phi_p\big([vb]_+\big) \\
\tau^{p,2} (vb)
	& := \chi_{\BA^+}(b) \, P(v)^{2/p} \, \big[\phi_p\big([vb]_+\big) - \phi_p(b) \big]\\
\tau^{p,3} (vb)
	& := \chi_{\BA^+}(b) \, P(v)^{2/p} \, \phi_p(b) \\
\cD(vb)
        & := b^{2\rho}
\end{aligned}
\quant v\in \OBN \quant b \in \BA.
$$
The following result is contained in \cite{I1}.  However, our proof differs from the original one
at some points, and the notation is also different.  This will be useful in the treatment of the reduced case.  

\begin{lemma} \label{l: one dim multipliers glob}
Suppose that $N> (n+3)/2$ and $1<p<2$.  Then there exists
a constant~$C$ such that for every $m$ in $\Horm{\TWp}{N}$ the following hold:
\begin{enumerate}
\item[\itemno1]
$\ds \int_{\BG} \bigmod{\ko(g)} \, \vp_{i\de(p)\rho}(g) \wrt g
\leq C \bignorm{m}{\Horm{T_p}{N}}$.  In particular, $\ko$ is in $\Cvp{\BX}$;
\item[\itemno2]
$\ku(a) = a^{-2\rho/p} \, \phi_p(a)$ and 
$\bigmod{\ku(vb)} \leq C \bignorm{m}{\Horm{T_p}{N}}\, b^{-2\rho/p}\, P(v)^{2/p}$
for all $a$ and $b$ in $\BA^+$;  
\item[\itemno3]
$\ds \int_{\OBN}\int_{\BA} \bigmod{[\chi_{\BA^-}\cD^{1/p}\ku](vb)} \wrt v  \dtt b  
\leq C \bignorm{m}{\Horm{T_p}{N}}$, i.e., $\chi_{\BA^-}\cD^{1/p}\ku$ is in $\lu{\OBN;\lu{\BA}}$;
\item[\itemno4]
we have the decomposition
$
\chi_{\BA^+}\cD^{1/p}\ku
= \tau^{p,1} + \tau^{p,2} + \tau^{p,3}.
$
The functions $\tau^{p,1}$ and $\tau^{p,2}$ are in $\lu{\OBN;\lu{\BA}}$, and $\tau^{p,3}$ 
is in $\lu{\OBN;\Cvp{\BA}}$, with norms controlled by $C \bignorm{m}{\Horm{T_p}{N}}$.  
Hence $\chi_{\BA^+}\cD^{1/p}\ku$ is in $\lu{\OBN;\Cvp{\BA}}$, and 
$
\bignorm{\chi_{\BA^+}\cD^{1/p}\ku}{\lu{\OBN;\Cvp{\BA}}}
\leq C \bignorm{m}{\Horm{T_p}{N}};
$ 
\item[\itemno5]
we have the estimate $\bignorm{(1-\Phi) \cH^{-1}m}{\Cvp{\BX}} \leq C \bignorm{m}{\Horm{T_p}{N}}$. 
\end{enumerate}
\end{lemma}

\begin{proof}
We sketch the proof of the lemma; this will be useful in the reduced case.  For all the details, see \cite{I1}.  
Notice that \rmi\ follows from the pointwise estimate
\begin{equation} \label{f: pointwise est kappaominfty}
\bigmod{\ko(a)} 
\leq C \bignorm{m}{\Horm{T_p}{N}} \frac{1-\Phi(a)}{(\log a)^N} \,\, a^{(\vep-2/p)\rho-2\al} 
\quant a \in \BA^+,
\end{equation}
which holds for each small $\vep>0$, and from standard estimates of the spherical function $\vp_{i\de(p)\rho}$.

To prove \rmii\ we start from the definition of $\ku$ (see \eqref{f: kernels rank one}) 
and move the contour of integration from $\fra^*$ to  $\fra^* +  i\de(p)\rho$.  To prove 
the estimate of $\ku$, simply write $\ku$ in Iwasawa co-ordinates, use \eqref{f: decom Iwas}
and the boundedness of $\phi_p$, which follows from the first estimate in Lemma~\ref{l: phip basic}.   

Part \rmiii\ follows from a clever argument involving the Abel transform.  
Recall that the Abel transform of a function $F$ on $\OBN\BA$ is defined by 
$\cA F(b) := b^\rho \,\int_{\OBN} \, F(vb) \wrt v$ for every $b$ in $\BA$.  Thus, 
$$
\int_{\OBN}\int_{\BA} \bigmod{[\chi_{\BA^-}\cD^{1/p}\ku](vb)} \wrt v  \dtt b  
= \int_{\BA^-}  b^{((2/p)-1)\rho}\,\cA \big(\mod{\ku}\big)(b)  \dtt b.
$$
Since the Abel transform of a $\BK$--bi-invariant function is a Weyl-invariant function on $\BA$,
we may write $\cA \big(\mod{\ku}\big)(b^{-1})$ instead of $\cA \big(\mod{\ku}\big)(b)$ 
in the last integral above, and then change variables ($b_1 = b^{-1}$):  we obtain 
$$
\int_{\OBN}\int_{\BA} \bigmod{[\chi_{\BA^-}\cD^{1/p}\ku](vb)} \wrt v  \dtt b  
= \int_{\BA^+}  b_1^{-\de(p)\rho}\,\cA \big(\mod{\ku}\big)(b_1) \dtt {b_1}.
$$
We deduce from \rmii\ that 
$\cA \big(\mod{\ku}\big)(b_1)
\leq C \bignorm{m}{\Horm{T_p}{N}}\, b_1^{-\de(p)\rho}\, P(v)^{2/p}$.  We insert this estimate in the
last integral above and obtain $\int_{\BA^+}  b_1^{-2\de(p)\rho}\,\dtt {b_1}$.  
Since $1<p<2$, and $-\de(p) = {1/p'-1/p}$, the last integral is convergent, and the desired estimate follows.  

The decomposition of $\chi_{\BA^+}\cD^{1/p}\ku$ in \rmiv\ is a trivial consequence of 
its definition and of formula \eqref{f: decom Iwas} that relates the Cartan and the Iwasawa decompositions 
of elements of the form $vb$, with $v$ in $\OBN$ and $b$ in $\BA$.

To prove that $\tau^{p,1}$ is in $\lu{\OBN;\lu{\BA}}$ it suffices to observe that $\phi_p$ is
bounded by Lemma~\ref{l: phip basic}, $P^{2/p}$ is in $\lu{\OBN}$, and that $\bigmod{\exp\big(E(v,b)H_0\big)-1} \leq C\,b^{-2}$
for all $b$ in $\BA^+$.  

Next we show that $\tau^{p,2}$ is in $\lu{\OBN;\lu{\BA}}$.  We need to estimate 
$$
\begin{aligned}
\int_{\BA} \int_{\OBN} \, \bigmod{\tau^{p,2}(vb)} \wrt v \dtt b
& \leq \int_{\BA^+} \int_{\OBN} \, P(v)^{2/p} \, \bigmod{\phi_p([vb]_+)-\phi_p(b)} \wrt v \dtt b \\ 
& \leq \int_{\BA^+}\dtt b \int_{\OBN} \wrt v \, P(v)^{2/p} \, \int_b^{[vb]_+} 
	\bigmod{a\partial_a \phi_p(a)}\dtt a. 
\end{aligned}
$$
We use the estimate for $a\partial_a \phi_p$ in Lemma~\ref{l: phip basic} to obtain that the inner integral above is
bounded by 
$$
\begin{aligned}
C \bignorm{m}{\Horm{\TWp}{N}}\, \int_b^{[vb]_+} \frac{1-\Phi(a)}{\log^2 a} \dtt a 
& =    C \bignorm{m}{\Horm{\TWp}{N}}\, \int_b^{[vb]_+} \frac{1-\Phi(a)}{\log^2a} \dtt a \\ 
& \leq C \bignorm{m}{\Horm{\TWp}{N}}\, \int_b^{[vb]_+} \frac{1}{1+\log^2 a} \dtt a \\ 
& \leq C \bignorm{m}{\Horm{\TWp}{N}}\, \frac{\log [vb]_+-\log b}{1+\log^2 b}. 
\end{aligned}
$$
We observe that, by \eqref{f: decom Iwas}, $\log [vb]_+-\log b \leq \al\big(H(v)\big)+2$ for $b$ in~$\BA^+$.  
By combining the estimates above we see that 
$$
\begin{aligned}
\int_{\BA} \int_{\OBN} \, \bigmod{\tau^{p,2}(vb)} \wrt v \dtt b
& \leq C \bignorm{m}{\Horm{\TWp}{N}}\, \int_{\BA^+} \frac{1}{1+\log^2 b}\dtt b\,
	 \int_{\OBN} \, P(v)^{2/p} \, \big[\al\big(H(v)\big)+2\big]\wrt v; 
\end{aligned}
$$
the integrals on the right hand side are convergent, and the required estimate of $\tau^{p,2}$ follows.  

Finally, we may use the fact $\chi_{\BA^+} \phi_p$ is in $\Cvp{\BA}$ (see Lemma~\ref{l: phip basic}) and
apply directly Theorem~\ref{t: Transference principle}, 
and show that $\tau^{p,3}$ is in $\lu{\OBN;\Cvp{\BA}}$, thereby concluding the proof of \rmiv.  

As for the proof of \rmv, observe that 
$\chi_{\BA^-}\cD^{1/p}\ku$ is in $\lu{\OBN;\lu{\BA}}$ by \rmiii, and 
$\chi_{\BA^+}\cD^{1/p}\ku$ is in $\lu{\OBN;\Cvp{\BA}}$ by \rmiv, whence 
$\cD^{1/p}\ku$ is in $\lu{\OBN;\Cvp{\BA}}$.  
Theorem~\ref{t: Transference principle} then implies that $\ku$ is in $\Cvp{\OBN\BA}$.
This, in turn, implies that $\ku$ is in $\Cvp{\BX}$, for $\ku$ 
is a $\BK$-invariant function on $\BX$.  It is not hard to check that the steps above yield
$\bignorm{\ku}{\Cvp{\BX}} \leq C \bignorm{m}{\Horm{T_p}{N}}$. This and \rmi\ then ensure that 
$(1-\Phi) \cH^{-1}m$ is in $\Cvp{\BX}$ with the required control of the norm.  
\end{proof}

\subsection{The case where $\BX$ is reducible}  \label{s: BX reducible}
In this paper we are interested in the particular case where $\BG = \BG_1\times \BG_2$, where
$\BG_1$ and $\BG_2$ are noncompact semisimple Lie groups with finite
centre and real rank one.  Denote by $\BK_1$ and $\BK_2$ maximal compact subgroups of $\BG_1$
and $\BG_2$, respectively, and consider the associated Riemannian symmetric
spaces of the noncompact type $\BX_1 = \BG_1/\BK_1$ and $\BX_2 = \BG_2/\BK_2$ with
dimensions $n_1$ and $n_2$, respectively.
Then $\BX_1\times \BX_2$ is a (reducible) symmetric space of noncompact type,
which hereafter we shall denote by $\BX$; $\BX$ may also be regarded as
the symmetric space $\BG/\BK$, where $\BG = \BG_1 \times \BG_2$ and
$\BK = \BK_1\times \BK_2$. 

Recall the Iwasawa decompositions $\BG_1 = \BN_1 \BA_1 \BK_1$ and $\BG_2 = \BN_2 \BA_2 \BK_2$; 
an Iwasawa decomposition of $\BG$ is then 
$\BG= \BN\BA\BK$, where $\BN = \BN_1\times\BN_2$, $\BA = \BA_1\times\BA_2$.
Note that $\BA$ is two dimensional.  We denote by $\cD_1$ and $\cD_2$ the Radon--Nikodym
derivatives of the actions of $\BA_1$ on $\OBN_1$ and of $\BA_2$ on $\OBN_2$, respectively.  
Explicitly, $\cD_1(b_1) = b_1^{2\rho_1}$ and $\cD_1(b_2) = b_2^{2\rho_2}$.  Sometimes we abuse the notation 
and denote still by $\cD_1$ and $\cD_2$ the natural extensions of $\cD_1$ and $\cD_2$ to 
$\OBN_1\BA_1$ and $\OBN_2\BA_2$, respectively.  
The Lie algebra $\fra$ of $\BA$ is two-dimensional, and the root
system is of type $A_1\times A_1$.  Thus, $\fra^+$ may be identified
with the first quadrant of $\BR^2$ via the identification $(\la_1,\la_2) \equiv \big(\la_1(H_0^{(1)}), \la_2(H_0^{(2)})\big)$,
where $H_0^{(1)}$ and $H_0^{(2)}$ are the analogues for $\BX_1$ and $\BX_2$ of the vector $H_0$ defined in the rank one case.  If $a$ is in $\BA$ and $\la$ belongs to $\fra^*$, we 
	write $a^\la$ instead of $\e^{\la_1(\log a_1)}\e^{\la_2(\log a_2)}$.
Set $\BA^+ := \exp{(\fra^+)}$, and notice that $\BA^+ = \BA_1^+ \times \BA_2^+$. 

The Weyl group $W$ of $(\BG,\BK)$ is then generated by the reflections
with respect to the co-ordinate axes in $\BR^2$.
The interior $\bW$ of the convex hull of the points $\{w\cdot \rho: w \in W\}$
is an open rectangle in $\fra^*$, with sides parallel
to the axes.  The vertex in $\fra^+$ of this rectangle is the vector $\rho$;
its co-ordinates are $(\rho_1,\rho_2)$, where $\rho_1$ and $\rho_2$ are the half-sums of
the positive roots with multiplicities of the pairs $(\frg_1,\fra_1)$
and $(\frg_2,\fra_2)$, respectively.  The tube $\TWp$ is then the product
$\TWp^{(1)}\times \TWp^{(2)}$, where $\TWp^{(1)} = \fra_1^* + i \bW_p^{(1)}$
and $\bW_p^{(1)}$ denotes the convex hull of the orbits of $\rho_1$ under the Weyl
group $W_1$; $\TWp^{(2)}$ is defined similarly.

Bounded spherical functions are then indexed by elements of $\TWu$.  For $\la$
in $\TWu$,
$
\vp_{(\la_1,\la_2)}
= \vp^1_{\la_1}\otimes \vp^2_{\la_2},
$
where $\vp_{\la_1}^1$ and $\vp_{\la_2}^2$ are spherical functions on $\BG_1$ and
$\BG_2$, respectively.
Also the Harish-Chandra function and the Plancherel measure on $\BX$ are given by
$\bc = \bc_1\otimes \bc_2$ and $\nu = \nu_1\otimes \nu_2$.
Note that the density of the measure in Cartan co-ordinates is just $\de = \de_1\otimes \de_2$,
where $\de_1$ and $\de_2$ denote the corresponding densities on $\BX_1$ and $\BX_2$, respectively.  

\subsection{The Fig\`a-Talamanca--Herz space}
Denote by $\Ga$ a noncompact abelian Lie group and suppose that $p$ is in $(1,\infty)$.  The space $A_p(\Ga)$
is the space of all continuous functions on $\Ga$ vanishing at infinity that admit a representation of the form
\begin{equation} \label{f: def FTH} 
f
= \sum_j f_j*g_j,
\end{equation} 
where $\{h_j\}$ and $\{g_j\}$ are sequences of functions in $\lp{\Ga}$ and $L^{p'}(\Ga)$, respectively,
such that 
\begin{equation} \label{f: Ap norm}
\sum_j \bignorm{h_j}{\lp{\Ga}} \, \bignorm{g_j}{L^{p'}(\Ga)} < \infty.
\end{equation} 
The norm of $f$ is then the infimum of \eqref{f: Ap norm} over all representation of $f$ of the form \eqref{f: def FTH}.  
The information we shall need on $A_p(\Ga)$ is contained in the following result.

\begin{theorem} \label{t: Ap spaces}
Suppose that $\Ga$ is a noncompact abelian Lie group and that $1<p<\infty$.  The following hold:
\begin{enumerate}
\item[\itemno1]
the Banach dual of $A_p(\Ga)$ is $\Cvp{\Ga}$;
\item[\itemno2]
$A_p(\Ga)$ is an algebra under pointwise multiplication, containing $C_c^\infty(\Ga)$;
\item[\itemno3]
elements of $A_p(\Ga)$ are pointwise multipliers of $\Cvp{\Ga}$, i.e. $\phi$ in $A_p(\Ga)$ and 
$k$ in $\Cvp{\Ga}$ imply that $\phi\, k$ is in $\Cvp{\Ga}$;
\item[\itemno4]
if $\Ga = \Ga_1\times\Ga_2$, where $\Ga_1$ and $\Ga_2$ are noncompact abelian groups, and $\phi_1$ is in $C_c^\infty(\Ga_1)$,
then $\phi_1\otimes \One$ and $(1-\phi_1)\otimes \One$ are in $A_p(\Ga_1\times\Ga_2)$.
\end{enumerate}
\end{theorem}

\noindent
Part \rmi\ is a well known result of Fig\`a-Talamanca in the abelian case; \rmii\ was proved by Herz.
We refer the reader to \cite{Co} for the proof of \rmiv\ and for references concerning \rmi-\rmiii.

\section{Boundedness results for convolution operators}
\label{s: Boundedness results for convolution operators}

In this section we group together various boundedness results for convolution operators
which will be used in the proof of our main result.
Let $\Ga$ be a locally compact group with left Haar measure $\wrt y$, and modular function $\Delta_\Ga$.  
Convolution of $f$ and $k$ on $\Ga$ is defined by
$$
f*k(x)
= \int_\Ga f(xy) \, k(y^{-1}) \wrt y
= \int_\Ga f(y) \, k(y^{-1}x) \wrt y.
$$
Throughout the paper $\Cvp{\Ga}$ will denote the space of bounded operators on $\lp{\Ga}$
that commute with \emph{left} translations (equivalently the space of \emph{right} convolutors), 
equipped with the operator norm on $\lp{\Ga}$.
We recall the following basic convolution inequality \cite[Corollary 20.14 (ii) and (iv)]{HR} 
\begin{equation} \label{f: second convolution}
\bignorm{\kappa}{\Cvp{\Ga}}
\leq \bignorm{\Delta_\Ga^{-1/p'}\kappa}{\lu{\Ga}}.
\end{equation}
Notice that \eqref{f: second convolution} may be rewritten as 
\begin{equation*} 
\bignorm{\kappa}{\Cvp{\Ga}}
\leq \bignorm{\Delta_\Ga^{1/p}\kappa}{\lu{\Ga,\rho}},
\end{equation*}
where $\rho$ denotes the right Haar measure on $\Ga$ corresponding to the chosen left Haar measure.  
Consider now two locally compact groups $\Ga_1$ and $\Ga_2$, with modular functions $\Delta_{\Ga_1}$
and $\Delta_{\Ga_2}$, respectively.  Set $\Ga := \Ga_1\times \Ga_2$.
We shall repeatedly use the following simple product variant of \eqref{f: second convolution}.

\begin{lemma} \label{l: iterated conv}
The following hold:
\begin{enumerate}
\item[\itemno1]
$\bignorm{\kappa}{\Cvp{\Ga}} 
\leq \bignorm{\kappa}{\Cvp{\Ga_1;\Cvp{\Ga_2}}}
\leq \bignorm{(\Delta_{\Ga_1}^{-1/p'}\otimes 1)\kappa}{\lu{\Ga_1;\Cvp{\Ga_2}}}$.
\end{enumerate}

\noindent
Suppose further that 
$\Ga_1$ and $\Ga_2$ are noncompact semisimple Lie groups, $K_1$ and $K_2$ are maximal compact subgroups
of $\Ga_1$ and $\Ga_2$, and set $K:= K_1\times K_2$.  The following hold:  
\begin{enumerate}
\item[\itemno2]
if $g_1\mapsto \bignorm{\kappa(g_1,\cdot)}{\Cvp{\Ga_2}}$ is $K_1$--bi-invariant, 
then 
$$
\bignorm{\kappa}{\Cvp{\Ga}}
\leq C \, \int_{\Ga_1} \bignorm{\kappa(g_1,\cdot)}{\Cvp{\Ga_2}} \, \vp_{i\de(p) \rho_1}(g_1) \wrt g_1;
$$
\item[\itemno3]
if $\kappa$ is $K$--bi-invariant, 
then 
$$
\bignorm{\kappa}{\Cvp{\Ga/K}}
\leq C \, \int_{\Ga_1} \bignorm{\kappa(g_1,\cdot)}{\Cvp{\Ga_2/K_2}} \, \vp_{i\de(p) \rho_1}(g_1) \wrt g_1. 
$$
\end{enumerate}
\end{lemma}

\begin{proof} 
First we prove \rmi. For a function $f$ on $\Ga$ and $h_1$ in $\Ga_1$ 
we let $f_{h_1}(h_2)=f(h_1,h_2)$ for every $h_2$ in $\Ga_2$. Fubini's theorem and Minkowski's integral inequality imply that
$$
\begin{aligned}
\bignorm{f*\kappa}{\lp{\Ga}}
& =   \Big[\int_{\Ga_1} \wrt g_1\Bignormto{\int_{\Ga_1} [f_{g_1 h_1}*_{\Ga_2}\kappa_{h_1^{-1}}]
         \wrt h_1}{\lp{\Ga_2}}{p}\Big]^{1/p}  \\
& \leq\Big[\int_{\Ga_1} \wrt g_1 \Big(\int_{\Ga_1} 
         \bignorm{f_{g_1 h_1}*_{\Ga_2}\kappa_{h_1^{-1}}}{\lp{\Ga_2}} \wrt h_1\Big)^p\Big]^{1/p}  \\
& \leq\Big[\int_{\Ga_1} \wrt g_1 \Big(\int_{\Ga_1} 
         \bignorm{f_{g_1 h_1}}{\lp{\Ga_2}} \bignorm{\kappa_{h_1^{-1}}}{\Cvp{\Ga_2}} 
         \wrt h_1\Big)^p\Big]^{1/p}.
\end{aligned}
$$
The inner integral is just the convolution on $\Ga_1$ between 
$g_1 \mapsto \bignorm{f_{g_1}}{\lp{\Ga_2}}$ and $g_1 \mapsto \bignorm{\kappa_{g_1}}{\Cvp{\Ga_2}}$, 
and the first of the two inequality follows.
The second inequality follows from the first and~\eqref{f: second convolution}.     

The proof of \rmii\ is very similar to that of \rmi, the only difference being that we
use Herz's \textit{principe de majoration} instead of \eqref{f: second convolution} to bound 
the convolution on $\Ga_1$ between $g_1 \mapsto \bignorm{f_{g_1}}{\lp{\Ga_2}}$ 
and $g_1 \mapsto \bignorm{\kappa_{g_1}}{\Cvp{\Ga_2}}$.   

Finally, \rmiii\ follows from \rmii\ and the fact that the convolution
of a $K$--right-invariant function and a $K$--bi-invariant function is a $K$--right-invariant function.    
\end{proof}

\begin{remark}
Notice that if $\Ga_1$ is amenable, then  
$\bignorm{\kappa}{\Cvp{\Ga_1;\Cvp{\Ga_2}}}
= \bignorm{(\Delta_{\Ga_1}^{-1/p'}\otimes 1)\kappa}{\lu{\Ga_1;\Cvp{\Ga_2}}}$, so we do not loose anything 
in the second inequality in \rmi\ above.  
\end{remark}

\subsection{Transference principle for semidirect products} \label{s: Transference semidirect}

In this subsection we consider a group $\Ga$, which is the semidirect product of two
unimodular subgroups $N$ and~$H$, with $N\cap H = \{e\}$;
here $N$ is normal in $\Ga$, and $H$ acts on $N$ by conjugation.  Denote by $\wrt n$ and
$\wrt h$ two Haar measures on $N$ and $H$, respectively.  Then $\wrt n \wrt h$
is a right Haar measure on $\Ga$.
For each $h$ in $H$ and $n$ in $N$, denote by $n^h$ the conjugate $hnh^{-1}$ of $n$ by~$h$.
Denote by $\cD(h)$ the Radon--Nikodym derivative $\wrt(n^h)/\wrt n$.
It is straightforward to check that $\cD(h) \wrt n\wrt h$ is a left Haar measure on $\Ga$,
which we denote by $\wrt g$.  The following integral formulae hold
$$
\int_\Ga f(g) \wrt g
= \int_N \int_H f(nh) \, \cD(h) \wrt n \wrt h
= \int_H \int_N f(hn) \wrt n \wrt h.
$$
Note that the assumption $N\cap H = \{e\}$ implies that the extension of $\cD$ to $NH$
given by $\cD(nh) := \cD(h)$ is well defined.   
Consistently with our previous notation, we denote by $\Cvp{\Ga}$ and $\Cvp{H}$ 
the space of \emph{right} convolutors on $\Ga$ and on $H$, respectively.  

Suppose that $S$ is a ``nice'' function on $\Ga$.  For each $n$ in $N$ we denote by 
$S(n\cdot)$ the restriction of~$S$ to $nH$, i.e.  $S(n \cdot)(h) := S(nh)$ for all $h$ in $H$.
We define $S(\cdot h)$ similarly, but with the role of $n$ and~$h$ interchanged.
These definitions extend in a natural way to the case where $S$ is a distribution on $\Ga$.
Suppose that $S$ is a distribution on $\Ga$ and assume that for (almost) each $n$ in $N$
the distribution $S(n \cdot)$ (as a distribution on $H$) is in $\Cvp{H}$, and that
the function $n\mapsto \bignorm{S(n \cdot)}{\Cvp{H}}$ is in $\laq{N}$.  Then we say
that $S$ belong to the space $\laq{N;\Cvp{H}}$, and set
\begin{equation*} 
\bignorm{S}{\laq{N;\Cvp{H}}}
:=  \Big[\int_{N} \, \bignormto{S(n \cdot)}{\Cvp{H}}{q} \wrt n\Big]^{1/q}.
\end{equation*}
The following is a special case of a more general transference principle \cite[Corollary~3.4]{CMW}. 

\begin{theorem} \label{t: Transference principle}
Suppose that $\Ga$, $N$ and $H$ are as above and that $p$ is in $(1,\infty)$.  Then
$$
\bignorm{\kappa}{\Cvp{\Ga}}
\leq \bignorm{\cD^{1/p}\kappa}{\lu{N;\Cvp{H}}}.
$$
In particular, if $\big(\cD^{1/p}\kappa\big) (nh) := Q(n) \, \phi(h)$, where $\phi$ is in $\Cvp{H}$ and $Q$ is in $\lu{N}$, then
$$
\bignorm{\kappa}{\lp{\Ga}}
\leq \bignorm{Q}{\lu{N}} \bignorm{\phi}{\Cvp{H}}.
$$
\end{theorem}

\noindent
We deduce an important consequence of Theorem~\ref{t: Transference principle},
which we shall use to obtain bounds for the operator $B_1$ (see the 
decomposition \eqref{f: splitting of kB} in the proof of Theorem~\ref{t: main}).
We consider the direct product $\BG_1\times \BG_2$ of two rank one semisimple Lie
groups with finite centre, and a $\BK$--bi-invariant kernel $\kappa$ on $\BG_1 \times \BG_2$.
We consider an Iwasawa decomposition $\OBN_1\BA_1\BK_1$ of $\BG_1$,
and view the kernel $\kappa$ as a function on the group $\OBN_1\BA_1 \times \BG_2$.
We consider the right convolution operator on $\OBN_1\BA_1 \times \BG_2$, acting on 
functions in $\lp{\OBN_1\BA_1 \times \BG_2}$.  It is straightforward to check that if $f$ is 
such a function, then $f*\kappa = [\Pi_{\BK_2}f]*\kappa$, 
where $[\Pi_{\BK_2}f](v_1b_1,g_2\cdot o_2) := \int_{\BK_2} f(v_1b_1, g_2k_2) \wrt k_2$ is
the projection of $f$ on functions in $\lp{\OBN_1\BA_1 \times \BG_2/\BK_2}$.  

\begin{corollary} \label{c: Transference principle}
Suppose that $p$ is in $(1,\infty)$ and that $\kappa$ is a $\BK$--bi-invariant function
on $\BG_1\times \BG_2$.  For each $v_1$ in $\OBN_1$, set $\kappa_{v_1}(a_1, a_2) := \kappa(v_1 a_1, a_2)$.
The following hold:
\begin{enumerate}
\item[\itemno1]
$
\ds \bignorm{\kappa}{\Cvp{\OBN_1\BA_1\times \BG_2}}
\leq \int_{\OBN_1} \bignorm{[\cD_1^{1/p}\otimes 1]\, \kappa_{v_1}}{\Cvp{\BA_1 \times \BG_2}} \wrt v_1;
$
\item[\itemno2]
$\ds
\bignorm{[\cD_1^{1/p}\otimes 1]\, \kappa_{v_1}}{\Cvp{\BA_1 \times \BG_2}}
\leq C \bignorm{[\cD_1^{1/p}\otimes \de_2]\, \kappa_{v_1}}{\Cvp{\BA_1 \times \BA_2}}. 
$
\end{enumerate}
\end{corollary}

\begin{proof}
First we prove \rmi.  
It is straightforward to check that $\OBN_1 \times\{e_2\}$ is normal in $\OBN_1\BA_1 \times \BG_2$,
and that $\OBN_1\BA_1\times \BG_2$ may be viewed as the semidirect product of
$\OBN_1 \times\{e_2\}$ and $\BA_1 \times \BG_2$, with $\BA_1 \times \BG_2$ acting on $\OBN_1 \times\{e_2\}$
by conjugation. 
Hence \rmi\ follows from
Theorem~\ref{t: Transference principle} (with $\OBN_1 \times\{e_2\}$
in place of $N$ and $\BA_1 \times \BG_2$ in place of $H$).

Next we prove \rmii.  
Recall that the transference result of Coifman and Weiss
\cite[Theorem~8.7]{CW} holds for all unimodular groups that admit a ``Cartan decomposition''.
We apply this result to the group $\BA_1 \times \BG_2$, which admits the Cartan decomposition
$\big(\{e_1\}\times \BK_2\big) \cdot \big(\BA_1\times \BA_2\big) \cdot
\big(\{e_1\}\times \BK_2\big)$, and to the $\{e_1\}\times\BK_2$--bi-invariant extension
of $\kappa(v_1\cdot,\cdot)$ (thought of as a function on $\BA_1\times \BA_2$) 
to $\BA_1\times \BG_2$.  The required estimate follows.
\end{proof}

\section{Statement of the main result} \label{s: Statement}

Hereafter $\BX = \BX_1\times \BX_2$, where $\BX_1$ and $\BX_2$ are symmetric spaces of the noncompact
type and real rank one.  Our notation is consistent to that of
Subsection~\ref{s: BX reducible} above.
If $B$ is a $\BG$--invariant operator on $\BX$, bounded on $\ld{\BX}$, then $Bf = f \ast k_B$ 
for a suitable $\BK$--bi-invariant distribution $k_B$. 
We denote by $m_B$ the spherical Fourier transform of $k_B$.
If~$B$ extends to a bounded operator on~$\lp{\BX}$, then we say that
$k_B$ is a convolutor of $\lp{\BX}$.  The space $\Cvp{\BX}$ of all convolutors of $\lp{\BX}$
is a Banach space with respect to the operator norm
$
\bignorm{k_B}{\Cvp{\BX}}
:= \bigopnorm{B}{\lp{\BX}}.
$
We need to consider four classes of multipliers, satisfying various Marcinkiewicz-Mihlin (MM in the sequel)
type conditions.  It is convenient to introduce the following notation: $n_1$ and $n_2$
denote the dimensions of $\BX_1$ and $\BX_2$, and we set $n := (n_1,n_2)$.
Given two multi-indices $N = (N_1,N_2)$ and $J = (J_1,J_2)$, the relation $J\leq N$ means that
$J_1\leq N_1$ and $J_2\leq N_2$.  We denote by $\Theta_p^{(1)}$ and $\Theta_p^{(2)}$ the analogue 
on $\BX_1$ and $\BX_2$ of the~$\Theta_p$ function defined in \eqref{f: Theta} in the rank one case.  

\begin{definition} \label{def: Marc}
Suppose that $N$ is a multi-index and that $p$ is in $(1,\infty) \setminus\{2\}$.
We define $\Mar{\TWp}{N}$ [resp. $\Marinfty{\TWp}{N}$] to be the space of all 
bounded Weyl-invariant holomorphic functions $m$
on $\TWp$ such that 
\begin{equation} \label{f: Marc}
\bignorm{m}{\Mar{\TWp}{N}}
:= \sup_{(\la_1,\la_2) \in \TWp} \,  \Theta_p^{(1)}(\la_1)^{J_1} \, \Theta_p^{(2)}(\la_2)^{J_2}
\, \bigmod{\partial^J m(\la_1,\la_2)}
<\infty  
\end{equation}
[resp. 
\begin{equation} \label{f: Marc infty}      
\bignorm{m}{\Mar{\TWp}{N}}   
:= \sup_{(\la_1,\la_2) \in \TWp} \, \big[1+\Theta_p^{(1)}(\la_1)\big]^{J_1} \, \big[1+\Theta_p^{(2)}(\la_2)\big]^{J_2}
\, \bigmod{\partial^J m(\la_1,\la_2)}
< \infty ]
\end{equation}
for all multi-indices $J=(J_1,J_2) \leq N$.
When \eqref{f: Marc} [resp. \eqref{f: Marc infty}] holds we say that $m$ satisfies a 
\emph{MM condition of order $N$ on the tube $\TWp$ [resp. at infinity
on the tube $\TWp$]}. 
\end{definition}

\begin{remark} \label{rem: comparison MW Ion}
Ionescu \cite[formula~(4.1)]{I3} considered on a higher rank symmetric space an interesting condition, which,
in the reduced case, may be written as follows 
\begin{equation} \label{f: Ion}
\bignorm{m}{\Mar{\TWp}{N}'}
:= \sup_{(\la_1,\la_2) \in \TWp} \,  d_p(\la_1,\la_2)^{J_1+J_2} \, \bigmod{\partial^J m(\la_1,\la_2)}
<\infty,
\end{equation}
where $d_p(\la_1,\la_2)^2 := (\Re \la_1)^2+(\Re \la_2)^2 + \dist[(\Im \la_1,\Im\la_2), \bW_p^c]^2$,
and $\dist$ denotes the Euclidean distance in $\BR^2$.  
Notice that if $(\la_1,\la_2)$ is in $\TWp$ and it is away from $0+i\bW_p$, then a function $m$ for which 
$\bignorm{m}{\Mar{\TWp}{N}}$ is finite satisfies the estimates
\begin{equation} \label{f: MW infty}
\bigmod{\partial^J m(\la_1,\la_2)}
\leq C \, \bigmod{\Re \la_1}^{-J_1} \, \bigmod{\Re \la_2}^{-J_2}
\end{equation}
whereas a function $m$ for which 
$\bignorm{m}{\Mar{\TWp}{N}'}$ is finite satisfies 
\begin{equation} \label{f: I infty}
\bigmod{\partial^J m(\la_1,\la_2)}
\leq C \, \big[\bigmod{\Re \la_1}^2+ \, \bigmod{\Re \la_2}^2\big]^{-(J_1+J_2)/2}. 
\end{equation}
Clearly, if $m$ satisfies \eqref{f: I infty}, then it satisfies \eqref{f: MW infty}.  

Now we compare conditions \eqref{f: Marc} and \eqref{f: Ion} when $(\la_1,\la_2)$ is in $\TWp$ and it is close to $0+i\bW_p$.  
Assume first that $(\la_1,\la_2)$ belongs to $0+i\bW_p$, i.e., $\Re\la_1=0=\Re\la_2$, that $(\Im \la_1, \Im \la_2)$ is in 
$(\fra^*)^+$, and that 
$\min\big[\mod{\Im\la_1 - \de(p) \rho_1}, \mod{\Im\la_2 - \de(p) \rho_2}\big] = \mod{\Im\la_1 - \de(p) \rho_1}$.  
Notice that \eqref{f: Marc} is then equivalent to the condition 
\begin{equation*} 
\bigmod{\partial^J m(\la_1,\la_2)}
\leq C \, \mod{\Im\la_1 - \de(p) \rho_1}^{-J_1}
\end{equation*}
and \eqref{f: Ion} is equivalent to 
\begin{equation*} 
\bigmod{\partial^J m(\la_1,\la_2)}
\leq C \, \mod{\Im\la_1 - \de(p) \rho_1}^{-J_1-J_2}.
\end{equation*}
Thus, in this case if $m$ satisfies \eqref{f: Marc}, then it satisfies \eqref{f: Ion}, but not conversely.  

Next assume that $(\Im \la_1, \Im \la_2)$ is in $(\fra^*)^+$, that $\Re \la_1 = 0$, 
and that $\Re \la_2 = \mod{\Im\la_1 - \de(p) \rho_1}^{1/4}= \mod{\Im\la_2 - \de(p) \rho_2}^{1/2}$ is small.  Then 
$$
\begin{aligned}
\Theta_p^{(1)}(\la_1)^{J_1} \, \Theta_p^{(2)}(\la_2)^{J_2}
& =      \mod{\Im\la_1 - \de(p) \rho_1}^{J_1} \,\,  
	 \big[\mod{\Im\la_1 - \de(p) \rho_1} +  \mod{\Im\la_1 - \de(p) \rho_1}^2\big]^{J_2/2} \\ 
& \asymp \mod{\Im\la_1 - \de(p) \rho_1}^{J_1+J_2/4},
\end{aligned}
$$
and  
$$
\begin{aligned}
d_p(\la_1,\la_2)^{J_1+J_2} 
& =      \big[\mod{\Im\la_1 - \de(p) \rho_1}^{1/2} +  \mod{\Im\la_1 - \de(p) \rho_1}^2\big]^{(J_1+J_2)/2} \\ 
& \asymp \mod{\Im\la_1 - \de(p) \rho_1}^{(J_1+J_2)/4}.
\end{aligned}
$$
Therefore , in this case if $m$ satisfies \eqref{f: Ion}, then it satisfies \eqref{f: Marc}, but not conversely.  

The last two observations allow us to conclude that for $(\la_1,\la_2)\in \TWp$ close to $0+i\bW_p$   
the two conditions \eqref{f: Marc} and \eqref{f: Ion} are independent.   
\end{remark}

\begin{definition}
We define $\Mar{\fra^*}{N}$ [resp. $\Marinfty{\fra^*}{N}$] to be the space of all bounded   
functions~$m$ on $\fra^*$ for which 
\begin{equation} \label{f: Marc I}
\bignorm{m}{\Mar{\fra^*}{N}}
:= \sup_{(\la_1,\la_2) \in \fra^*} \, \bigmod{\la_1}^{j_1} \, \bigmod{\la_2}^{j_2}
\, \bigmod{\partial^J m(\la_1,\la_2)}
< \infty
\end{equation}
[resp. 
\begin{equation} \label{f: Marc infty I}
\bignorm{m}{\Marinfty{\fra^*}{N}} 
:= \sup_{(\la_1,\la_2) \in \fra^*} \, \big[1+\bigmod{\la_1}\big]^{j_1} \, \big[1+\bigmod{\la_2}\big]^{j_2}
\, \bigmod{\partial^J m(\la_1,\la_2)}] 
\end{equation}
for all multi-indices $J:=(J_1,J_2) \leq N$.
We call \eqref{f: Marc I} and \eqref{f: Marc infty I} \emph{MM condition} and
\emph{MM condition at infinity}, respectively.
\end{definition}

\begin{remark} \label{rem: extension}
Suppose that $m$ is in $\Mar{\TWp}{N}$.  Then $m$ extends to a function, still denoted by $m$, in 
$C^N(\OV{T_p}\setminus i\partial \bW_p)$, and the boundary values $m(\cdot+i\zeta)$,
where $\zeta \in \partial \bW_p$, satisfy a Marcinkiewicz condition $\Mar{\fra^*}{N}$. 
The proof of this fact follows the lines of the proof of Proposition~\ref{p: extension}.  We omit the details.
\end{remark}

\noindent
We recall the following corollary of the classical Marcinkiewicz multiplier theorem.

\begin{theorem} \label{t: classical Marc}
Suppose that $m$ is 
in $\Mar{\fra^*}{N}$, with $N = (N_1,N_2) \geq (1,1)$.
Then $m$ is an $\lp{\fra_1^*\times \fra_2^*}$ Fourier multiplier for all $p$ in $(1,\infty)$.
\end{theorem}

\noindent
Our main result is Theorem~\ref{t: main}, which we restate for the reader's convenience.  
We denote by $\tre$ the two dimensional vector $(3,3)$.  Recall that $n = (n_1,n_2)$.  

\begin{theorem*} 
Suppose that $N>(n+\tre)/2$ 
and that $p$ is in $(1,\infty) \setminus\{2\}$.  Then there exists a constant~$C$ such that
$
\bigopnorm{B}{\lp{\BX}}
\leq C \, \bignorm{m_B}{\Mar{\TWp}{N}}
$
for every $\BG$-invariant operator $B$.
\end{theorem*}
\begin{remark} \label{rem: comparison}
Let $\cL _j$ be the Laplace--Beltrami operator on $\BX_j,$ $j=1,2$. 
Using \cite[Corollary 3.2]{WroJSMMSO} one may prove a Marcinkiewicz-type multiplier theorem on $L^p(\BX)$ 
for joint spectral multipliers $m(\cL)$ of the pair $\cL=(\cL_1,\cL_2).$ 
This result however is significantly weaker then Theorem \ref{t: main}. 
Indeed, in \cite[Corollary 3.2]{WroJSMMSO} one requires $m$ to be defined and bi-holomorphic 
on certain products of sectors (depending on $p$) in the right half-plane. On the other hand it is 
well known that the $L^p(\BX_j)$ spectrum of $\cL_j$ (for $p\neq 2$) is an ellipsoid 
in the right half plane (the $p$-ellipsoid being contained in the $p$-sector). 
Moreover, in order to conclude the $L^p(\BX)$ boundedness of $m(\cL)$ by using Theorem~\ref{t: main}, 
we would only require that $m$ is bi-holomorphic on the product of the two $p$-ellipsoids, which 
is a significantly smaller set than the product of the two $p$-sectors. 
\end{remark}

\noindent
\textit{Outline of the proof of Theorem~\ref{t: main}.}
Since $\bignorm{k_B}{\Cvp{\BX}} = \bignorm{k_B}{Cv_{p'}(\BX)}$ for all $p$ in $(1,\infty)$,
we may assume that $p$ is in $(1,2)$ in the rest of the proof.   
The strategy of the proof is to obtain estimates of the kernel $k_B$ of $B$ via
the inverse spherical Fourier transform.
Since $k_B$ is $\BK$--bi-invariant,
it suffices to estimate the restriction of $k_B$ to $\BA^+ = \BA_1^+\times \BA_2^+$.
It is common practice to split up the analysis of $k_B$ into three parts: (i) the analysis near the origin;
(ii) the analysis near the walls of the Weyl chamber, but away from the origin;
(iii) the analysis at infinity, but away from the walls.

Recall that $\bC$ was defined just after Proposition~\ref{p: extension}.   
We denote by $\Phi_1$ the $\BK_1$--bi-invariant function on $\BG_1$ defined by 
$
\Phi_1\big(a_1\big)
:= \bC \big(\al_1(\log a_1)\big)
$
for every $a_1$ in $\BA_1$, and define $\Phi_2$ similarly on $\BG_2$.  
We define the $\BK$--bi-invariant
functions $k_{B_0}$, $k_{B_1}$ and $k_{B_2}$ on $\BG$ by
\begin{equation} \label{f: splitting of kB}
\begin{aligned}
k_{B_0}
& = k_B \, \big[\Phi_1\otimes \Phi_2\big]  \\
k_{B_1}
& = k_B \, \big[\Phi_1\otimes (1-\Phi_2)+ (1-\Phi_1)\otimes \Phi_2\big]   \\
k_{B_2}
& = k_B \, \big[(1-\Phi_1)\otimes (1-\Phi_2)\big]  \\
\end{aligned}
\end{equation}
and denote by $B_0$, $B_1$ and $B_2$ the $\BG$-invariant operators with
convolution kernels $k_{B_0}$, $k_{B_1}$ and $k_{B_2}$, respectively.
Of course, $B = B_0 + B_1 + B_2$.
The operators $B_0$, $B_1$ and $B_2$ will be analysed in Sections~\ref{s: loc-loc},
\ref{s: loc-glob} and \ref{s: glob-glob}, respectively.
In order to ensure the convergence of various integrals appearing in the proof, \label{p:outline}
we assume that $m_B$ is pre-multiplied by the factor $\wt h_\vep (\la_1,\la_2):=\e^{-\vep (\la_1^2+\la_2^2)}$.
This is no loss of generality as our final bounds depend on $\bignorm{m_B}{\Mar{\TWp}{N}}$,
and $\bignorm{m_B\,\wt h_\vep}{\Mar{\TWp}{N}}$
is convergent to $\bignorm{m_B}{\Mar{\TWp}{N}}$ as $\vep$ tends to $0$.
Recall that
\begin{equation*} 
k_B(a)
=  \int_{\fra^*} \vp_{\la}(a)\, m_B(\la) \wrt \nu(\la), 
\end{equation*}
and that $\vp_{\la}= \vp_{\la_1}^{(1)}\otimes \vp_{\la_2}^{(2)}$.  
The strategy to analyse $B_0$, $B_1$ and $B_2$ is simple:  we expand
the spherical functions $\vp_{\la_1}^{(1)}$ and $\vp_{\la_2}^{(2)}$ by
using either the local or the Harish-Chandra
asymptotic expansion of Subsection~\ref{s: BX rank one}, according to whether the
spherical functions are evaluated near the origin or away from the origin,
and multiply these expansions out.   

\noindent
\textbf{Analysis of $k_{B_0}$}.
The support of $k_{B_0}$ is contained in a neighbourhood of the origin.
Thus, only the values of $a_1$ and $a_2$ near $1$ matter.
By \eqref{f: SpherDecSmall},
applied to both $\vp_{\la_1}^{(1)}$ and $\vp_{\la_2}^{(2)}$, 
\begin{equation} \label{f: dec for B0}
k_{B_0}
= \kAA + \kAR + \kRA + \kRR.
\end{equation}  
Explicitly,
$
\ds \kAA (a_1,a_2)
= \Phi_1(a_1) \, \Phi_2(a_2) \, \int_{\fra^*} A_1(\la_1,a_1) \, A_2(\la_2,a_2)
  \, m_B(\la_1,\la_2) \, \wrt \nu_1(\la_1) \wrt \nu_2(\la_2);
$
similar formulae hold for $\kRA$, $\kAR$ and $\kRR$.
We see that 
$\kAR$ is in $\lu{\BX_2; \Cvp{\BX_1}}$ and $\kRA$ is in $\lu{\BX_1; \Cvp{\BX_2}}$
by Proposition~\ref{p: kB022}~\rmii, and that $\kRR$ is in $\lu{\BX}$, by Proposition~\ref{p: kB022}~\rmi.  
By Theorem~\ref{t: Transference principle}~\rmi, $\kAR$, $\kRA$ and $\kRR$ belong to $\Cvp{\BG}$,
hence to $\Cvp{\BX}$ for these kernels are $\BK$--bi-invariant.
Also $\kAA$ is in $\Cvp{\BX}$ by Proposition~\ref{p: kB011}~\rmiv.  
Therefore $k_{B_0}$ is in $\Cvp{\BX}$.  

\noindent
\textbf{Analysis of $k_{B_1}$}.
The support of the kernel $k_{B_1}$ is contained in a neighbourhood of the walls
of the Weyl chamber and has positive distance from the origin.
Then either $a_1$ is large and $a_2$ is close to the identity, or conversely.
We shall focus on the first case; the second case follows from the first by simply interchanging the 
roles of $\BX_1$ and $\BX_2$.  
We apply the Harish-Chandra expansion \eqref{f: HC expansion} to $\vp_{\la_1}^{(1)}$, 
the local expansion \eqref{f: SpherDecSmall} to~$\vp_{\la_2}^{(2)}$, and obtain
\begin{equation} \label{f: dec for B1}
k_{B_1} 
= 2\kuA + 2\kuR + 2\kop
\end{equation}
near the wall of the Weyl chamber that is annihilated by $\al_2$, where 
$$
\begin{aligned}
\kuA (a_1,a_2)
& =  [1-\Phi_1(a_1)]\Phi_2(a_2)\int_{\fra^*} \!\!a_1^{i\la_1-\rho_1}\,
     \, A_2(\la_2,a_2)\, \cinvmB(\la_1,\la_2) \wrt \la_1 \wrt \nu_2(\la_2)\\
\kuR(a_1,a_2)
& =  [1-\Phi_1(a_1)]\Phi_2(a_2)\int_{\fra^*} \!\!a_1^{i\la_1-\rho_1}\,
     \, R_2(\la_2,a_2)\, \cinvmB(\la_1,\la_2) \wrt \la_1 \wrt \nu_2(\la_2) \\
\kop(a_1,a_2)
& =  [1-\Phi_1(a_1)]\Phi_2(a_2)\int_{\fra^*}\!\! a_1^{i\la_1-\rho_1-2\al_1}\,  \om(\la_1,a_1)
     \, \vp_{\la_2}(a_2)\, \cinvmB(\la_1,\la_2) \wrt \la_1 \wrt \nu_2(\la_2)
\end{aligned}
$$
for every $(a_1,a_2)$ in $\BA_1^+\times\BA_2^+$.  
We have used Remark~\ref{rem: spherical inversion} twice to obtain the formulae above.  
We shall prove that $\kop$, $\kuR$ and $\kuA$ are in $\Cvp{\BX}$ in Proposition~\ref{p: kB1}~\rmi-\rmiii\ below.

\noindent
\textbf{Analysis of $k_{B_2}$}.
The support of the kernel $k_{B_2}$ is contained in a set where both $a_1$ and $a_2$ are large.
We apply Harish-Chandra's expansion \eqref{f: HC expansion} to both $\vp_{\la_1}^{(1)}$
and~$\vp_{\la_2}^{(2)}$, and use Remark~\ref{rem: spherical inversion} twice, and obtain that
\begin{equation} \label{f: dec for B2}
k_{B_2} 
= 4\kuu + 4\kuo + 4\kou + 4\koo,
\end{equation}
where
$$
\begin{aligned}
\kuu(a_1,a_2) 
& = [1-\Phi_1(a_1)]\, [1-\Phi_2(a_2)] \int_{\fra^*} \!\! a_1^{i\la_1-\rho_1} a_2^{i\la_2 -\rho_2}\, \mbcuu(\la)\wrt \la\\
\kuo(a_1,a_2) 
& = [1-\Phi_1(a_1)]\, [1-\Phi_2(a_2)] \int_{\fra^*} \!\! a_1^{i\la_1-\rho_1} a_2^{i\la_2 -\rho_2-2\al_2}\,
	 \om_2(\la_2,a_2) \, \mbcuu(\la)\wrt \la\\
\kou(a_1,a_2) 
& = [1-\Phi_1(a_1)]\, [1-\Phi_2(a_2)] \int_{\fra^*} \!\! a_1^{i\la_1-\rho_1-2\al_1} a_2^{i\la_2 -\rho_2}\,
	 \om_1(\la_1,a_1) \, \mbcuu(\la)\wrt \la\\
\koo(a_1,a_2) 
& = [1-\Phi_1(a_1)]\, [1-\Phi_2(a_2)] \int_{\fra^*} \!\! a_1^{i\la_1-\rho_1-2\al_1} a_2^{i\la_2 -\rho_2-2\al_2}\,
	 \om_1(\la_1,a_1) \, \om_2(\la_2,a_2) \, \mbcuu(\la)\wrt \la 
\end{aligned}
$$
for every $(a_1,a_2)$ in $\BA_1^+\times\BA_2^+$.  
Here we have set $\mbcuu := \big[\check\bc_1^{-1}\otimes \check\bc_2^{-1}\big] \, m_B$.
We shall prove that $\kuo$, $\kou$, $\koo$ and $\kuu$ belong to $\Cvp{\BX}$ in Proposition~\ref{p: kB2} below.
\endproof

\section[Analysis of $B_0$]{Analysis of the operator $B_0$} \label{s: loc-loc}

In this section we prove that the kernel $k_{B_0}$, defined in \eqref{f: splitting of kB}, is in $\Cvp{\BX}$.  
Recall the decomposition \eqref{f: dec for B0}.  

\begin{proposition} \label{p: kB022}
There exists a constant $C$, independent of $B$, such that the following hold:
\begin{enumerate}
\item[\itemno1]
$\bignorm{\kRR}{\lu{\BX}} \leq C  \bignorm{m_B}{\Marinfty{\fra^*}{N}}$;
\item[\itemno2]
$ \bignorm{\kAR}{\lu{\BX_2;\Cvp{\BX_1}}} \leq C \bignorm{m_B}{\Marinfty{\fra^*}{N}}$.
A similar statement holds for~$\kRA$, with the roles of $\BX_1$ and $\BX_2$
interchanged.
\end{enumerate}
\end{proposition}

\begin{proof}
The proof of \rmi\ is a straightforward consequence of Lemma~\ref{l: one dim multipliers loc}.  Indeed,  
$$
\kRR (a_1,a_2)
= \Phi_1(a_1) \, \int_{\fra_1^*}\!\! \wrt \nu_1(\la_1) \, R_1(\la_1,a_1)
\, \Phi_2(a_2) \int_{\fra_2^*} \, R_2(\la_2,a_2) \, m_B(\la_1,\la_2) \, \wrt \nu_2(\la_2),
$$
whence, by Lemma~\ref{l: one dim multipliers loc}~\rmii\ applied to $\BX_1$,  
$$
\bignorm{\kRR (\cdot,a_2)}{\lu{\BX_1}}
\leq C \,  \Phi_2(a_2) \Bignorm{\int_{\fra_2^*} \, R_2(\la_2,a_2) \, m_B(\cdot,\la_2) \, 
\wrt \nu_2(\la_2)}{\Horminfty{\fra_1^*}{N_1}}.
$$
Thus, we need to estimate 
$$
\Phi_2(a_2) \sup_{\la_1\in \fra_1^*} \, (1+\mod{\la_1})^{j_1} 
\Bigmod{\int_{\fra_2^*} \, R_2(\la_2,a_2) \, \partial_{\la_1}^{j_1} m_B(\cdot,\la_2) \, \wrt \nu_2(\la_2)},
$$
which, by Lemma~\ref{l: one dim multipliers loc}~\rmii\ applied to $\BX_2$, is dominated by 
$$
\frac{\Phi_2(a_2)}{(\log a_2)^{n_2-1} }  \sup_{\la_1\in \fra_1^*} \, (1+\mod{\la_1})^{j_1} 
\sup_{\la_2\in \fra_2^*} \, (1+\mod{\la_2})^{j_2} 
\bigmod{\partial_{\la_2}^{j_2}\partial_{\la_1}^{j_1} m_B(\la_1,\la_2)}
\leq  \bignorm{m_B}{\Marinfty{\fra^*}{N}} \, \frac{\Phi_2(a_2)}{(\log a_2)^{n_2-1}}.
$$
Therefore we have proved the pointwise bound
$$
\bignorm{\kRR (\cdot,a_2)}{\lu{\BX_1}}
\leq C \bignorm{m_B}{\Marinfty{\fra^*}{N}} \, \frac{\Phi_2(a_2)}{(\log a_2)^{n_2-1}}.
$$
The required estimate follows by integrating both sides on $\BA_2^+$ with respect to the measure $\ds \de_2(a_2) \wrt {a_2}$.  

Next we prove the estimate of $\kAR$ in \rmii.
The proof of that of~$\kRA$ is similar, and is omitted.  Recall that 
$$
\kAR (a_1,a_2)
= \Phi_1(a_1) \, \int_{\fra_1^*}\!\! \wrt \nu_1(\la_1) \, A_1(\la_1,a_1)
\, \Phi_2(a_2) \int_{\fra_2^*}  \! R_2(\la_2,a_2) \, m_B(\la_1,\la_2) \, \wrt \nu_2(\la_2).
$$
We proceed much as in the proof of \rmi.  By Lemma~\ref{l: one dim multipliers loc}~\rmi, 
applied to $\BX_1$, we see that 
$$
\bignorm{\kAR (\cdot,a_2)}{\Cvp{\BX_1}}
\leq C \,  \Phi_2(a_2) \Bignorm{\int_{\fra_2^*} \! R_2(\la_2,a_2) \, m_B(\cdot,\la_2) \, 
\wrt \nu_2(\la_2)}{\Horminfty{\fra_1^*}{N_1}}.
$$
As in the proof of \rmi\ this implies the pointwise bound
$$
\bignorm{\kAR (\cdot,a_2)}{\Cvp{\BX_1}}
\leq C \bignorm{m_B}{\Marinfty{\fra^*}{N}} \, \frac{\Phi_2(a_2)}{(\log a_2)^{n_2-1}}.
$$
The required estimate follows by integrating both sides on $\BA_2^+$ with respect to the measure $\ds \de_2(a_2) \wrt {a_2}$.  
\end{proof}

\noindent
Next we analyse $\kAA$.
Recall that the support of $\kAA$ is a compact neighbourhood of the origin.
It is convenient to further decompose $\kAA$ via a partition of unity ``on the
Fourier transform side'' as follows.  Recall the function $\bC$, defined just above 
\eqref{f: kernels rank one}.  Set $\Psi_1(\la_1) := \bC \big(\la_1(H_0)\big)$ for every $\la_1$
in $\fra_1^*$, and define $\Psi_2$ similarly on $\fra_2^*$.  
We shall use the following smooth finite partition of unity on $\fra_1^*\times\fra_2^*$
\begin{equation} \label{f: smooth finite partition}
1
= \Psi_1\otimes\Psi_2 + (1-\Psi_1)\otimes \Psi_2 + \Psi_1\otimes (1-\Psi_2) + (1-\Psi_1)\otimes (1-\Psi_2).
\end{equation}
Correspondingly, we may write
\begin{equation*} 
\kAA
= \kappa_{0,0} + \kappa_{1,0} + \kappa_{0,1} + \kappa_{1,1}.
\end{equation*}
Note that each of the kernels on the right hand side is a $\BK$--bi-invariant function on~$\BG$.
The kernel $\kappa_{0,0}$ is given by the formula
\begin{equation} \label{f: kappa00}
\begin{aligned}
\kappa_{0,0}(a_1, a_2) 
& = \Phi_1(a_1)\, \Phi_2(a_2) \, \int_{\fra*} A(\la,a_1,a_2) \, [\Psi_1\otimes\Psi_2](\la) 
       \, m_B (\la) \wrt \nu(\la), 
\end{aligned}
\end{equation}
where
$A(\la,a_1,a_2) := A_1(\la_1,a_1)\, A_2(\la_2,a_2)$.
The formulae for $\kappa_{1,0}$, $\kappa_{0,1}$ and $\kappa_{1,1}$ are similar, but
with  $(1-\Psi_1)\otimes \Psi_2$, $\Psi_1\otimes (1-\Psi_2)$ and $(1-\Psi_1)\otimes (1-\Psi_2)$
in place of $\Psi_1\otimes\Psi_2$.
We denote by 
$\cO_1$ 
and by $\cO_1^*$ 
the differential operators 
$$
\cO_1 \psi (\la_1)
= \frac{1}{\la_1} \partial_{\la_1}\psi(\la_1)
\qquad\hbox{and}\qquad
\cO_1^* \psi (\la_1)
= -\partial_{\la_1}\Big[\frac{\psi (\la_1)}{\la_1}\Big].   
$$ 
We define $\cO_2$ and $\cO_2^*$ similarly on $\fra_2^*$.   

\begin{lemma} \label{l: est for M}
Suppose that $m$ is in $\Marinfty{\fra^*}{N}$, and set  
$M := [1-\Psi_1]\,[1-\Psi_2] \, m\mod{\bc}^{-2}$.  The following hold
\begin{enumerate}
\item[\itemno1]
the function $M$ satisfies the estimate
$$
\bigmod{\partial^\al (\cO_1^*)^{j_1} (\cO_2^*)^{j_2} M(\la_1,\la_2)}
		\leq C\, \frac{\norm{m}{\Marinfty{\fra^*}{N}}}{\big(1+\mod{\la_1}\big)^{1+2j_1+\al_1-n_1}  
\,  \big(1+\mod{\la_2}\big)^{1+2j_2+\al_2-n_2}}  
$$
for all multi-indices $\al = (\al_1,\al_2)$ with $\al_1$ and $\al_2$ in $\{0,1,2\}$, and 
$(j_1,j_2) \leq N-\al$;   
\item[\itemno2]
if $n_1$ and $n_2$ are even, 
then there exists a constant $C$ such that 
for all multi-indices $\al = (\al_1,\al_2)$ with $\al_1$ and $\al_2$ in $\{0,1\}$
$$
\Bigmod{\partial_v^\al\partial_{v_1v_2}^2\int_0^1 \int_0^1 
\frac{\big[(\cO_1^*)^{n_1/2-1}(\cO_2^*)^{n_2/2-1}  M\big](v_1\la_1',v_2\la_2')}{
     \sqrt{1-\la_1'^2} \sqrt{1-\la_2'^2}} \wrt \la_1'\wrt \la_2'} 
\leq C \, \frac{\bignorm{m}{\Marinfty{\fra^*}{N}}}{(1+\mod{v_1})^{\al_1}\, (1+\mod{v_2})^{\al_2}};
$$
\item[\itemno3]
if $n_1$ is even and $n_2$ is odd, then there exists a constant $C$ such that 
for all multi-indices $\al = (\al_1,\al_2)$ with $\al_1$ and $\al_2$ in $\{0,1\}$
$$
\Bigmod{\partial_v^\al\partial_{v_1}\int_0^1 
\frac{\big[(\cO_1^*)^{n_1/2-1}(\cO_2^*)^{(n_2-1)/2}  M\big](v_1\la_1',v_2)}{\sqrt{1-\la_1'^2}} \wrt \la_1'} 
\leq C \, \frac{\bignorm{m}{\Marinfty{\fra^*}{N}}}{(1+\mod{v_1})^{\al_1}\, (1+\mod{v_2})^{\al_2}}.
$$
\end{enumerate} 
\end{lemma}

\begin{proof}
The proof of \rmi\ is a straightforward consequence of an elementary induction
argument and the estimate \eqref{f: HCest} for the Harish-Chandra function~$\mod{\bc}^{-2}$.  
We omit the details.  

Next we prove \rmii.  By differentiating under the integral sign, we are led to estimate
$$
\mod{\la_1'}^{\al_1+1}\, \mod{\la_2'}^{\al_2+1} \,
\bigmod{\big[\partial_{v_1}^{\al+1}\partial_{v_2}^{\al_2+1}
(\cO_1^*)^{n_1/2-1}(\cO_2^*)^{n_2/2-1}  M\big](v_1\la_1',v_2\la_2')},
$$
which, by \rmi, is dominated by  
\begin{equation} \label{f: function on square}
C\, \mod{\la_1'}^{\al_1+1}\, \mod{\la_2'}^{\al_2+1} \,
\frac{\norm{m}{\Marinfty{\fra^*}{N}}}{\big(1+\mod{v_1\la_1'}\big)^{\al_1}  
\,  \big(1+\mod{v_2\la_2'}\big)^{\al_2}}.  
\end{equation} 
Notice that 
$\ds
\int_0^1 
\frac{\mod{\la_1'}^{\al_1+1}}{\big(1+\mod{v_1\la_1'}\big)^{\al_1}}  
\, \frac{\wrt \la_1'}{\sqrt{1-\la_1'^2}} 
$
is uniformly bounded as long as $v_1$ is bounded, say when $v_1$ is in $[-1,1]$, and 
is majorized by
$$
\frac{2}{\sqrt 3} \int_0^{1/2} 
\frac{\mod{\la_1'}^{\al_1+1}}{\big(1+\mod{v_1\la_1'}\big)^{\al_1}}  \, \wrt \la_1'
+ \int_{1/2}^1 \frac{\mod{\la_1'}^{\al_1+1}}{\big(1+\mod{v_1/2}\big)^{\al_1}}  
\, \frac{\wrt \la_1'}{\sqrt{1-\la_1'^2}}
$$
when $\mod{v_1} \geq 1$.  
The second of these two integrals is clearly dominated by $C \, \big(1+\mod{v_1}\big)^{\al_1}$.
To bound the first, we change variables ($v_1\la_1'=\la_1''$) and are led to estimate
$$
\mod{v_1}^{-\al_1-2} \int_0^{v_1/2} \frac{\mod{\la_1''}^{\al_1+1}}{\big(1+\mod{\la_1''}\big)^{\al_1}}  
\, \wrt \la_1''
\leq C \, \mod{v_1}^{-\al_1}. 
$$
A similar estimate holds for 
$\ds
\int_0^1 
\frac{\mod{\la_2'}^{\al_2+1}}{\big(1+\mod{v_2\la_2'}\big)^{\al_2}}  
\, \frac{\wrt \la_2'}{\sqrt{1-\la_2'^2}}.
$
Thus, the integral of \eqref{f: function on square} on $[0,1]^2$ with respect to the measure 
$\ds \frac{\wrt \la_1'\wrt \la_2'}{\sqrt{1-\la_1'^2} \sqrt{1-\la_2'^2}}$ 
is dominated by 
$\ds C \, \frac{\bignorm{m}{\Marinfty{\fra^*}{N}}}{(1+\mod{v_1})^{\al_1}\, (1+\mod{v_2})^{\al_2}}$,
as required.

Finally, the proof of \rmiii\ follows the lines of the proof of \rmii.  We omit the details. 
\end{proof}

\begin{proposition} \label{p: kB011}
There exists a constant $C$ such that the following hold:
\begin{enumerate}
\item[\itemno1]
$
\bignorm{\kappa_{0,0}}{\lu{\BX}}
\leq C \bignorm{m_B}{\Marinfty{\fra^*}{N}};
$
\item[\itemno2]
$
\bignorm{\kappa_{1,0}}{\lu{\BX_2;\Cvp{\BX_1}}} 
\leq C \bignorm{m_B}{\Marinfty{\fra^*}{N}}.
$
A similar statement, with the roles of $\BX_1$ and $\BX_2$ interchanged, holds for $\kappa_{0,1}$;
\item[\itemno3]
$
\bignorm{\kappa_{1,1}}{\Cvp{\BX}}
\leq C \bignorm{m_B}{\Marinfty{\fra^*}{N}};
$
\item[\itemno4]
$
\bignorm{\kAA}{\Cvp{\BX}}
\leq C \bignorm{m_B}{\Marinfty{\fra^*}{N}}.
$
\end{enumerate}
\end{proposition}

\begin{proof}
First we prove \rmi.  By using formula \eqref{f: SpherDecSmall I} and the estimates
\eqref{f: HCest} for the Harish-Chandra $\bc$ function we see that 
$\Psi_1 \otimes \Psi_2 (\cdot) \, \bigmod{A(\cdot,a_1,a_2)}\, \mod{\bc(\cdot)}^{-2}$
is uniformly bounded on the support of $\Phi_1\otimes\Phi_2$.  
Fr{}om \eqref{f: kappa00} we then deduce that
$$
\begin{aligned}
\bigmod{\kappa_{0,0}(a_1, a_2)}
& \leq C \bignorm{m_B}{\infty} \, \Phi_1(a_1)\, \Phi_2(a_2) \, \int_{\fra^*} 
     (\Psi_1 \otimes \Psi_2)(\la)  \wrt \la \\
& \leq C \bignorm{m_B}{\infty} \,  \Phi_1(a_1)\, \Phi_2(a_2).
\end{aligned}
$$
By integrating both sides in Cartan co-ordinates on $\BG$, and
using the trivial estimate $\bignorm{m_B}{\infty} \leq \bignorm{m_B}{\Marinfty{\fra^*}{N}}$, yields \rmi.

Next we prove \rmii.  At least formally, we may write
\begin{equation*}
\kappa_{1,0}(a_1, a_2)
	= \Phi_1(a_1)\, \int_{\fra_1^*} A_1(\la_1,a_1) \, \mu(\la_1,a_2) \wrt\nu_1(\la_1),
\end{equation*}
where
$$
\mu(\la_1,a_2)
:=  \Phi_2(a_2) 
	\int_{\fra_2^*} \Psi_2(\la_2)\, A_2(\la_2,a_2) \, [(1-\Psi_1)\,m_B(\cdot,\la_2)](\la_1) \wrt\nu_2(\la_2).
$$
We show that there exists a constant $C$, independent of $a_2$, such that, for $j=1,\ldots,N_1,4$ 
\begin{equation} \label{f: Horm mu}
	\sup_{\la_1 \in \fra_1^*} \, (1+ \mod{\la_1})^j  \bigmod{\partial_{\la_1}^j \mu(\la_1,a_2)} 
\leq C \, \Phi_2(a_2).  
\end{equation}
Indeed, 
$$
\la_1^j\partial^j_{\la_1} \mu(\la_1,a_2)
= \Phi_2(a_2)\, \int_{\fra_2^*} \Psi_2(\la_2) \, A_2(\la_2,a_2) \, 
\la_1^j \partial^j_{\la_1}[(1-\Psi_1)\,m_B(\cdot,\la_2)](\la_1) \wrt\nu_2(\la_2).
$$
Clearly $(1-\Psi_1)\,m_B(\cdot,\la_2)$ satisfies a Mihlin--H\"ormander condition of order $N_1 > (n_1+3)/2$ 
at infinity; hence there exists a constant~$C$ such that 
$$
\bigmod{\partial^j_{\la_1}[(1-\Psi_1)\,m_B(\cdot,\la_2)](\la_1)}
\leq C \bignorm{m_B}{\Horminfty{\fra_1^*}{N_1}} \, (1+\mod{\la_1})^{-j}  
\quant \la_2 \in \fra_2^*.  
$$
This and the fact that $A_2$ is bounded yield the required estimate \eqref{f: Horm mu}. 
Now, Lemma~\ref{l: one dim multipliers loc}~\rmi\ implies that 
$$
\bignorm{\kappa_{1,0}(\cdot, a_2)}{\Cvp{\BX_1}} 
\leq C \bignorm{m_B}{\Horminfty{\fra_1^*}{N_1}} \, \Phi_2(a_2).
$$
Integrating on $\BA_2$ with respect to the measure $\de_2(a_2) \wrt a_2$ yields the required estimate.  

We now prove \rmiii.  
Set
$
M := [(1-\Psi_1) \otimes (1-\Psi_2)] \,m_B\mod{\bc}^{-2}.
$
We shall distinguish three cases: both $n_1$ and $n_2$ are odd, both $n_1$ and $n_2$ are even,  
$n_1$ and $n_2$ have different parities.  We shall use repeatedly the fact that 
$\cJ_{\nu+m}(z_j) = (-1)^m \cO^m_j \cJ_\nu(z_j)$ (see, for instance, \cite[formula~5.3.7, p.~103]{L}).  

\label{pag:k24}
Assume first that both $n_1$ and $n_2$ are odd. Then
$$
\mJ_{n_j/2-1}(z)
= (-1)^{(n_j-1)/2} \cO^{(n_j-1)/2} \cJ_{-1/2}(z).
$$ 
This, the fact that $\ds \cJ_{-1/2}(z) = \sqrt{\frac{2}{\pi}} \, \cos z$, and the chain rule yield
\begin{equation*} 
\cJ_{n_j/2-1}(\la_j t_j)
= t_j^{1-n_j} \, \cO_{\la_j}^{(n_j-1)/2} \big[(\cos (\la_j t_j)\big].
\end{equation*}
Integrating by parts we obtain
$$
\kappa_{1,1}(a) 
= \eta(a) \int_{\fra_1^*\times\fra_2^*} P(\la_1,\la_2;\partial_{\la_1},\partial_{\la_2})
     M(\la_1,\la_2) \, \cos (\la_1t_1)\cos (\la_2t_2) \wrt \la_1\wrt \la_2,
$$
where 
$P(\la_1,\la_2;\partial_{\la_1},\partial_{\la_2})
:= [\cO_{\la_1}^{(n_1-1)/2}]^* [\cO_{\la_2}^{(n_2-1)/2}]^*$ and  
$
\ds \eta(a) 
:= c_0^{(1)} c_0^{(2)} \,\frac{\Phi_1(a_1)\Phi_2(a_2)}{t_1^{n_1-1}t_2^{n_2-1}} \, w_1(a_1)w_2(a_2).  
$
The functions $w_1$ and $w_2$ are the analogues on $\BA_1$ and $\BA_2$ of the function $w$ defined,
in the rank one case, below formula \eqref{f: SpherDecSmall I}.
Now, the function $P(\cdot,\cdot;\partial_{\la_1},\partial_{\la_2})M$ is Weyl-invariant; hence $\kappa_{1,1}$
may be rewritten as 
$$
\kappa_{1,1}(a) 
= \eta(a) \int_{\fra_1^*\times\fra_2^*} P(\la_1,\la_2;\partial_{\la_1},\partial_{\la_2})
     M(\la_1,\la_2) \, a_1^{i\la_1} a_2^{i\la_2} \wrt \la_1\wrt \la_2.
$$    
We denote by $K(a)$ the integral above.  
By Lemma~\ref{l: est for M}~\rmi, the function $P(\cdot,\cdot;\partial_{\la_1},\partial_{\la_2})M$
satisfies a Marcinkiewicz condition of order 
$(1,1)$ on $\fra_1^*\times\fra_2^*$.  Hence $K$ 
belongs to $\Cvp{\BA}$ by the classical Marcinkiewicz multiplier theorem.  
Since $A_p(\BA)$ is an algebra under pointwise 
multiplication and $C_c^\infty (\BA)$ is contained in $A_p(\BA)$ 
for all $p$ in $(1,\infty)$ (see  Theorem~\ref{t: Ap spaces} and, in particular, \cite[Theorem~3.4]{Co} 
for a proof of the last statement), the function 
$
\Phi_1(a_1)\Phi_2(a_2) \, \, w_1(a_1)w_2(a_2)\, 
K(a) 
$
is in $\Cvp{\BA}$, whence $\kappa_{1,1}$ is in $\Cvp{\BX}$, by 
the Coifman--Weiss transference result, as required.  

Assume now that both $n_1$ and $n_2$ are even.  Recall that 
$
\ds \cJ_0(z)
= \frac{2}{\pi} \int_1^\infty (v^2-1)^{-1/2}\, \sin v z \wrt v  
$ 
\cite[p.~180]{W}.  Changing variables, we see that 
\begin{equation} \label{f: cJ0}
\cJ_0(\la_j t_j)
= \frac{2}{\pi} \int_{\la_j}^{\infty}(v^2-\la_j^2)^{-1/2}\, \sin v t_j \wrt v.
\end{equation} 
By arguing much as in the previous case, we may write 
\begin{equation*} 
\cJ_{n_j/2-1}(\la_j t_j)
= t_j^{2-n_j} \cO^{n_j/2-1} \big(\cJ_0 (\la_j t_j)\big).
\end{equation*}
Integrating by parts we obtain
$$
\kappa_{1,1}(a) 
= \eta(a_1,a_2) \, t_1t_2 \int_{\fra_1^*\times\fra_2^*} P'(\la_1,\la_2;\partial_{\la_1},\partial_{\la_2})
     M(\la_1,\la_2) \, \cJ_0 (\la_1t_1)\cJ_0 (\la_2t_2) \wrt \la_1\wrt \la_2,
$$
where $\eta$ is as before, and 
$P'(\la_1,\la_2;\partial_{\la_1},\partial_{\la_2})
= \big[\cO_{\la_1}^{n_1/2-1}\big]^* \big[\cO_{\la_2}^{n_2/2-1}\big]^*$. By using \eqref{f: cJ0}
and Fubini's theorem, the last integral tranforms to  
$$
\frac{4}{\pi^2} \int_{\fra_1^*\times\fra_2^*} \wrt v_1\wrt v_2\,  \sin v_1t_1 \sin v_2t_2 
\int_0^{v_1} \int_0^{v_2} 
\frac{P'(\la_1,\la_2;\partial_{\la_1},\partial_{\la_2})M(\la_1,\la_2) }{\sqrt{v_1^2-\la_1^2} 
     \sqrt{v_2^2-\la_2^2}} \wrt \la_1\wrt \la_2.
$$
 By integration by parts the product of $t_1t_2$ and the last integral may be written as a constant times  
$$
\int_{\fra_1^*\times\fra_2^*} \wrt v_1\wrt v_2 \cos v_1t_1 \cos v_2t_2 \, 
\partial_{v_1}\partial_{v_2} \int_0^{v_1} \int_0^{v_2} 
\frac{P'(\la_1,\la_2;\partial_{\la_1},\partial_{\la_2})M(\la_1,\la_2) }{\sqrt{v_1^2-\la_1^2} 
     \sqrt{v_2^2-\la_2^2}} \wrt \la_1\wrt \la_2,
$$
which, by a change of variables ($v_1\la_1' = \la_1$, $v_2\la_2' = \la_2$) and the Weyl-invariance 
of $\partial_{v_1}\partial_{v_2} J(v_1,v_2)$,  may be rewritten as 
$
\ds \int_{\fra_1^*\times\fra_2^*} a_1^{i v_1} a_2^{i v_2}
\, \partial_{v_1}\partial_{v_2} J(v_1,v_2)\wrt v_1\wrt v_2,
$
where 
$$
J(v_1,v_2)
:= \int_0^1 \int_0^1 \frac{\big[P'(\la_1,\la_2;\partial_{\la_1},\partial_{\la_2})M
     \big](v_1\la_1',v_2\la_2')}{\sqrt{1-\la_1'^2} \sqrt{1-\la_2'^2}} \wrt \la_1'\wrt \la_2'.
$$
By Lemma~\ref{l: est for M}~\rmii\ the function $\partial_{v_1}\partial_{v_2}J$ satisfies 
a Marcinkiewicz condition of order $(1,1)$ on $\fra_1^*\times \fra_2^*$, 
whence its inverse Mellin transform belongs to $\Cvp{\BA}$ for all $p$ in $(1,\infty)$.  
By arguing much as in the conclusion of the first part \rmiii, we see that $\kappa_{1,1}$ is in $\Cvp{\BX}$, as required.  

Finally, assume that $n_1$ is even and $n_2$ is odd.  Define
$
P''(\la_1,\la_2;\partial_{\la_1},\partial_{\la_2})
= [\cO_{\la_1}^{n_1/2-1}]^* [\cO_{\la_2}^{(n_2-1)/2}]^*.
$ 
Integrating by parts, and mimicking the reasoning in \rmi\ and \rmii, we see that 
$$
\kappa_{1,1}(a) 
= \frac{2}{\pi} \, \eta(a_1,a_2) \, 
\int_{\fra_1^*\times\fra_2^*} \cos v_1t_1 \cos v_2t_2 \, \partial_{v_1} H(v_1,v_2) \wrt v_1\wrt v_2, 
$$
where
$$
H(v_1,v_2)
:= \int_0^{v_1} \frac{P''(\la_1,v_2;\partial_{\la_1},\partial_{v_2})M(\la_1,v_2)}{
\sqrt{v_1^2-\la_1^2}} \wrt \la_1.
$$
A reasoning similar to that used in the proof of \rmi\ and \rmii\ proves \rmiii.  We omit the details.  

Finally, \rmiv\ follows from \rmi-\rmiii\ and the fact that 
$\kAA = \kappa_{0,0} + \kappa_{1,0} + \kappa_{0,1} + \kappa_{1,1}$.

The proof of the proposition is complete.
\end{proof}

\section{Analysis of the operator $B_1$} \label{s: loc-glob} 

In this section we estimate $k_{B_1} (a)$ where $a=(a_1,a_2)$, $a_1$ is large and $a_2$ is close to the identity.  Recall that
$k_{B_1} = \big[(1-\Phi_1)\otimes\Phi_2 + \Phi_1\otimes(1-\Phi_2)\big]\,k_{B}$.  Thus, for such values of $a_1$ and 
$a_2$ we simply have $k_{B_1} = \big[(1-\Phi_1)\otimes\Phi_2 \big]\,k_{B}$.  Also, recall that 
$$
\kuA (a)
=  [1-\Phi_1(a_1)]\, \Phi_2(a_2)\int_{\fra^*} \!\! a_1^{i\la_1-\rho_1}\,
     \, A_2(\la_2,a_2)\, \cinvmB(\la_1,\la_2) \wrt \la_1 \wrt \nu(\la_2) 
\quant a \in \BA^+;
$$
here, with a slight abuse of notation we denote by $\bc_1$ 
the natural extension to $\fra^*$ of the Harish-Chandra $\bc$-function on $\fra_1^*$
obtained by making it constant on sections.  We define 
\begin{equation*} 
\begin{aligned}
\MAd(\la_1,a_2)
& := \Phi_2(a_2) \int_{\fra_2^*} A_2(\la_2,a_2)\, 
       m_B(\la_1,\la_2) \wrt \nu(\la_2) \\ 
\MRd(\la_1,a_2)
& := \Phi_2(a_2) \int_{\fra_2^*} R_2(\la_2,a_2)\, 
       m_B(\la_1,\la_2) \wrt \nu(\la_2) \\ 
\end{aligned}
\end{equation*}
and $\Mpd := \MAd + \MRd$.  We also define
\begin{equation} \label{f: N}
\begin{aligned}
\Nu(a_1,\la_2)
& := [1-\Phi_1(a_1)] \, a_1^{\de(p)\rho_1} \int_{\fra_1^*} a_1^{i\la_1}\, 
       \cinvmB(\la_1,\la_2) \wrt \la_1  \\  
\Nou(a_1,\la_2)
& := [1-\Phi_1(a_1)] \, \int_{\fra_1^*} a_1^{i\la_1-\rho_1-2\al_1}\,\om_1(\la_1,a_1)  
       \, \cinvmB(\la_1,\la_2) \wrt \la_1  \\  
\end{aligned}
\end{equation}
Thus,  
\begin{equation*} 
\begin{aligned}
\kuA (a_1,a_2)
& =  [1-\Phi_1(a_1)]\, \int_{\fra_1^*} \!\! a_1^{i\la_1-\rho_1}\, \cinvMAd(\la_1,a_2) \wrt \la_1 \\ 
	& =  \Phi_2(a_2)\, \int_{\fra_2^*} \!\! \, A_2(\la_2,a_2)\,a_1^{-2\rho_1/p} \Nu(a_1,\la_2) \wrt \nu_2(\la_2). 
\end{aligned}
\end{equation*} 
For notational convenience, for each $p$ in $(1,2)$ and $\vep>0$ (small) 
we set $\rho_1^{p,\vep} := (\de(p)-\vep) \rho_1$ and $\rho_2^{p,\vep} := (\de(p)-\vep) \rho_2$.
Inspired by the notation adopted in the proof of Lemma~\ref{l: one dim multipliers glob}, we define 
\begin{equation} \label{f: phip product}
\begin{aligned} 
\phipuA(a_1,a_2)
&           := \big[1-\Phi_1(a_1)\big]\, a_1^{\de(p)\rho_1} \int_{\fra_1^*} a_1^{i\la_1} \, \cinvMAd(\la_1,a_2)\wrt\la_1 \\
& \phantom{:}= \Phi_2(a_2) \int_{\fra_2^*} A_2(\la_2, a_2) \, \Nu(a_1,\la_2) \wrt\nu_2(\la_2)
\end{aligned}
\end{equation}
for every $a\in \BA_1\times\BA_2^+$.  
Changing the path of integration from $\fra_1$ to $\fra_1 + i\rho_1^{p,\vep(a_1)}$, we obtain the useful formula 
$$
\phipuA(a_1,a_2)
:= \big[1-\Phi_1(a_1)\big]\, \e^{\sign(\log a_1) \vep(a_1) \mod{\rho_1}}
\int_{\fra_1^*} a_1^{i\la_1} \, \cinvMAd(\la_1+i\rho_1^{p,\vep(a_1)},a_2)\wrt\la_1
$$
where $\vep(a_1) = \vep/ \mod{\log a_1}$.  Observe that, trivially, 
\begin{equation} \label{f: kuA and phip}
\kuA (a_1,a_2)
= a_1^{-2\rho_1/p} \, \phipuA(a_1,a_2)
\quant (a_1,a_2)\in \BA^+.  
\end{equation}
We interpret $\phipuA$ as a $\BK_2$--bi-invariant function on $\OBN_1\BA_1\times \BX_2$, 
and define the functions
$$
\begin{aligned}
\taupuuA (v_1b_1,x_2)
	& := \chi_{\BA_1^+}(b_1) \, P(v_1)^{2/p} \, \big[\exp\big(E(v_1,b_1)H_0\big)-1\big] \, \phipuA\big([v_1b_1]_+,x_2\big) \\
\taupduA (v_1b_1,x_2)
	& := \chi_{\BA_1^+}(b_1) \, P(v_1)^{2/p} \, \big[\phipuA\big([v_1b_1]_+,a_2\big) - \phipuA(b_1,x_2) \big]\\
\tauptuA (v_1b_1,x_2)
	& := \chi_{\BA_1^+}(b_1) \, P(v_1)^{2/p} \, \phipuA(b_1,x_2) 
\end{aligned}
$$
for every $v_1\in \OBN$,  $b_1 \in \BA$ and $x_2 \in \BX_2$.
Hereafter, with a slight abuse of notation, $\chi_{\BA_1^+}$ will also denote the function on $\OBN_1\BA_1\times \BG_2$,
defined by $\chi_{\BA_1^+}(v_1b_1, g_2) = \chi_{\BA_1^+}(b_1)$, and similarly for $\chi_{\BA_1^-}$.  This 
notation is consistent with that adopted in Section~\ref{s: Background} for rank one symmetric spaces.  

\begin{lemma} \label{l: multipliers loc glob}
Suppose that $N> (n+\tre)/2$ and $1<p<2$.  Then there exists
a constant~$C$ such that for every $m$ in $\Mar{\TWp}{N}$ the following hold:
\begin{enumerate}
\item[\itemno1]
for every $a_1\in \BA_1^+$ 
$$
\begin{aligned}
\bignorm{\phipuA(a_1,\cdot)}{\Cvp{\BX_2}}
& \leq C  \bignorm{m_B}{\Mar{T_p}{N}} \, \frac{1-\Phi_1(a_1)}{\log a_1}\\ 
\bignorm{a_1\partial_{a_1}\phipuA(a_1,\cdot)}{\Cvp{\BX_2}}  
& \leq C  \bignorm{m_B}{\Mar{T_p}{N}} \, \frac{1-\Phi_1(a_1)}{\log^2 a_1}; 
\end{aligned}
$$
\item[\itemno2]
$\kuA(a_1,a_2) = a_1^{-2\rho_1/p} \, \phipuA(a_1, a_2)$ for all $a_1$ in $\BA_1^+$ and $a_2$ in $\BA_2^+$, and 
$$
\bignorm{\kuA(v_1 b_1, \cdot)}{\Cvp{\OBN_2\BA_2}} 
\leq C \bignorm{m}{\Mar{T_p}{N}} \big[1-\Phi_1([v_1b_1]_+)\big] \,  b_1^{-2\rho_1/p}\, P(v_1)^{2/p} 
$$  
for every $v_1\in\OBN_1$ and $b_1\in\BA_1^+$;  
\item[\itemno3]
$\ds \int_{\OBN_1}\int_{\BA_1} \bignorm{[\chi_{\BA_1^-}\cD_1^{1/p}\kuA](v_1b_1,\cdot)}{\Cvp{\OBN_2\BA_2}} 
\wrt v_1  \dtt {b_1}  
\leq C \bignorm{m}{\Mar{T_p}{N}}$;
\item[\itemno4]
$\ds \bignorm{\phipuA([v_1b_1]_+,\cdot)-\phipuA(b_1,\cdot)}{\Cvp{\BX_2}}
\leq C  \bignorm{m_B}{\Mar{\TWp}{N}}\,  \frac{\mod{\al(H(v_1))}+1}{1+\log^2b_1}.
$ 
\end{enumerate}
\end{lemma}

\begin{proof}
We prove \rmi.  
Recall formula \eqref{f: phip product} that we rewrite here for the reader's convenience
\begin{equation} \label{f: phip in terms of Nu}
\phipuA(a_1,a_2)
= \Phi_2(a_2) \int_{\fra_2^*} A_2(\la_2,a_2) \, \Nu(a_1,\la_2) \wrt\nu_2(\la_2),  
\end{equation}  
By moving the contour of integration from $\fra_1$ to $\fra_1+i\de(p)\rho_1$ in the definition of $\Nu$, we see that  
\begin{equation*} 
\Nu(a_1,\la_2)
= [1-\Phi_1(a_1)] \int_{\fra_1^*} a_1^{i\la_1}\, 
       \cinvmB(\la_1+i\de(p) \rho_1,\la_2) \wrt \la_1.
\end{equation*}
Now, Lemma \ref{l: one dim multipliers loc}~\rmi\ (with $\Nu(a_1,\cdot)$ in place of $m$) 
implies that 
\begin{equation*} 
\bignorm{\phipuA(a_1,\cdot)}{\Cvp{\BX_2}} 
\leq C \bignorm{\Nu(a_1,\cdot)}{\Horminfty{\fra_2^*}{N_2}}.  
\end{equation*}
Thus, the first estimate in \rmi\ will follow from 
\begin{equation} \label{f: legame 1}
\bignorm{\Nu(a_1,\cdot)}{\Horminfty{\fra_2^*}{N_2}} 
\leq C \bignorm{m_B}{\Mar{T_p}{N}}\, \frac{1-\Phi_1(a_1)}{\log a_1} 
\quant a_1\in \BA_1^+.  
\end{equation}
To prove \eqref{f: legame 1} we observe that, 
by Lemma~\ref{l: phip basic} (with $\partial_{\la_2}^{j_2}m_B(\cdot,\la_2)$ in place of $m$), 
$$
\begin{aligned}
\bigmod{\partial_{\la_2}^{j_2}\Nu(a_1,\la_2)}
& = [1-\Phi_1(a_1)] \, \Bigmod{\int_{\fra_1^*} a_1^{i\la_1}\, \cinvderladjdmB(\la_1+i\de(p) \rho_1,\la_2) \wrt \la_1} \\
& \leq C \bignorm{\partial_{\la_2}^{j_2}m_B(\cdot,\la_2)}{\Horm{T_p^{(1)}}{N_1}} \,  \frac{1-\Phi_1(a_1)}{\log a_1} 
\end{aligned}
$$
for all $a_1\in \BA_1^+$ and $j_2\in \{0,1,\cdots,N_2\}$.  It is straightforward to check that there 
exists a constant $C$ such that 
$\ds\bignorm{\partial_{\la_2}^{j_2}m_B(\cdot,\la_2)}{\Horm{T_p^{(1)}}{N_1}} 
\leq C\, \frac{\bignorm{m_B}{\Mar{T_p}{N}}}{(1+\mod{\la_2})^{j_2}}$.  
By combining the last two estimates, we get \eqref{f: legame 1}, thereby concluding the proof of the first 
estimate in \rmi.

As to the second bound in \rmi, we start from formula \eqref{f: phip in terms of Nu}.  
Then Lemma~\ref{l: one dim multipliers loc}~\rmi\ implies that 
$$
\bignorm{a_1\partial_{a_1} \phipuA(a_1,\cdot)}{\Cvp{\BX_2}}
\leq C \bignorm{a_1\partial_{a_1}\Nu(a_1,\cdot)}{\Horminfty{\fra_2^*}{N_2}}.  
$$
The required estimate will follow from 
\begin{equation} \label{f: legame 2}
\bignorm{a_1\partial_{a_1}\Nu(a_1,\cdot)}{\Horminfty{\fra_2^*}{N_2}} 
\leq C \bignorm{m_B}{\Mar{T_p}{N}}\, \frac{1-\Phi_1(a_1)}{\log^2 a_1} 
\quant a_1\in \BA_1^+.  
\end{equation}
Suppose that $j_2\in \{0,1,\ldots,N_2\}$, and observe that 
$$
\partial_{\la_2}^{j_2} \Nu(a_1,\la_2)
= [1-\Phi_1(a_1)] \, \e^{\vep\mod{\rho_1} \sign\, a_1} \int_{\fra_1^*} a_1^{i\la_1} \, 
   \cinvderladjdmB(\la_1+i\rho_1^{p,\vep(a_1)}, \la_2) \wrt \la_1,  
$$
which follows by changing of the path of integration in the integral defining 
$\Nu$ (see \eqref{f: N}) from $\fra_1^*$ to $\fra_1^*+i\rho_1^{p,\vep(a_1)}$.  
Notice that the factor $\e^{\vep\mod{\rho_1} \sign\, a_1}$
is constant on $\BA_1^-$ and on $\BA_1^+$, hence its derivative on $\BA_1\setminus \{0\}$ vanishes.  Therefore
$a_1\partial_{a_1} \partial_{\la_2}^{j_2} \Nu(a_1,\la_2)$ is the sum~of 
\begin{equation} \label{f: legame 3}
\ds -a_1\, \Phi_1'(a_1) \, \e^{\vep\mod{\rho_1} \sign\, a_1} 
\int_{\fra_1^*} a_1^{i\la_1} \, \cinvderladjdmB(\la_1+i\rho_1^{p,\vep(a_1)}, \la_2) \wrt \la_1  
\end{equation} 
and 
$$
\Om(a_1,\la_2) 
:= i\, [1-\Phi_1(a_1)] \, \e^{\vep\mod{\rho_1} \sign\, a_1} 
\int_{\fra_1^*} a_1^{i\la_1} \, 
    \la_1\, \cinvderladjdmB(\la_1+i \rho_1^{p,\vep(a_1)}, \la_2) \wrt \la_1.  
$$
The first of the two terms above is much easier to estimate than the second, and we leave it to the 
interested reader.  We integrate by parts $N_1$ times, and obtain that the integral above may be written as 
$$
\frac{1}{(i\log a_1)^{N_1}} \int_{\fra_1^*} a_1^{i\la_1} \, 
\big(\la_1\, \partial_{\la_1}^{N_1}+ {N_1}\partial_{\la_1}^{N_1-1}\big)\big\{\cinvderladjdmB(\cdot+i \rho_1^{p,\vep(a_1)}, \la_2)
\big\}(\la_1) \wrt \la_1. 
$$
A straightforward calculation shows that the absolute value of the integrand is dominated by 
$$
C \, \bignorm{\partial_{\la_2}^{j_2}m_B(\cdot,\la_2)}{\Horm{T_p^{(1)}}{N_1}} \, (1+\mod{\la_2})^{-j_2}  \,{
\Big[\frac{\la_1}{\mod{\la_1-i\rho_1\vep/\log a_1}^{N_1}} + \frac{1}{\mod{\la_1-i\rho_1\vep/\log a_1}^{N_1-1}}\Big]},   
$$
whence
$$
\bigmod{\Om(a_1,\la_2)} 
\leq C \bignorm{\partial_{\la_2}^{j_2}m_B(\cdot,\la_2)}{\Horm{T_p^{(1)}}{N_1}} \, (1+\mod{\la_2})^{-j_2}  
\, \frac{1-\Phi_1(a_1)}{\log^2 a_1}. 
$$
Since the term \eqref{f: legame 3} satisfies a similar estimate, 
$$
\bigmod{\partial_{\la_2}^{j_2} \Nu(a_1,\la_2)} 
\leq C \bignorm{\partial_{\la_2}^{j_2}m_B(\cdot,\la_2)}{\Horm{T_p^{(1)}}{N_1}} \, (1+\mod{\la_2})^{-j_2}  
\, \frac{1-\Phi_1(a_1)}{\log^2 a_1},
$$
the bound \eqref{f: legame 2} follows.  This completes the proof of \rmi.  

The first equality in \rmii\ is formula \eqref{f: kuA and phip} above.  
For each $v_1$ in $\OBN_1$ and $b_1$ in $\BA_1$ the function $\kuA(v_1b_1,\cdot)$ is $\BK_2$--bi-invariant. 
Therefore $\bignorm{\kuA (v_1b_1,\cdot)}{\Cvp{\OBN_2\BA_2}} = \bignorm{\kuA (v_1b_1,\cdot)}{\Cvp{\BX_2}}$,  
and 
$$
\begin{aligned}
\bignorm{\kuA (v_1b_1,\cdot)}{\Cvp{\BX_2}}
& =    [v_1b_1]_+^{-2\rho_1/p} \bignorm{\phipuA (v_1b_1,\cdot)}{\Cvp{\BX_2}} \\ 
& \leq C \bignorm{m_B}{\Mar{T_p}{N}}\, \big[1-\Phi_1(v_1b_1)\big] \,[v_1b_1]_+^{-2\rho_1/p} \\ 
\end{aligned}
$$
by \rmi.  Now, the fact that $[v_1b_1]_+^{-2\rho_1/p} = P(v_1)^{2/p} \,  b_1^{-2\rho_1/p} \, 
\e^{-(2/p)\mod{\rho_1} E(v_1,b_1)} \leq P(v_1)^{2/p} \,  b_1^{-2\rho_1/p}$ (for $b_1 \in \BA_1^+$ by assumption) 
implies the required estimate.   

We now prove \rmiii.  
Set $F(v_1b_1) := \bignorm{\kuA (v_1b_1,\cdot)}{\Cvp{\OBN_2\BA_2}}$.  
Notice that $F$, interpreted as a function on $\BX_1$, is $\BK_1$--invariant.  
Thus, its Abel transform 
$\ds \cA_1 F(b_1) := b_1^{\rho_1} \int_{\OBN_1} F(v_1b_1) \wrt v_1$
is a Weyl-invariant function on $\BA_1$, whence $\cA_1 F(b_1) = \cA_1 F(b_1^{-1})$ and 
the left hand side of the required estimate may be re-written~as  
$$
\int_{\BA_1^-} b_1^{\de(p)\rho_1} \, \cA_1 F(b_1) \dtt{b_1}  
= \int_{\BA_1^+} b_1^{-\de(p)\rho_1} \, \cA_1 F(b_1) \dtt{b_1}.    
$$
By \rmii, the right hand side can be  majorized by 
$$
\begin{aligned}
C \bignorm{m_B}{\Mar{T_p}{N}} \int_{\OBN_1} \wrt v_1 \, P(v_1)^{2/p} 
	\int_{\BA_1^+} b_1^{2(1/p'-1/p)\rho_1} \, \dtt{b_1} 
& \leq C \bignorm{m_B}{\Mar{T_p}{N}}; 
\end{aligned}
$$
we have used 
the fact that $1/p'<1/p$ (because $1<p<2$) in the last inequality. 
This concludes the proof of \rmiii.  

To prove \rmiv\ we start from formula \eqref{f: phip in terms of Nu}, and observe that 
Lemma~\ref{l: one dim multipliers loc}~\rmi\ implies the bound
$$
\begin{aligned}
\bignorm{\phipuA([v_1b_1]_+,\cdot) - \phipuA(b_1,\cdot)}{\Cvp{\BX_2}}
& \leq C \bignorm{\Nu([v_1b_1]_+,\cdot) - \Nu(b_1,\cdot)}{\Horminfty{\fra_2^*}{N_2}} \\
& \leq C \int_{b_1}^{[v_1b_1]_+} \!\! \bignorm{\partial_{a_1} \Nu(a_1,\cdot)]}{\Horminfty{\fra_2^*}{N_2}} \wrt a_1, 
\end{aligned}
$$
where the second inequality follows from the mean value theorem.  By \eqref{f: legame 2}, the last integral is dominated by 
$$
\begin{aligned}
C \bignorm{m_B}{\Mar{T_p}{N}} \int_{b_1}^{[v_1b_1]_+}\frac{1-\Phi_1(a_1)}{\log^2 a_1} \dtt {a_1} 
& \leq C \bignorm{m_B}{\Mar{T_p}{N}} \, \frac{1-\Phi_1(b_1)}{\log^2 b_1} \, [\log{[v_1b_1]_+}-\log {b_1}] \\ 
& \leq C \bignorm{m_B}{\Mar{T_p}{N}} \, \frac{\log{[v_1b_1]_+}-\log {b_1}}{1+\log^2 b_1}.
\end{aligned}
$$
Notice that $\log{[v_1b_1]_+}-\log {b_1} = \al_1(H(v_1)) + E(v_1,b_1)\leq \al_1(H(v_1)) + 2$ for every $v_1\in \OBN_1$
and every $b_1\in \BA_1^+$.  By combining the estimates above, we obtain the required bound.  
\end{proof}

\begin{proposition} \label{p: kB1}
There exists a constant $C$ such that the following hold:
\begin{enumerate}
\item[\itemno1]
$\bignorm{\kop}{\Cvp{\BX}} 
\leq C \bignorm{m_B}{\Mar{T_p}{N}}$;  
\item[\itemno2]
$\bignorm{\kuR}{\Cvp{\BX}} 
\leq C \bignorm{m_B}{\Mar{T_p}{N}}$;  
\item[\itemno3]
$\bignorm{\kuA}{\Cvp{\BX}} 
\leq C \bignorm{m_B}{\Mar{T_p}{N}}$.  
\end{enumerate}
\end{proposition}

\begin{proof}
First we prove \rmi.  It is convenient to write the kernel $\kop$, which is defined below formula \eqref{f: dec for B1}, 
as follows
$$
\kop(a_1,a_2)
	:= \Phi_2(a_2) \int_{\fra_2^*} \vp_{\la_2}(a_2) \, \Nou(a_1,\la_2) \wrt \nu_2(\la_2).  
$$
By Lemma~\ref{l: one dim multipliers loc} (with $\Nou(a_1,\cdot)$ in place of $m$), 
$
\norm{\kop(a_1,\cdot)}{\Cvp{\BX_2}}
\leq C \, \norm{\Nou(a_1,\cdot)}{\Horminfty{\fra_2^*}{N_2}}.
$
Thus, we are led to estimating $\bigmod{\partial_{\la_2}^{j_2}\Nou(a_1,\la_2)}$ for $j_2\leq N_2$.  
We move the contour of integration from~$\fra_1^*$ to $\fra_1^* + i\rho_1^{p,\vep}$
in the definition of $\Nou$ (see \eqref{f: N}), differentiate under the integral sign, and obtain that 
\begin{equation*} 
\partial_{\la_2}^{j_2} \Nou(a_1,\la_2)
:= a_1^{(\vep-2/p)\rho_1-2\al_1} \int_{\fra_1^*} a_1^{i\la_1}\, 
\om_1(\la_1+i\rho_1^{p,\vep},a_1)\, \partial_{\la_2}^{j_2} \cinvmB(\la_1+i\rho_1^{p,\vep},\la_2) \wrt \la_1.
\end{equation*} 
We now use the
estimate \eqref{f: pointwise est kappaominfty} obtained in the rank one case, and conclude that 
$$
\bigmod{\partial_{\la_2}^{j_2}\Nou(a_1,\la_2)}
\leq C \bignorm{\partial_{\la_2}^{j_2}m_B(\cdot, \la_2)}{\Horm{T_p^{(1)}}{N_1}} \, \frac{1-\Phi_1(a_1)}{(\log a_1)^{N_1}} 
	\, a_1^{(\vep-2/p)\rho_1 -2\al_1}, 
$$
whence
$$
\bignorm{\kop(a_1,\cdot)}{\Cvp{\BX_2}}
\leq C \bignorm{m_B}{\Mar{T_p}{N}} \, \frac{1-\Phi_1(a_1)}{(\log a_1)^{N_1}} 
	\, a_1^{(\vep-2/p)\rho_1 -2\al_1}. 
$$
Now we multiply both sides by $\vp_{i\de(p)\rho_1}^{(1)}$, integrate on $\BA_1^+$ with respect to 
the measure $\de(a_1) \wrt a_1$, use Lemma~\ref{l: iterated conv}~\rmiii, and obtain that 
$$
\begin{aligned}
\bignorm{\kop}{\Cvp{\BX}} 
& \leq C \bignorm{m_B}{\Mar{T_p}{N}}
    \int_{\BA_1^+} \big[1-\Phi_1(a_1)\big] \, a_1^{(\vep-2/p)\rho_1-2\al_1} \vp_{i\de(p)\rho_1}^{(1)}(a_1) 
    \, \de(a_1) \wrt a_1. 
\end{aligned}     
$$
Since the last integrand is $\asymp a_1^{(\vep-2/p)\rho_1-2\al_1-2\rho_1/p'+2\rho_1}$ as $a_1$ tends to infinity, 
the last integral is convergent, provided that $\vep$ is small enough.  The required conclusion follows directly from
this.

Next we prove \rmii.  The kernel $\kuR$ is defined slightly below formula \eqref{f: dec for B1}.  Observe that 
$$
\kuR(a_1,a_2)
=  [1-\Phi_1(a_1)]\, \int_{\fra_1^*} a_1^{i\la_1-\rho_1}\, \cinvMRd(\la_1,a_2) \wrt \la_1. 
$$
Note that $\bignorm{\kuR(\cdot,a_2)}{\Cvp{\BX_1}} \leq C \bignorm{\MRd(\cdot, a_2)}{\Horm{T_p^{(1)}}{N_1}}$, 
by Lemmata~\ref{l: phip basic} and~\ref{l: one dim multipliers glob}~\rmiii-\rmiv\ 
(or see the proof of \rmv\ in that lemma).
In order to majorize the right hand side, we are led to estimating 
$\Theta_1(\la_1)^{j_1}\, \bigmod{\partial_{\la_1}^{j_1} \MRd(\la_1, a_2)}$.  By differentiating under the integral sign,
we see that 
$$
\partial_{\la_1}^{j_1} \MRd(\la_1, a_2)
= \Phi_2(a_2)\int_{a_2^*} \!\! R_2(\la_2,a_2)\, \partial_{\la_1}^{j_1}m_B(\la_1,\la_2) \wrt \nu(\la_2).  
$$
Now, Lemma~\ref{l: one dim multipliers loc}~\rmii\ implies the pointwise bound 
$$
\bigmod{\partial_{\la_1}^{j_1} m(\la_1, a_2)}
\leq C \bignorm{\partial_{\la_1}^{j_1}m_B(\la_1,\cdot)}{\Horminfty{\fra_2^*}{N_2}} \, \frac{\Phi_2(a_2)}{(\log a_2)^{n_2-1}}.
$$
Thus, 
$\ds\sup_{\la_1\in T_p^{(1)}} \, \Theta_1(\la_1)^{j_1}\, \bigmod{\partial_{\la_1}^{j_1} m(\la_1, a_2)}
\leq C \bignorm{m_B}{\Mar{T_p}{N}} \, \frac{\Phi_2(a_2)}{(\log a_2)^{n_2-1}}$, and we end up with the estimate
$\ds\bignorm{\kuR(\cdot,a_2)}{\Cvp{\BX_1}} \leq C \bignorm{m_B}{\Mar{T_p}{N}} \, \frac{\Phi_2(a_2)}{(\log a_2)^{n_2-1}}$ 
We now multiply both sides by $\de_2(a_2)$ and integrate on $\BA_2$, and obtain that 
$\ds\bignorm{\kuR}{\lu{\BX_2;\Cvp{\BX_1}}} \leq C \bignorm{m_B}{\Mar{T_p}{N}}$, which, in view of 
Lemma~\ref{l: iterated conv}~\rmi, implies the required estimate.  

Finally, we prove \rmiii.   
We interpret $\kuA$ as a function on $\OBN_1\BA_1\times\OBN_2\BA_2$, viewed as the semidirect product of 
$\OBN_1\times\{e\}$ acted on by $\BA_1\times \OBN_2\BA_2$.  
By Corollary~\ref{c: Transference principle}~\rmi\ (with $\OBN_2\BA_2$ in place of $\BG_2$),  
\begin{equation} \label{f: chiP and chiM}
\bignorm{\kuA}{\Cvp{\OBN_1\BA_1\times\OBN_2\BA_2}} 
\leq \int_{\OBN_1} \bignorm{[\cD_1^{1/p}\,\kuA] (v_1\cdot, \cdot)}{\Cvp{\BA_1\times\OBN_2\BA_2}} 
\wrt v_1.  
\end{equation} 
Here $\cD_1$ denotes the function on $\OBN_1\BA_1\times\OBN_2\BA_2$
defined by $\cD_1(v_1b_1,v_2b_2) = b_1^{2\rho_1}$.   
We write  $\kuA = \kuA \, \chi_{\BA_1^-} +   \kuA\, \chi_{\BA_1^+}$.
This, \eqref{f: chiP and chiM} and the triangle inequality imply that it suffices to prove the estimates
\begin{equation} \label{f: chiM}
\int_{\OBN_1} \bignorm{[\cD_1^{1/p} \kuA \, \chi_{\BA_1^-}](v_1\cdot,\cdot)}{\Cvp{\BA_1\times\OBN_2\BA_2}}  \wrt v_1
\leq C\bignorm{m_B}{\Mar{T_p}{N}}
\end{equation} 
and 
\begin{equation} \label{f: chiP}
\int_{\OBN_1} \bignorm{[\cD_1^{1/p}\, \kuA \, \chi_{\BA_1^+}](v_1\cdot,\cdot)}{\Cvp{\BA_1\times\OBN_2\BA_2}}  \wrt v_1
\leq C\bignorm{m_B}{\Mar{T_p}{N}}.
\end{equation} 
The estimate of \eqref{f: chiM} is the easiest of the two.  Indeed,  
Lemma~\ref{l: iterated conv}~\rmi\ (with $\BA_1$ in place of $\Ga_1$ and $\OBN_2\BA_2$
in place of $\Ga_2$) implies that \eqref{f: chiM} is dominated by 
\begin{equation*} 
\int_{\OBN_1} \wrt v_1 \int_{\BA_1^-} b_1^{2\rho/p} \, \bignorm{\kuA (v_1b_1,\cdot)}{\Cvp{\OBN_2\BA_2}} \dtt {b_1}.
\end{equation*}
Then \eqref{f: chiM} follows directly from Lemma~\ref{l: multipliers loc glob}~\rmiii. 

It remains to prove \eqref{f: chiP}.  
In view of the decomposition $\cD_1^{1/p}\kuA \chi_{\BA_1^+} = \taupuuA + \taupduA + \tauptuA$, 
it suffices to show that  
\begin{equation*} 
\int_{\OBN_1} \bignorm{\tau_{1A_2}^{p,j}(v_1\cdot,\cdot)}{\Cvp{\BA_1\times\OBN_2\BA_2}}  \wrt v_1
\leq C\bignorm{m_B}{\Mar{T_p}{N}}
\end{equation*} 
for $j=1,2,3$.  
This will be a consequence of the following estimates 
\begin{equation} \label{f: chiP 1 and 2}
\int_{\OBN_1}\wrt v_1 \int_{\BA_1} \bignorm{\tau_{1A_2}^{p,j}(v_1b_1,\cdot)}{\Cvp{\OBN_2\BA_2}} \dtt{b_1}
\leq C\bignorm{m_B}{\Mar{T_p}{N}}
\end{equation} 
when $j=1,2$, and, by Corollary~\ref{c: Transference principle}~\rmi\ and \rmii, 
\begin{equation} \label{f: chiP 3}
\int_{\OBN_1} \bignorm{\big[1\otimes\de_2\big]\tauptuA(v_1\cdot,\cdot)}{\Cvp{\BA_1\times\BA_2}} \wrt v_1  
\leq C\bignorm{m_B}{\Mar{T_p}{N}}.  
\end{equation} 
Fr{}om the definition of $\taupuuA$ (just above Lemma~\ref{l: multipliers loc glob}) and 
the inequality $|E(v_1,b_1)-1|\leq 2b_1^{-2\al_1}$ for all $b_1$ in $\BA_1^+$, we see that  
$$
\begin{aligned}
	\bignorm{\taupuuA(v_1b_1,\cdot)}{\Cvp{\OBN_2\BA_2}}
& \leq 2 P(v_1)^{2/p} \,\chi_{\BA_1^+}(b_1)  \bignorm{\phipuA([v_1b_1]_+,\cdot)}{\Cvp{\BX_2}} \, b_1^{-2\al_1} \\
& \leq C \bignorm{m_B}{\Mar{T_p}{N}} \,  P(v_1)^{2/p} \,\chi_{\BA_1^+}(b_1) \,  b_1^{-2\al_1}; 
\end{aligned}
$$
we have used Lemma~\ref{l: multipliers loc glob}~\rmi\ in the last inequality.  Inserting this estimate in 
the left hand side of \eqref{f: chiP 1 and 2}, 
and recalling that $P(v_1)^{2/p}$ is in $\lu{\OBN_1}$, because $2/p>1$, proves \eqref{f: chiP 1 and 2} when $j=1$. 

Similarly, fr{}om the definition of $\taupduA$ (just above Lemma~\ref{l: multipliers loc glob}), 
we see that  
$$
\begin{aligned}
\bignorm{\taupduA(v_1b_1,\cdot)}{\Cvp{\OBN_2\BA_2}}
& \leq 2 P(v_1)^{2/p} \,\chi_{\BA_1^+}(b_1)  \bignorm{\phipuA([v_1b_1]_+,\cdot)-\phipuA(b_1,\cdot)}{\Cvp{\BX_2}} \\
& \leq C \bignorm{m_B}{\Mar{T_p}{N}} \,  \big[\bigmod{\al_1(H(v_1))}+1\big] P(v_1)^{2/p} \,
	\frac{\chi_{\BA_1^+}(b_1)}{1+(\log b_1)^2}; 
\end{aligned}
$$
we have used Lemma~\ref{l: multipliers loc glob}~\rmiv\ in the last inequality.  By inserting this estimate in 
the left hand side of \eqref{f: chiP 1 and 2}, 
and recalling that $\big[\mod{\al_1(H(v_1))}+1\big]\, P(v_1)^{2/p}$ is in $\lu{\OBN_1}$, 
because $2/p>1$, we obtain \eqref{f: chiP 1 and 2} when $j=2$. 

It remains to prove \eqref{f: chiP 3}.  Notice that the left hand side of \eqref{f: chiP 3} is equal to 
$$
\bignorm{\big[\chi_{\BA_1^+} \otimes\de_2\big]\phipuA}{\Cvp{\BA_1\times\BA_2}} \int_{\OBN_1} P(v_1)^{2/p} \wrt v_1
\leq C \bignorm{\big[\chi_{\BA_1^+} \otimes\de_2\big]\phipuA}{\Cvp{\BA_1\times\BA_2}}. 
$$  
We focus on estimating the right hand side of the last inequality above.  It is convenient to define
$$
\Upsilon (b_1,a_2)
:= b_1^{\de(p)\rho_1} \,\de_2(a_2) \int_{\fra^*} 
      b_1^{i\la_1} \, A_2(\la_2,a_2)  \, \cinvmB (\la_1,\la_2) \wrt \la_1 \wrt \nu_2(\la_2).  
$$
We write $\Psi_2(\la_2) + [1-\Psi_2(\la_2)]$ in place of $1$ in the integral above: then 
$\Upsilon$ may be correspondingly written as the sum of two terms, which we
denote by $\Upsilon^{\Psi_2}$ and $\Upsilon^{1-\Psi_2}$.  Thus,  
$$
\big[\chi_{\BA_1^+} \otimes\de_2\big]\phipuA 
= [\chi_{\BA_1^+}(1-\Phi_1) \otimes \Phi_2]\, \big(\Upsilon^{\Psi_2} + \Upsilon^{1-\Psi_2}\big)
$$
We claim that 
$\bignorm{[\chi_{\BA_1^+}(1-\Phi_1) \otimes \Phi_2]\Upsilon^{\Psi_2}}{\lu{\BA_2;\Cvp{\BA_1}}}
\leq C \bignorm{m}{\Mar{\TWp}{N}}$. 
Indeed, notice that 
$$
\Upsilon^{\Psi_2} (b_1,a_2)
= \de_2(a_2) \, b_1^{\de(p)\rho_1} 
   \int_{\fra_1^*} b_1^{i\la_1} \, [\check\bc_1^{-1} M_{\Psi_2A_2}] (\la_1,a_2) \wrt \la_1,
$$
where $M_{\Psi_2A_2} (\la_1,a_2) := \int_{\fra_2^*} A_2(\la_2,a_2) \, \Psi_2(\la_2) \, m_B(\la_1,\la_2) \wrt\nu_2(\la_2)$.  
By Lemma~\ref{l: one dim multipliers glob} 
$$
\bignorm{\chi_{\BA_1^+}(1-\Phi_1)\Upsilon^{\Psi_2}(\cdot, a_2)}{\Cvp{\BA_1}} 
\leq C \bignorm{M_{\Psi_2A_2}(\cdot, a_2)}{\Horm{T_p^{(1)}}{N_1}} \, \de_2(a_2).
$$
The claim follows by multiplying both sides by $\Phi_2$,
integrating both sides on $\BA_2$ and observing that 
$\bignorm{M_{\Psi_2A_2}(\cdot, a_2)}{\Horm{T_p^{(1)}}{N}} \leq \bignorm{m_B}{\Mar{T_p}{N}}$.  

Now we estimate $\Upsilon^{1-\Psi_2}$.  Notice that 
$$
\Upsilon^{1-\Psi_2} (b_1,a_2)
= \de_2(a_2) \,  b_1^{\de(p)\rho_1} 
   \int_{\fra_1^*} b_1^{i\la_1} \, [\check\bc_1^{-1} M_{[1-\Psi_2]A_2}] (\la_1,a_2) \wrt \la_1,
$$
where $M_{[1-\Psi_2]A_2} (\la_1,a_2) := \int_{\fra_2^*} A_2(\la_2,a_2) \, [1-\Psi_2(\la_2)] \, m_B(\la_1,\la_2) 
\wrt\nu_2(\la_2)$.  We observe preliminary that 
\begin{equation*} 
\bignorm{[\chi_{\BA_1^-}(1-\Phi_1) \otimes \Phi_2] \Upsilon^{1-\Psi_2}}{\lu{\BA_1;\Cvp{\BA_2}}}
\leq C \bignorm{m}{\Mar{T_p}{N}}.
\end{equation*}
Indeed, by integrating by parts $N_1$ times, we see that 
$$
\Upsilon^{1-\Psi_2} (b_1,a_2)
= \de_2(a_2) \, \frac{ b_1^{\de(p)\rho_1}}{(i\log b_1)^{N_1}}
   \int_{\fra_2^*}\wrt\nu_2(\la_2)  A_2(\la_2,a_2) \, [1-\Psi_2(\la_2)] 
   \int_{\fra_1^*} b_1^{i\la_1} \, \partial_{\la_1}^{N_1}[\check\bc_1^{-1} \, m_B](\la_1,\la_2) \wrt \la_1,
$$
The proof of \cite[Lemma~5.3]{ST} shows that 
$$
\bignorm{\Phi_2 \Upsilon^{1-\Psi_2}(b_1,\cdot)}{\Cvp{\BA_2}}
\leq C \, \frac{ b_1^{\de(p)\rho_1}}{\mod{\log b_1}^{N_1}} 
\Bignorm{[1-\Psi_2] 
\int_{\fra_1^*} b_1^{i\la_1} \, \partial_{\la_1}^{N_1}[\check\bc_1^{-1} \, m_B](\la_1,\cdot) \wrt \la_1}{\Horm{T_p^{(2)}}{N_2}}.
$$
A straightforward calculation shows that the right hand side is dominated by $C \bignorm{m}{\Mar{T_p}{N}}$, whence
$$
\begin{aligned}
\bignorm{[\chi_{\BA_1^-}(1-\Phi_1) \otimes \Phi_2] \Upsilon^{1-\Psi_2}}{\lu{\BA_1;\Cvp{\BA_2}}}
& \leq C \bignorm{m}{\Mar{T_p}{N}}  \int_{\BA_1^-}  [1-\Phi_1(b_1)]\, 
	\frac{ b_1^{\de(p)\rho_1}}{\mod{\log b_1}^{N_1}} \dtt{b_1}  \\
& \leq C \bignorm{m}{\Mar{T_p}{N}}, 
\end{aligned}
$$
as required.  

Now we move the path of integration from $\fra_1^*$ to $\la_1+i\de(p)\rho_1$, and obtain that 
$$
\Upsilon^{1-\Psi_2} (b_1,a_2)
= \de_2(a_2) \, 
   \int_{\fra_1^*} b_1^{i\la_1} \, [\check\bc_1^{-1} M_{[1-\Psi_2]A_2}] (\la_1+i\de(p) \rho_1,a_2) \wrt \la_1,
$$
We shall prove that $\Phi_2\Upsilon^{1-\Psi_2}$ is in $\Cvp{\BA_1\times\BA_2}$.
Define $\wt m(\la_1,\la_2) := [1-\Psi_2(\la_2)] \,[(\check\bc_1^{-1}\otimes \mod{\bc_2}^{-2})m_B](\la_1+i\de(p)\rho_1,\la_2)$. 
Recall the operators $\cO_2$ and $\cO_2^*$ defined just above Lemma~\ref{l: est for M},  
and denote by $q$ and $\Pi$ the function on $\fra^*$ and the measure on $\fra_2^*$, defined by 
$$
\hbox{$q(\la) 
:= 
\begin{cases}
(\cO_2^*)^{(n_2-1)/2} \wt m(\la)                & \hbox{if $n_2$ is odd} \\ 
\partial_{\la_2}(\cO_2^*)^{n_2/2-1}\wt m(\la) & \hbox{if $n_2$ is even} 
\end{cases}$}
\qquad
\hbox{$\Pi 
:= 
\begin{cases}
\de_{{1_2}}                         & \hbox{if $n_2$ is odd} \\ 
(1-\mu_2^2)^{-1/2}\mu_2\wrt \mu_2 & \hbox{if $n_2$ is even}, 
\end{cases}$}
$$
respectively; the symbol $\de_{1_2}$ above stands for the Dirac delta at $1$ in $\fra_2^*$.  A careful inspection of the proof of Proposition~\ref{p: kB011}~\rmiii, or of the proof of \cite[Lemma~5.3]{ST}, shows that 
$$
M_{[1-\Psi_2]A_2} (\la_1,a_2) := c_0^{(2)} \, \frac{\Phi_2(a_2)}{\log^{n_2-1}a_2} \, w_2(a_2)
\int_{\fra_2^*} \cos(\la_2t_2) \, \cP(\la_1,\la_2) \wrt\la_2,
$$  
where 
$\ds
\cP(\la_1,\la_2)
:= \int_0^1 q(\la_1,\mu_2 \la_2)\wrt \Pi(\mu_2).
$
Note that for each $\la_1\in \fra_1^*$ the function $\cP(\la_1,\cdot)$ is Weyl-invariant on $\fra_2^*$. 
Thus,  
$$
\begin{aligned}
\big[\Phi_2 \Upsilon^{1-\Psi_2}\big] (b_1,a_2)
& = c_0^{(2)} \, \frac{\Phi_2(a_2)}{w_2(a_2)} \, 
  \int_{\fra^*} b_1^{i\la_1} \, \cos(\la_2t_2) \, \cP(\la_1,\la_2)  \wrt \la_1\wrt \la_2 \\
& = c_0^{(2)} \, \frac{\Phi_2(a_2)}{w_2(a_2)} \, 
	\int_{\fra^*} b_1^{i\la_1} \, a_2^{i\la_2} \, \cP(\la_1,\la_2)  \wrt \la_1\wrt \la_2.  
\end{aligned}
$$
%
%
A straightforward computation and Lemma~\ref{l: est for M}~\rmi\ imply that $q$ satisfies the estimate
$$
\sup_{\la\in\fra^*}\bigmod{\la_1}^{\mod{j_1}} \, \bigmod{\la_2}^{\mod{j_2}}
	\bigmod{\partial_{\la_1}^{j_1}\partial_{\la_2}^{j_2} q(\la_1,\la_2)}
\leq C\,\bignorm{m}{\Mar{\TWp}{N}}
\quad\hbox{for $j_1\leq N_1$ and $j_2\leq N_2$}.  
$$
Consequently, a similar estimate is true for $\cP$.  An application of the classical Marcinkiewicz theorem
(see Theorem~\ref{t: classical Marc}) gives that the function 
$$
\cM^{-1}\cP(b_1,a_2) 
:= \int_{\fra^*} b_1^{i\la_1} \, a_2^{i\la_2} \, \cP(\la_1,\la_2)  \wrt \la_1\wrt \la_2
$$ 
satisfies the estimate $\bignorm{\cM^{-1}\cP}{\Cvp{\BA}} \leq C\,\bignorm{m}{\Mar{\TWp}{N}}$.
Notice that $\Phi_2/w_2$ is a smooth function with compact support on $\BA_2$.  Since
$\Phi_1\otimes (\Phi_2/w_2)$ is in $C_c^\infty(\BA)$,
it is also in $A_p(\BA)$, by Theorem~\ref{t: Ap spaces}.  Therefore $[\Phi_1\otimes (\Phi_2/w_2)]\, \cM^{-1}\cP$ 
is in $\Cvp{\BA}$, i.e.,   
$[\Phi_1\otimes\Phi_2]\, \Upsilon^{1-\Psi_2}$ is in $\Cvp{\BA_1\times\BA_2}$.  Also,
by results of Cowling and Haagerup, $[1\otimes (\Phi_2/w_2)]$ is in $A_p(\BA)$, so that 
$[1\otimes (\Phi_2/w_2)]\, \cM^{-1}\cP$ is in $\Cvp{\BA}$, i.e.,   
$[1\otimes\Phi_2]\, \Upsilon^{1-\Psi_2}$ is in $\Cvp{\BA}$.  
Recall that we have just proved that  
$[\chi_{\BA_1^-}(1-\Phi_1) \otimes \Phi_2] \, \Upsilon^{1-\Psi_2}$ is in ${\lu{\BA_1;\Cvp{\BA_2}}}$,
hence in $\Cvp{\BA}$.  By difference, we get that   
$[\chi_{\BA_1^+}(1-\Phi_1) \otimes \Phi_2] \, \Upsilon^{1-\Psi_2}$ is in $\Cvp{\BA_1\times\BA_2}$,
as required. 

This concludes the proof of \rmiii, and of the proposition. 
\end{proof}

\section{Analysis of the operator $B_2$}
\label{s: glob-glob}

In this section we estimate $k_{B_2}= [(1-\Phi_1)\otimes (1-\Phi_2)] \, k_{B}$.  As explained
in Section~\ref{s: Statement}, it suffices to estimate $\kuu$, $\kuo$, $\kou$ and $\koo$, which are defined just below
formula \eqref{f: dec for B2}.  We need more notation, which is reminiscent of that introduced in Section~\ref{s: Background}
and at the beginning of Section~\ref{s: loc-glob}.  Set
\begin{equation} \label{f: phip product II}
\phipuu(a_1,a_2)
:= \big[1-\Phi_1(a_1)\big]\,\big[1-\Phi_2(a_2)\big]\, a^{\de(p)\rho}\int_{\fra^*} a^{i\la} \,
	\mbcuu(\la)\wrt\la \\
\end{equation}
for every $(a_1, a_2) \in \BA$, where $\mbcuu = \big[\check\bc_1^{-1}\otimes\check\bc_2^{-1}\big] \, m_B$ 
(this notation has been already used slightly below formula \eqref{f: dec for B2}).  Recall that
$m_B$ is in $\Mar{T_p}{N}$ by assumption.  
Observe that trivially
\begin{equation*} 
\kuu (a)
= a^{-2\rho/p} \, \phipuu(a)
\quant a \in \BA^+.  
\end{equation*}
Also, set
\begin{equation*} 
\Xi(a)
:= \int_{\fra^*} a^{i\la} \, [\mbcuu_{\de(p)\rho}] (\la) \wrt\la, 
\end{equation*} 
where $[\mbcuu_{\de(p)\rho}] (\la) := \mbcuu(\la+i\de(p)\rho)$.  Notice that the integral above is convergent,
because we agreed to premultiply $m_B$ by $\wt h_\vep$
(see the beginning of \emph{Outline of the proof of Theorem~1.1} in Section~\ref{s: Statement}, p.\ \pageref{p:outline}).  
Recall the smooth partition of unity on $\fra^*$ 
$$
1
= \Psi_1\otimes\Psi_2 + (1-\Psi_1)\otimes \Psi_2 + \Psi_1\otimes (1-\Psi_2) + (1-\Psi_1)\otimes (1-\Psi_2),
$$
defined in \eqref{f: smooth finite partition}.  Correspondingly, we write 
$
\Xi = \Xi^{0,0} + \Xi^{\infty,0} + \Xi^{0,\infty} + \Xi^{\infty,\infty},
$
where 
$$
\ds \Xi^{\infty,0} (a) 
= \int_{\fra^*} [1-\Psi_1(\la_1)]\, \Psi_2(\la_2) \, a^{i\la} \, \mbcuu_{\de(p)\rho} (\la) \wrt\la,
$$ 
and similar formulae hold for 
$\Xi^{0,0}$, $\Xi^{0,\infty}$ and $\Xi^{\infty,\infty}$.

\begin{lemma} \label{l: prop of Xi}
The following hold:
\begin{enumerate}
\item[\itemno1]
$\Xi^{0,0}$ is in $Cv_q(\BA)$ for all $q$ in $(1,\infty)$;  
\item[\itemno2]
$[1-\Phi_1][1-\Phi_2]\, \Xi^{\infty,\infty}$ is in $\lu{\BA}$;  
\item[\itemno3]
$[1-\Phi_2]\, \Xi^{0,\infty}$ is in $\lu{\BA_2;Cv_q(\BA_1)}$ for all $q$ in $(1,\infty)$;  
\item[\itemno4]
$[1-\Phi_1]\, \Xi^{\infty,0}$ is in $\lu{\BA_1;Cv_q(\BA_2)}$ for all $q$ in $(1,\infty)$.  
\end{enumerate}
In each of the four statements above, the norms of the relevant functions are controlled by $C\, \Mar{T_p}{N}$.
\end{lemma}

\begin{proof}
First we prove \rmi.  We simply note that $\Xi^{0,0}$ is the inverse Mellin transform of
a multiplier in $\Mar{\fra^*}{N}$.  Indeed, 
$$
\bigmod{\partial_{\la_1}^{j_1}\partial_{\la_2}^{j_2} \{[\Psi_1\otimes\Psi_2]\, [\mbcuu]_{\de(p)\rho}\}(\la)}   
\leq C \bignorm{m_B}{\Mar{\TWp}{N}} \frac{\Psi_1(\la_1)}{\mod{\la_1}^{j_1} }\, \frac{\Psi_2(\la_2)}{\mod{\la_2}^{j_2}}
\quant \la\in \fra^* 
$$
for all multi-indices $(j_1,j_2) \leq (N_1,N_2)$.  This is more than needed to conclude that 
$\Xi^{0,0}$ is in $Cv_q(\BA)$ for all $q$ in $(1,\infty)$.

Next we prove \rmii.  By integrating by parts $N_1$ times with respect to $\la_1$ and $N_2$
times with respect to $\la_2$, we obtain that 
$$
\begin{aligned}
\bigmod{\Xi^{\infty,\infty}(a)}
& \leq \mod{\log a_1}^{-N_1} \,\mod{\log a_2}^{-N_2} \int_{\fra^*} \bigmod{\partial_{\la_1}^{N_1}
	\partial_{\la_2}^{N_2}\{[(1-\Psi_1)\otimes(1-\Psi_2)] [\mbcuu]_{\de(p)\rho}(\la) \}}  \wrt\la \\
& \leq C \, \frac{\bignorm{m_B}{\Mar{T_p}{N}}}{\mod{\log a_1}^{N_1} \,\mod{\log a_2}^{N_2}} \int_{\fra^*} 
         (1+\mod{\la_1})^{-N_1+(n_1-1)/2}\, (1+\mod{\la_2})^{-N_2+(n_2-1)/2} \wrt\la. 
\end{aligned}
$$
Hence
$$
[1-\Phi_1(a_1)]\, [1-\Phi_2(a_2)] \, \bigmod{\Xi^{\infty,\infty}(a)} 
\leq C \bignorm{m_B}{\Mar{T_p}{N}}\, \frac{1-\Phi_1(a_1)}{\mod{\log a_1}^{N_1}} \,\frac{1-\Phi_2(a_2)}{\mod{\log a_2}^{N_2}}.  
$$
It is straightforward to check that the right hand side is in $\lu{\BA}$, as required.

Now we prove \rmiii.   By integrating by parts $N_2$ times with respect to $\la_2$, we see that 
$$
\Xi^{0,\infty}(a)
= (\log a_2)^{-N_2} \int_{\fra^*} \Psi_1(\la_1)\, a^{i\la} \, 
	\partial_{\la_2}^{N_2} \{(1-\Psi_2) [\mbcuu]_{\de(p)\rho}\} (\la) \wrt\la.   
$$
A straightforward calculation shows that the function 
$$
\la_1\mapsto \Psi_1(\la_1)\,\int_{\fra_2} a^{i\la_2}\, 
\partial_{\la_2}^{N_2} \{(1-\Psi_2) [\mbcuu]_{\de(p)\rho}\} (\la)\wrt \la_2
$$ 
satisfies a H\"ormander condition of order $N_1$ on $\fra_1^*$, and 
$$
\begin{aligned}
& \Bignorm{\Psi_1\,\int_{\fra_2} a^{i\la_2}\, \partial_{\la_2}^{N_2} \{(1-\Psi_2) 
	[\mbcuu]_{\de(p)\rho}\} (\cdot,\la_2)\wrt \la_2}{\Horm{\fra_1^*}{N_1}}  \\
& \leq C \bignorm{m_B}{\Mar{T_p}{N}} \int_{\fra_2} \frac{1-\Psi_2(\la_2)}{\mod{\la_2}^{N_2}}
	\, (1 +\mod{\la_2})^{(n_2-1)/2}  \wrt \la_2\\
& \leq C \bignorm{m_B}{\Mar{T_p}{N}}.  
\end{aligned}
$$
We have used the assumption that $N_2> (n_2+3)/2$ in the last inequality.
The classical H\"ormander theorem implies that $\Xi^{0,\infty}(\cdot, a_2)$ is in $Cv_q(\BA_1)$ and 
$$
\bignorm{\Xi^{0,\infty}(\cdot, a_2)}{Cv_q(\BA_1)}
\leq \frac{\bignorm{m_B}{\Mar{T_p}{N}}}{\mod{\log a_2}^{N_2}}.
$$
Therefore $\bignorm{[1-\Phi_2]\, \Xi^{0,\infty}}{\lu{\BA_2;Cv_q(\BA_1)}} \leq \bignorm{m_B}{\Mar{T_p}{N}}$, as required. 

Finally, the proof of \rmiv\ is similar to the proof of \rmiii\ with the roles of $\BA_1$ and $\BA_2$ interchanged.
We omit the details.  
\end{proof}

\begin{lemma} \label{l: prop of Xi II}
The functions 
$[1-\Phi_1][1-\Phi_2]\, \Xi^{0,0}$, $\Phi_1\, [1-\Phi_2]\, \Xi^{0,\infty}$, 
$[1-\Phi_1]\,[1-\Phi_2]\, \Xi^{0,\infty}$, $[1-\Phi_1]\,\Phi_2\, \Xi^{\infty,0}$ 
and $[1-\Phi_1]\,[1-\Phi_2]\, \Xi^{\infty,0}$ are in $Cv_q(\BA)$ for all $q$ in $(1,\infty)$.
Furthermore, $\phipuu$ is in $Cv_q(\BA)$ for all $q$ in $(1,\infty)$.

The convolution norm of each of the functions above is controlled by $C \bignorm{m_B}{\Mar{T_p}{N}}$.
\end{lemma}

\begin{proof}
By Theorem~\ref{t: Ap spaces}, functions in the Fig\`a-Talamanca--Herz space $A_p(\BA)$ are pointwise multipliers 
of~$\Cvp{\BA}$,  
i.e., $\phi$ in $A_p(\BA)$ and $k$ in $\Cvp{\BA}$ imply that 
\begin{equation} \label{f: Ap and Cvp}
\bignorm{\phi k}{\Cvp{\BA}} 
\leq \bignorm{\phi}{A_p(\BA)} \, \bignorm{k}{\Cvp{\BA}}.
\end{equation}
It is well known that for each $q$ in $(1,\infty)$ the space $C_c^\infty(\BA)$ 
and the functions $1\otimes \Phi_2$ and $\Phi_1\otimes 1$ are in $A_q(\BA)$ \cite{Co}. 
Thus, by Lemma~\ref{l: prop of Xi}~\rmii, $[\Phi_1\otimes\Phi_2] \, \Xi^{0,0}$, $[\Phi_1\otimes 1] \, \Xi^{0,0}$ and 
$[1\otimes\Phi_2] \, \Xi^{0,0}$ are in $Cv_q(\BA)$ for all $q$ in $(1,\infty)$.  Hence the same is true for
$[(1-\Phi_1)\otimes\Phi_2] \, \Xi^{0,0}$, $[\Phi_1\otimes (1-\Phi_2)] \, \Xi^{0,0}$  and 
$[(1-\Phi_1)\otimes(1-\Phi_2)] \, \Xi^{0,0}$.  
The desired control of the norm follows from \eqref{f: Ap and Cvp}. 

Similar considerations apply to all the other functions in the statement.  We omit the details.  

Finally, observe that $\phipuu = [(1-\Phi_1)\otimes(1-\Phi_2)]\, 
[\Xi^{0,0} + \Xi^{\infty,0} + \Xi^{0,\infty} + \Xi^{\infty,\infty}]$.  The first part of the lemma and 
Lemma~\ref{l: prop of Xi} imply that $\phipuu$ is in 
$Cv_q(\BA)$ for all $q$ in $(1,\infty)$, as required to conclude the proof of the lemma.  
\end{proof}

Consistently with the  notation of the previous sections, we denote by $\cD$ the Radon--Nikodym derivative 
of the action of $\BA$ on $\OV\BN$.  Notice that $\cD(v_1b_1,v_2b_2) = \cD_1(v_1b_1) \, \cD_2(v_2b_2)$,
where $\cD_1$ and $\cD_2$ are defined much as $\cD$, but are referred to the groups $\OV\BN_1 \BA_1$ and $\OV\BN_2 \BA_2$;
$\cD_1$ and its natural extension to $\OV\BN \BA$ 
(still denoted by $\cD_1$) have been already used in Section~\ref{s: loc-glob}.  

%

\noindent
The following lemma is the analogue for $\kuu$ of Lemma~\ref{l: multipliers loc glob}.  

\begin{lemma} \label{l: multipliers glob glob}
Suppose that $N> (n+\tre)/2$ and $1<p<2$.  Then there exists
a constant~$C$ such that for every $m$ in $\Mar{\TWp}{N}$ the following hold:
\begin{enumerate}
\item[\itemno1]
for all $a$ in $\BA^+$
$$
\begin{aligned}
\bigmod{\phipuu(a)}
& \leq C \bignorm{m}{\Mar{\TWp}{N}}\, \frac{1-\Phi_1(a_1)}{\log a_1}\, \frac{1-\Phi_2(a_2)}{\log a_2}  \\
\bigmod{a_1 \partial_{a_1}\phipuu(a)}
& \leq C \bignorm{m}{\Mar{\TWp}{N}}\, \frac{1-\Phi_1(a_1)}{\log^2 a_1}  \, \frac{1-\Phi_2(a_2)}{\log a_2} \\ 
\bigmod{a_2 \partial_{a_2}\phipuu(a)}
& \leq C \bignorm{m}{\Mar{\TWp}{N}}\, \frac{1-\Phi_1(a_1)}{\log a_1}  \, \frac{1-\Phi_2(a_2)}{\log^2 a_2} \\ 
\bigmod{a_1a_2 \partial_{a_1a_2}^2 \phipuu(a)}
& \leq C \bignorm{m}{\Mar{\TWp}{N}}\, \frac{1-\Phi_1(a_1)}{\log^2 a_1}  \, \frac{1-\Phi_2(a_2)}{\log^2 a_2}.
\end{aligned}
$$
Furthermore $\phipuu$, $\phipuu\chi_{\BA_1^+\times \BA_2^+}$, 
$\phipuu\chi_{\BA_1^-\times \BA_2^-}$, $\phipuu\chi_{\BA_1^-\times \BA_2^+}$ and $\phipuu\chi_{\BA_1^+\times \BA_2^-}$
are in $\Cvp{\BA}$;
\item[\itemno2]
$\kuu(a) = a^{-2\rho/p} \, \phipuu(a)$ for all $a$ in $\BA^+$, and 
$$
\bignorm{\chi_{\BA_2}^+\kuu(a_1, \cdot)}{\Cvp{\OBN_2\BA_2}} 
\leq C \bignorm{m}{\Mar{T_p}{N}} \frac{1-\Phi_1(a_1)}{\mod{\log a_1}} \,  a_1^{-2\rho_1/p};
$$  
\item[\itemno3]
$\ds \int_{\OBN}\int_{\BA} \bigmod{[\cD^{1/p}\kuu\chi_{\BA_1^-\times \BA_2^-}](vb)} \wrt v  \dtt b  
\leq C \bignorm{m}{\Horm{T_p}{N}}$, i.e., $\cD^{1/p}\kuu\chi_{\BA_1^-\times \BA_2^-}$ is in $\lu{\OBN;\lu{\BA}}$;
\item[\itemno4]
$
\ds\chi_{\BA_1}^+(b_1) \bignorm{\chi_{\BA_2}^+[\phipuu([v_1b_1]_+,\cdot)-\phipuu(b_1,\cdot)]}{\Cvp{\BA_2}}
\leq C  \bignorm{m_B}{\Mar{\TWp}{N}}\,  \frac{\mod{\al_1(H(v_1))}+1}{1+\log^2b_1} 
$
and a similar estimate holds with the roles of $\OBN_1\BA_1$ and $\OBN_2\BA_2$ reversed.  
\end{enumerate}
\end{lemma}

\begin{proof}
First we prove \rmi.  We adapt an idea of Ionescu \cite[Proof of Theorem~8]{I1}. 
We move the contour of integration from $\fra^*$ to $\fra^*+i\rho^{p,\vep}$ in the 
definition of $\phipuu$ (see \eqref{f: phip product II} above),
where $\rho^{p,\vep} = \de(p)\rho - (\vep(a_1)\rho_1,\vep(a_2)\rho_2)$; the function 
$\vep$ will be determined later.  We obtain 
$$
\phipuu(a_1,a_2)
= \big[1-\Phi_1(a_1)\big]\,\big[1-\Phi_2(a_2)\big]\, a^{\vep(a)\rho} 
\int_{\fra^*} a^{i\la} \, \mbcuu(\la+i\rho^{p,\vep}) \wrt\la.   
$$
%
Suppose that $a$ is in $\BA^+$.  In the integral above we integrate by parts $N_1$ times with respect to $\la_1$ and $N_2$
times with respect to $\la_2$, and obtain that 
$$
\bigmod{\phipuu(a_1,a_2)}
\leq \frac{1-\Phi_1(a_1)}{(\log a_1)^{N_1}} \,\frac{1-\Phi_2(a_2)}{(\log a_2)^{N_2}} \, a^{\vep(a)\rho} 
\int_{\fra^*} \bigmod{\partial_{\la_1}^{N_1}\partial_{\la_2}^{N_2}  \mbcuu(\la+i\rho^{p,\vep})}  \wrt\la.   
$$
A straightforward computation shows that  
$$
\bigmod{\partial_{\la_1}^{N_1}\partial_{\la_2}^{N_2}  \mbcuu(\la+i\rho^{p,\vep})}  
\leq C \bignorm{m}{\Horm{\TWp}{N}} \frac{(1+\mod{\la_1})^{(n_1-1)/2}}{\bigmod{\la_1-\vep(a_1)}^{N_1}} \, 
\frac{(1+\mod{\la_2})^{(n_2-1)/2}}{ \bigmod{\la_2-\vep(a_2)}^{N_2}}. 
$$
We insert this estimate in the integral above, 
choose $\vep (a_1) = \vep/\log a_1$ and $\vep (a_2) = \vep/\log a_2$, and obtain the required estimate
$$
\bigmod{\phipuu(a_1,a_2)}
\leq C \, \frac{1-\Phi_1(a_1)}{\log a_1} \,\frac{1-\Phi_2(a_2)}{\log a_2}.  
$$
The required estimates for the derivatives of $\phipuu$ on $\BA^+$ are proved similarly.  We omit the details.

As for the remaining part of \rmi, 
we have already proved that $\phipuu$ is in $\Cvp{\BA}$ (see Lemma~\ref{l: prop of Xi II}).  
Suppose now that $a$ is in $\BA^-$.  The term $a^{\de(p)\rho}$ in \eqref{f: phip product II} 
vanishes exponentially as $a$ tends to~$0$.  
An integration by parts similar to that performed above, but without moving the contour of integration, 
shows that $\ds \Bigmod{\int_{\fra^*} a^{i\la} \, \mbcuu(\la)\wrt\la}
\leq C \bignorm{m}{\Horm{\TWp}{N}} \mod{\log a_1}^{-N_1} \, \mod{\log a_2}^{-N_2}$.  
Thus, $\phipuu$ is in $\lu{\BA^-}$, i.e., $\phipuu\chi_{\BA_1^-\times\BA_2^-}$ is in $\lu{\BA}$, whence
in $Cv_q{\BA}$ for all $q$ in $(1,\infty)$. 

Next, if $a_1\in \BA_1^-$ and $a_2\in \BA_2$, we move the path of integration from $\fra^*$ to
$\fra^* + i (\de(p)-\vep(a_2))\rho_2$, where $\vep(a_2) = \vep/\log a_2$, and integrate by parts $N_1$ times
with respect to $\la_1$.  We obtain 
$$
\phipuu(a_1,a_2)
= \frac{1-\Phi_1(a_1)}{(\log a_1)^{N_1}} \,a_1^{\de(p)\rho_1} \, \int_{\fra_1^*} \!\!\wrt \la_1\, a_1^{i\la_1}
[1-\Phi_2(a_2)] \int_{\fra_2^*} a_2^{i\la_2} \, \partial_{\la_1}^{N_1}\mbcuu_{\de(p)\rho_2}(\la) \wrt \la_2
$$
for all $a_1$ in $\BA_1^-$ and for every $a_2$ in $\BA_2$.  A straightforward calculation shows that 
$$
\bignorm{\partial_{\la_1}^{N_1}\mbcuu_{\de(p)\rho_2}(\la_1,\cdot)}{\Horm{T_p^{(2)}}{N_2}}
\leq C \bignorm{m}{\Horm{\TWp}{N}} (1+\mod{\la_1})^{-N_1+(n_1-1)/2}
\quant \la_1\in \fra_1^*.  
$$
Lemma~\ref{l: phip basic} implies that for each $\la_1$ in $\fra_1^*$ the function $\ds a_2\mapsto 
[1-\Phi_2(a_2)] \int_{\fra_2^*} a_2^{i\la_2} \, \partial_{\la_1}^{N_1}\mbcuu_{\de(p)\rho_2}(\la) \wrt \la_2$
is in $Cv_q(\BA_2)$ for all $q$ in $(1,\infty)$, with norm controlled by 
$C \bignorm{m}{\Horm{\TWp}{N}} (1+\mod{\la_1})^{-N_1+(n_1-1)/2}$.  Hence
$$
\begin{aligned}
\bignorm{\phipuu}{\lu{\BA_1^-;Cv_q(\BA_2)}}
& \leq C \bignorm{m}{\Horm{\TWp}{N}}  \int_{\BA_1^-} \!\!\wrt a_1\, 
   \frac{1-\Phi_1(a_1)}{(\log a_1)^{N_1}} \,a_1^{\de(p)\rho_1} \, \int_{\fra_1^*} (1+\mod{\la_1})^{-N_1+(n_1-1)/2} 
   \wrt \la_1 \\ 
& \leq C \bignorm{m}{\Horm{\TWp}{N}}.   
\end{aligned}
$$
Consequently $\phipuu\chi_{\BA_1^-\times\BA_2}$ is in $Cv_q(\BA)$ for every $q$ in $(1,\infty)$.   We have already proved that 
$\phipuu\chi_{\BA_1^-\times\BA_2^-}$ is in $Cv_q(\BA)$ for all $q$ in $(1,\infty)$.  Therefore 
$\phipuu\chi_{\BA_1^-\times\BA_2^+} = \phipuu\chi_{\BA_1^-\times\BA_2} - \phipuu\chi_{\BA_1^-\times\BA_2^-}$ 
is in $Cv_q(\BA)$ for all $q$ in $(1,\infty)$.  

A similar argument, with the roles of $\BA_1$ and $\BA_2$ interchanged, shows that $\phipuu\chi_{\BA_1\times\BA_2^-}$ 
and $\phipuu\chi_{\BA_1^+\times\BA_2^-}$ are in $Cv_q(\BA)$ for all $q$ in $(1,\infty)$.  Thus,
$
\phipuu\chi_{\BA_1^+\times\BA_2^+} 
= \phipuu - \phipuu\chi_{\BA_1^+\times\BA_2^-}-\phipuu\chi_{\BA_1^-\times\BA_2^+}-\phipuu\chi_{\BA_1^-\times\BA_2^-}  
$
is in $Cv_q(\BA)$ for all $q$ in $(1,\infty)$.  This concludes the proof of \rmi.  

%

Next, we prove \rmii.  The required formula for $\kuu$ follows directly from the definitions
of $\kuu$ and~$\phipuu$.  In order to prove the required bound for $\bignorm{\chi_{\BA_2}^+\kuu(b_1v_1, \cdot)}{\Cvp{\OBN_2\BA_2}}$,
we move the contour of integration in the definition of $\phipuu$ from $\fra^*$ to $\fra^*+ i \rho_1^{p,\vep(a_1)}$,
where $\rho_1^{p,\vep(a_1)} = \big(\de(p)-\vep/\mod{\log a_1}\big)\rho_1$, and we see that for each $a$
in $\BA^+$ 
$$
\kuu(a)
= \cB(a_1) \, [1-\Phi_2(a_2)] \, a^{-2\rho/p} \, 
    a_2^{\de(p)\rho_2} \int_{\fra_1^*}\int_{\fra_2^*} a_1^{i\la_1}a_2^{i\la_2} \, 
    \mbcuu\big(\la_1+i\rho_1^{p,\vep(a_1)},\la_2\big) \wrt \la_1\wrt \la_2,  
$$
where we have set $\cB(a_1) := \big[1-\Phi_1(a_1)\big] \,  \e^{\mod{\rho_1}\vep(a_1)\sign(\log a_1)}$
for the sake of brevity.  
We integrate by parts $N_1$ times in the integral above on $\fra_1^*$, and find that 
$$
\begin{aligned}
\kuu(a)
& = \frac{\cB(a_1)}{(\log a_1)^{N_1}} \, \big[1-\Phi_2(a_2)\big]\,  a^{-2\rho/p} 
    a_2^{\de(p)\rho_2} \int_{\fra_1^*}\!\int_{\fra_2^*} a_1^{i\la_1}a_2^{i\la_2} \, 
    [\partial_{\la_1}^{N_1} \mbcuu](\la_1+i\rho_1^{p,\vep(a_1)},\la_2) \wrt \la_1\wrt \la_2 \\  
& = [1-\Phi_2(a_2)]  \int_{\fra_2^*} a_2^{i\la_2-\rho_2} \, 
	[\check\bc_2^{-1}\cU(a_1, \cdot)](\la_2) \wrt \la_2,
\end{aligned} 
$$
where 
$$
\cU(a_1,\la_2)
:= \frac{\cB(a_1)}{(\log a_1)^{N_1}} \,
a_1^{-2\rho_1/p} \, \int_{\fra_1^*} a_1^{i\la_1}\, \partial_{\la_1}^{N_1}
\, [\check\bc_1^{-1}m_B](\la_1+i\rho_1^{p,\vep(a_1)},\la_2) \wrt \la_1.   
$$
By Theorem~\ref{t: Transference principle} (applied to $\OBN_2\BA_2$) 
and Lemma~\ref{l: one dim multipliers glob}~\rmiv, 
$$
\bignorm{\chi_{\BA_2^+}\kuu(a_1,\cdot)}{\Cvp{\OBN_2\BA_2}} 
\leq \bignorm{\cD_2^{1/p}\chi_{\BA_2^+}\kuu(a_1,\cdot)}{\lu{\OBN_2;\Cvp{\BA_2}}} 
\leq C \bignorm{\cU(a_1,\cdot)}{\Horm{T_p^2}{N_2}}.   
$$
In order to estimate the right hand side, we are led to consider 
$\partial_{\la_2}^{j_2}\cU(a_1,\la_2)$ for $j_2\in \{0,1,\ldots,N_2\}$.    Observe that 
$$
\begin{aligned}
\bigmod{\partial_{\la_2}^{j_2}\cU(a_1,\la_2)}
& \leq  \frac{1-\Phi_1(a_1)}{\mod{\log a_1}^{N_1}} \,
	a_1^{-2\rho_1/p} \, \int_{\fra_1^*}  \bigmod{\partial_{\la_1}^{N_1}\partial_{\la_2}^{j_2}
	\, [\check\bc_1^{-1}m_B](\la_1+i\rho_1^{p,\vep(a_1)},\la_2)} \wrt \la_1  \\
& \leq  C\, \frac{\bignorm{m_B}{\Mar{T_p}{N}}}{\Theta_p^{(2)}(\la_2)^{j_2} }  \frac{1-\Phi_1(a_1)}{\mod{\log a_1}^{N_1}} \,
	a_1^{-2\rho_1/p} \, \int_{\fra_1^*}  
	\frac{(1+\mod{\la_1})^{(n_1-1)/2}}{\Theta_p^{(1)}(\la_1+i\rho_1^{p,\vep(a_1)})^{N_1}} \wrt \la_1;
\end{aligned}
$$
we have used the obvious bound $\bigmod{\cB(a_1)} \leq C \, [1-\Phi_1(a_1)]$ in the first inequality above,
and the assumption on $m_B$ together with the estimate \eqref{f: HCest} for the Harish-Chandra function $\check\bc_1^{-1}$
in the second.  Now, there exists a constant $C$ such that 
$$
\mod{\log a_1}^{-N_1}  \int_{\fra_1^*}  
	\frac{(1+\mod{\la_1})^{(n_1-1)/2}}{\Theta_p^{(1)}(\la_1+i\rho_1^{p,\vep(a_1)})^{N_1}} \wrt \la_1
\leq C\, \mod{\log a_1}^{-1} 
$$
(see \cite[p.~114--115]{I1}).  Altogether, we obtain the estimate 
$$
\bignorm{\cU(a_1,\cdot)}{\Horm{T_p^2}{N_2}}
\leq  C\, \bignorm{m_B}{\Mar{T_p}{N}}  \frac{1-\Phi_1(a_1)}{\mod{\log a_1}} \, a_1^{-2\rho_1/p},  
$$
whence
\begin{equation}
\bignorm{\chi_{\BA_2^+}\kuu(a_1,\cdot)}{\Cvp{\OBN_2\BA_2}} 
\leq  C\, \bignorm{m_B}{\Mar{T_p}{N}}  \frac{1-\Phi_1(a_1)}{\mod{\log a_1}} \, a_1^{-2\rho_1/p}.  
\end{equation}
The proof of \rmii\ is complete.  

Now we prove \rmiii.  Our proof is the analogue in our case of the proof that Ionescu gave in the rank one case.   
Recall that the Abel transform of a function $F$ on $\OBN\BA$ is defined by 
$\cA F(b) := b^\rho \,\int_{\OBN} \, F(vb) \wrt v$ for every $b$ in $\BA$.  Thus, 
$$
\int_{\OBN}\int_{\BA} \bigmod{[\cD^{1/p}\kuu\chi_{\BA_1^-\times \BA_2^-}](vb)} \wrt v  \dtt b  
= \int_{\BA^-}  b^{((2/p)-1)\rho}\,\cA \big(\mod{\kuu}\big)(b)  \dtt b.
$$
Since the Abel transform of a $\BK$--bi-invariant function is a Weyl-invariant function on $\BA$,
we may write $\cA \big(\mod{\kuu}\big)(b^{-1})$ instead of $\cA \big(\mod{\kuu}\big)(b)$ 
in the last integral above, and then change variables ($b' = b^{-1}$):  we obtain 
\begin{equation} \label{f: insert}
\int_{\OBN}\int_{\BA} \bigmod{[\cD^{1/p}\kuu\chi_{\BA_1^-\times \BA_2^-}](vb)} \wrt v  \dtt b  
= \int_{\BA^+}  (b')^{-\de(p)\rho}\,\cA \big(\mod{\kuu}\big)(b') \dtt {b'}.
\end{equation}
We deduce from \rmi\ and the fact that $[vb']_+^{-2\rho/p} \leq (b')^{-2\rho/p} \, P(v)^{2/p}$
for all $b'$ in $\BA^+$ (by the analogue of \eqref{f: decom Iwas} in the reduced case) the following estimate 
$$
\begin{aligned}
\cA \big(\mod{\kuu}\big)(b')
& \leq C \, \bignorm{\phipuu}{\ly{\BA^+}} \, (b')^{-\de(p)\rho}\, \int_{\OBN} P(v)^{2/p} \wrt v \\
& \leq C \bignorm{m}{\Horm{T_p}{N}} (b')^{-\de(p)\rho}
\quant b'\in \BA^+.  
\end{aligned}
$$  
We insert this estimate in the integral on the right hand side of \eqref{f: insert} and observe that 
the integral 
$\int_{\BA^+}  (b')^{2(1/p'-1/p)\rho}\,\dtt {b'}$ is convergent, because $1<p<2$.  The desired estimate follows.  

Finally, we prove \rmiv.  We prove the first of the two estimate.  The proof of the second is similar, and it is
obtained from the first by interchanging the roles of $\BX_1$ and $\BX_2$.  Observe that, trivially, 
$\ds\phipuu([v_1b_1]_+,b_2)-\phipuu(b_1,b_2) = \int_{b_1}^{[v_1b_1]_+} \!\! a_1\partial_{a_1} \phipuu(a_1,b_2) \dtt{a_1}$.
Recall that $\phipuu$ is defined in formula~\eqref{f: phip product II}.  Thus,
$$
\chi_{\BA_2^+} (b_2)\, a_1\partial_{a_1} \phipuu(a_1,b_2)
= \chi_{\BA_2^+} (b_2)\,  [1-\Phi_1(b_2)] \int_{\fra_2^*} b_2^{\de(p)+i\la_2} \, 
[\check \bc_2^{-1} \cU(a_1,\cdot)](\la_2) \wrt \la_2, 
$$
where $\cU(a_1,\la_2) = \Upsilon_0(a_1,\la_2) + \Upsilon_0(a_1,\la_2)$, 
$$
\Upsilon_0(a_1,\la_2) 
:= -\Phi_1'(a_1) \,  a_1^{1+\de(p)\rho_1} \, \int_{\fra_1^*}  a_1^{i\la_1}\,  
[\check\bc_1^{-1} m_B(\cdot,\la_2)](\la_1) \wrt\la_1
$$
and 
$$
\Upsilon_1(a_1,\la_2)
:= [1-\Phi_1(a_1)] \,  a_1^{\de(p)\rho_1} \,\int_{\fra_1^*} [\de(p)\rho_1+i\la_1] \,  a_1^{i\la_1}\,  
[\check\bc_1^{-1} m_B(\cdot,\la_2)](\la_1) \wrt\la_1.  
$$
By Lemma~\ref{l: phip basic}, 
\begin{equation} \label{f: cU Up}
\begin{aligned}
\bignorm{\chi_{\BA_2^+} \, a_1\partial_{a_1} \phipuu(a_1,\cdot)}{\Cvp{\BA_2}}
& \leq \bignorm{\cU(a_1,\cdot)}{\Horm{T_p^{(2)}}{N_2}} \\
& \leq \bignorm{\Upsilon_0(a_1,\cdot)}{\Horm{T_p^{(2)}}{N_2}} + \bignorm{\Upsilon_1(a_1,\cdot)}{\Horm{T_p^{(2)}}{N_2}}. 
\end{aligned}
\end{equation} 
%
%
We show how to bound $\bignorm{\Upsilon_1(a_1,\cdot)}{\Horm{T_p^{(2)}}{N_2}}$.  
The estimate of $\bignorm{\Upsilon_0(a_1,\cdot)}{\Horm{T_p^{(2)}}{N_2}}$ is easier, and the details are left to
the interested reader.  
In order to estimate $\bignorm{\Upsilon_1(a_1,\cdot)}{\Horm{T_p^{(2)}}{N_2}}$, we are led to consider,
for all $j_2 \in \{0,1,\ldots, N_2\}$,
$$
\partial_{\la_2}^{j_2}\Upsilon_1(a_1,\la_2)
= [1-\Phi_1(a_1)] \,  a_1^{\de(p)\rho_1} \,\int_{\fra_1^*} [\de(p)\rho_1+i\la_1] \,  a_1^{i\la_1}\,  
[\check\bc_1^{-1} \partial_{\la_2}^{j_2}m_B(\cdot,\la_2)](\la_1) \wrt\la_1.  
$$
We first shift the contour of integration from $\fra_1^*$ to $\fra_1^*+i[\de(p)-\vep(a_1)]\rho_1$ (here $\vep(a_1) = \vep/\log a_1$),
and then integrate by parts $N_1$ times with respect to $\la_1$.  We obtain that 
$$
\partial_{\la_2}^{j_2}\Upsilon_1(a_1,\la_2)
= \frac{1-\Phi_1(a_1)}{(i\log a_1)^{N_1}} \, a_1^{\vep(a_1)\rho_1}\,\vep(a_1)\rho_1  \int_{\fra_1^*}   a_1^{i\la_1}\,  
\partial_{\la_1}^{N_1} [\check\bc_1^{-1} \partial_{\la_2}^{j_2}m_B(\cdot,\la_2)](\la_1+i[\de(p)-\vep(a_1)]\rho_1 )\wrt \la_1.  
$$
By assumption,
$$
\begin{aligned}
& \bigmod{\partial_{\la_1}^{N_1} [\check\bc_1^{-1} \partial_{\la_2}^{j_2}m_B(\cdot,\la_2)](\la_1+i[\de(p)-\vep(a_1)]\rho_1 )}\\
& \leq C \bignorm{m_B}{\Mar{T_p}{N}} \, \Theta_p^{(1)}(\la_1+i[\de(p)-\vep(a_1)]\rho_1)^{-N_1} \,  
	\Theta_p^{(2)}(\la_2)^{-j_2} \, (1+\mod{\la_1})^{(n_1-1)/2}\\ 
& \leq C \bignorm{m_B}{\Mar{T_p}{N}} \, \bigmod{\la_1-\vep(a_1)\rho_1}^{-N_1} \,  
	\Theta_p^{(2)}(\la_2)^{-j_2} \, (1+\mod{\la_1})^{(n_1-1)/2},
\end{aligned}
$$
whence
$$
\begin{aligned}
\bignorm{\Upsilon_1(a_1,\cdot)}{\Horm{T_p^{(2)}}{N_2}}
& \leq C \bignorm{m_B}{\Mar{T_p}{N}} \, 
	\frac{1-\Phi_1(a_1)}{\mod{\log a_1}^{N_1}} \, \vep(a_1) \int_{\fra_1^*} 
	\frac{(1+\mod{\la_1})^{(n_1-1)/2}}{\bigmod{\la_1-\vep(a_1)\rho_1}^{N_1}} \wrt\la_1 \\  
& \leq C \bignorm{m_B}{\Mar{T_p}{N}} \, \frac{1-\Phi_1(a_1)}{\mod{\log a_1}^2}.  
\end{aligned}
$$
A similar estimate holds for $\bignorm{\Upsilon_0(a_1,\cdot)}{\Cvp{\OBN_2\BA_2}}$.  By combining the estimates
for $\Upsilon_0$ and $\Upsilon_1$ with \eqref{f: cU Up}, gives the required estimate.  

The proof of the lemma is complete.  
\end{proof}

\begin{proposition} \label{p: kB2}
There exists a constant $C$ such that the following hold:
\begin{enumerate}
\item[\itemno1]
$\bignorm{\koo}{\Cvp{\BX}} 
\leq C \bignorm{m_B}{\Mar{T_p}{N}}$;  
\item[\itemno2]
$\bignorm{\kuo}{\Cvp{\BX}} 
\leq C \bignorm{m_B}{\Mar{T_p}{N}}$ and a similar estimate holds for $\kou$;  
\item[\itemno3]
$\bignorm{\kuu}{\Cvp{\BX}} 
\leq C \bignorm{m_B}{\Mar{T_p}{N}}$.  
\end{enumerate}
\end{proposition}

\begin{proof}
First we prove \rmi.  We claim that 
$\ds \int_{\BG} \bigmod{\koo(g)} \, \vp_{i\de(p)\rho}(g) \wrt g
\leq C \bignorm{m}{\Horm{T_p}{N}}$.  The claim and Herz's majorizing principle imply that 
	$\koo$ is in $\Cvp{\BX}$, with the required norm estimate (see Lemma~\ref{l: iterated conv}~\rmiii). 

To prove the claim, let $\vep>0$ be a number to be determined in a moment.  In the formula defining $\koo$
(just below \eqref{f: dec for B2}) we move the path of integration from $\fra^*$ to $\fra^*+i\rho^{p,\vep}$ 
(here $\rho = (\rho_1,\rho_2)$ and $\rho^{p,\vep} = (\de(p)-\vep)\rho$), and obtain that 
$$
\koo(a_1,a_2) 
= [1-\Phi_1(a_1)]\, [1-\Phi_2(a_2)] \,  a_1^{(\vep-2/p)\rho_1-2\al_1} a_2^{(\vep-2/p)\rho_2-2\al_2}\, \Xi_\om(a_1,a_2),
$$
where
$$
\Xi_\om(a_1,a_2)
	= \int_{\fra^*} a_1^{i\la_1} \, a_2^{i\la_2} \, \om(\la+i\rho^{p,\vep},a) \, \mbcuu(\la+i\rho^{p,\vep})
	\wrt \la. 
$$
and $\om(\zeta,a) := \om_1(\zeta,a_1)\, \om_2(\zeta,a_2)$.  It is not hard to check that $\Xi_\om$ is bounded.  
This follows from a routine $N$-fold integration by parts, which gives
$$
\bigmod{\Xi_\om(a_1,a_2)}
\leq \frac{1}{\log^{N_1}\!a_1} \frac{1}{\log^{N_2}\!a_2}  \int_{\fra^*} \bigmod{[\partial_{\la_1}^{N_1} \partial_{\la_2}^{N_2} 
	\big[\om(\la+i\rho^{p,\vep},a) \, \mbcuu(\la+i\rho^{p,\vep})\big]} \wrt \la. 
$$
Now, the assumption that $N > (n+\tre)/2$ and the estimates \eqref{f: HCest} for the Harish-Chandra $\bc$ function 
and \eqref{f: estIonescu} for $\om$ imply that the integral above is bounded by 
$C \bignorm{m_B}{\Mar{T_p}{N}}$.  Thus, 
$$
\begin{aligned}
& \int_{\BG} \bigmod{\koo(g)} \, \vp_{i\de(p)\rho}(g) \wrt g \\
& \leq C \bignorm{m_B}{\Mar{T_p}{N}} \int_{\BA^+} \frac{1-\Phi_1(a_1)}{\log^{N_1}\!a_1} \frac{1-\Phi_2(a_2)}{\log^{N_2}\!a_2}  
	\,  a_1^{(\vep-2/p)\rho_1-2\al_1} a_2^{(\vep-2/p)\rho_2-2\al_2}\, \, \vp_{i\de(p)\rho}(a)\, \de(a) \wrt a.
\end{aligned}
$$
The last integral is easily seen to be convergent as long as $\vep$ is small enough, thereby completing the proof of \rmi.   

Next we prove \rmii.  We prove the estimate for $\kuo$; the estimate for $\kou$ can be proved in a similar way, with the roles
of $\BX_1$ and $\BX_2$ interchanged. 
We claim that 
$$
\int_{\BG_2} \bignorm{\kuo(\cdot,g_2)}{\Cvp{\BX_1}} \, \vp_{i\de(p)\rho}^{(2)}(g_2) \wrt g_2
\leq C \bignorm{m}{\Horm{T_p}{N}}.
$$  
The claim implies that $\kuo$ is in $\Cvp{\BX}$, with the required norm estimate,
by Lemma~\ref{l: iterated conv}~\rmii.
It is convenient to set
\begin{equation} \label{f: Nod}
\Nod(\la_1,a_2)
:= [1-\Phi_2(a_2)] \int_{\fra_2^*} a_2^{i\la_2-\rho_2-2\al_2}\,\om_2(\la_2,a_2)  
       \, \cdinvmB(\la_1,\la_2) \wrt \la_2. 
\end{equation} 
The function $\Nod$ is the analogue of $\Nou$, defined in \eqref{f: N}, but with the role of the variables 
reversed.  Fr{}om the formula defining $\kuo$ (just below formula \eqref{f: dec for B2}) we see that 
$$
\kuo(a_1,a_2)
= [1-\Phi_1(a_1)]\, \int_{\fra_1^*} a^{i\la_1-\rho_1} \, [\check\bc_1^{-1} \Nod](\la_1,a_2) \wrt \la_1.    
$$
 From Lemma~\ref{l: one dim multipliers glob}~\rmi\  we see that $\bignorm{\kuo(\cdot,a_2)}{\Cvp{\BX_1}}
\leq C \bignorm{\Nod(\cdot,a_2)}{\Horm{T_p^{(1)}}{N_1}}$.  In order to estimate the decay of the last norm
as $a_2$ tends to infinity, we move the path of integration from $\fra_2^*$ to $\fra_2^*+i\rho_2^{p,\vep}$ 
in the integral defining $\partial_{\zeta_1}^j\Nod$ (see \eqref{f: Nod} above), and obtain
$$
\partial_{\la_1}^j\!\Nod(\la_1,a_2)
= [1-\Phi_2(a_2)] \, a_2^{(\vep-2/p)\rho_2-2\al_2}
   \int_{\fra_2^*}\!\! a_2^{i\la_2}\,\om_2(\la_2+i\rho_2^{p,\vep}\!,a_2) 
   \, \cdinvderzetaujumB(\la_1,\la_2+i\rho_2^{p,\vep}) \wrt \la_2. 
$$
Much as in \rmi\ above, a routine $N_2$-fold integration by parts in the last integral shows that 
$$
\begin{aligned}
& \bigmod{\partial_{\la_1}^j\!\Nod(\la_1,a_2)} \\
& \leq \frac{1-\Phi_2(a_2)}{\log^{N_2}\!\! a_2} \, a_2^{(\vep-2/p)\rho_2-2\al_2}
     \int_{\fra_2^*}\!\! \bigmod{\partial_{\la_2}^{N_2}\big\{\om_2(\la_2+i\rho_2^{p,\vep}\!,a_2) 
     \, \cdinvderlaujumB(\la_1,\la_2+i\rho_2^{p,\vep})\big\}} \wrt \la_2 \\  
& \leq C \bignorm{m}{\Horm{T_p}{N}} \frac{1-\Phi_2(a_2)}{\log^{N_2}\!\! a_2} \, a_2^{(\vep-2/p)\rho_2-2\al_2}
	\, \Theta_p^{(1)}(\la_1)^{-j_1}.  
\end{aligned}
$$
By combining the estimates above, we obtain that 
$$
\begin{aligned}
& \int_{\BG_2} \bignorm{\kuo(\cdot,g_2)}{\Cvp{\BX_1}} \, \vp_{i\de(p)\rho_2}^{(2)}(g_2) \wrt g_2 \\
& \leq C \bignorm{m_B}{\Mar{T_p}{N}} \int_{\BA_2^+} \frac{1-\Phi_2(a_2)}{\log^{N_2}\!a_2}
	\,  a_2^{(\vep-2/p)\rho_2-2\al_2}\, \, \vp_{i\de(p)\rho_2}^{(2)}(a_2)\, \de_2(a_2) \wrt a_2.
\end{aligned}
$$
The last integral is easily seen to be convergent as long as $\vep$ is small enough, thereby completing the proof of \rmii.   

Finally, we prove \rmiii.  With a slight abuse of notation we denote by $\kuu$ both the function on $\BA^+$
defined just below formula \eqref{f: dec for B2} and its $\BK$--bi-invariant extension to $\BG$.  
Then we interpret $\kuu$ as a function on the semidirect product $\OBN\BA$. 
It is convenient to split $\kuu$ as follows
\begin{equation*} 
\kuu
= \kuu \chi_{\BA_1^-\times \BA_2^-} +  \kuu \chi^{-,+} +  \kuu \chi^{+,-} +  \kuu \chi^{+,+},
\end{equation*}
where, for the sake of brevity, we write $\chi_{\BA_1^-\times \BA_2^-}$ instead of $\chi_{\OBN_1\BA_1^-\times\OBN_2\BA_2^-}$,
and similarly for $\chi^{-,+}$, $\chi^{+,-}$ and $\chi^{+,+}$.
By Lemma~\ref{l: multipliers glob glob}~\rmiii\
$$
\bignorm{\kuu\chi_{\BA_1^-\times \BA_2^-}}{\Cvp{\OBN\BA}} 
\leq C  \bignorm{m}{\Mar{T_p}{N}}.    
$$
We shall prove similar estimates for $\kuu \chi^{-,+}$, $\kuu \chi^{+,-}$ and $\kuu \chi^{+,+}$.  

First we analyze $\kuu \chi^{-,+}$.  By Lemma~\ref{l: iterated conv}~\rmi\ and Theorem~\ref{t: Transference principle} 
(the latter applied to the group $\OBN_1\BA_1$) 
$$
\begin{aligned}
\bignorm{\kuu\chi^{-,+}}{\Cvp{\OBN\BA}} 
& \leq \Bignorm{\chi_{\BA_1^-}\bignorm{\chi_{\BA_2^+} \kuu}{\Cvp{\OBN_2\BA_2}}}{\Cvp{\OBN_1\BA_1}} \\
& \leq \int_{\OBN_1} \int_{\BA_1} \chi_{\BA_1^-}(b_1) \, \cD_1^{1/p}(b_1) 
	\bignorm{\chi_{\BA_2^+} \kuu(v_1b_1, \cdot)}{\Cvp{\OBN_2\BA_2}} \wrt v_1 \dtt{b_1}.
\end{aligned}
$$
By arguing as in the proof of Lemma~\ref{l: one dim multipliers glob}~\rmiii, we see that the right hand side may be rewritten as 
$\ds \int_{\BA_1^+} b_1^{-\de(p)\rho_1} \, \cA(\eta) (b_1) \dtt {b_1}$, where we have set 
$\eta(v_1b_1) :=  \bignorm{\chi_{\BA_2^+} \kuu(v_1b_1, \cdot)}{\Cvp{\OBN_2\BA_2}}$.
Recall the estimate 
$$
\bignorm{\chi_{\BA_2^+}\kuu(v_1b_1,\cdot)}{\Cvp{\OBN_2\BA_2}} 
\leq C \bignorm{m_B}{\Mar{T_p}{N}}  \frac{1-\Phi_1(a_1)}{\mod{\log a_1}} \, a_1^{-2\rho_1/p}.  
$$
proved in Lemma~\ref{l: multipliers glob glob}~\rmii, where $a_1 = [v_1b_1]_+$.  Since $b_1$ is in $\BA_1^+$,
$([v_1b_1]_+)^{-2\rho_1/p} \leq b_1^{-2\rho_1/p} \, P_1(v_1)^{2/p}$ (see \eqref{f: decom Iwas}).
By combining these estimates, we obtain the bound 
$$
\begin{aligned}
\bignorm{\kuu\chi^{-,+}}{\Cvp{\OBN\BA}} 
& \leq C \bignorm{m_B}{\Mar{T_p}{N}}  \int_{\BA_1^+} b_1^{-2\de(p)\rho_1} \, \dtt {b_1}
	\, \int_{\OBN_1} P_1(v_1)^{2/p}	\wrt v_1 \\  
& \leq C \bignorm{m_B}{\Mar{T_p}{N}};  
\end{aligned}
$$
the last inequality follows from the fact that $1<p<2$, and that $P^{1+\vep}$ is in $\lu{\OBN_1}$
for each $\vep>0$.  

A similar argument, with the roles of $\OBN_1\BA_1$ and $\OBN_2\BA_2$ interchanged, proves the estimate
$$
\bignorm{\kuu\chi^{+,-}}{\Cvp{\OBN\BA}} 
\leq C \bignorm{m_B}{\Mar{T_p}{N}}.  
$$
Thus, it remains to analyses $\kuu \chi^{+,+}$. 
By Theorem~\ref{t: Transference principle} (with $\OBN$ in place of $N$ and $\BA$ in place of~$H$),  
\begin{equation*} 
\bignorm{\kuu\chi^{+,+}}{\Cvp{\OBN\BA}} 
\leq \int_{\OBN} \bignorm{\cD^{1/p}\,(\kuu)_v\chi^{+,+}}{\Cvp{\BA}} \wrt v.  
\end{equation*} 
Recall that $\cD$ denotes the function on $\OBN\BA$ defined by $\cD(vb) = b^{2\rho}$, and $(\kuu)_v(b) := \kuu(vb)$ 
for each $v$ in $\OBN$ and $b$ in $\BA$.  

Notice that if $b_1$ is in $\BA_1^+$ and $b_2$ is in $\BA_2^+$, then 
formula \eqref{f: decom Iwas} (used twice) yields 
$$
\begin{aligned}
(\cD^{1/p}\kuu)(vb)
& = b_1^{2\rho_1/p} b_2^{2\rho_2/p}\,  [v_1b_1]_+^{-2\rho_1/p}\, [v_2b_2]_+^{-2\rho_2/p} \, \phipuu([v_1b_1]_+,[v_2b_2]_+) \\
& = P(v_1)^{2/p} P(v_2)^{2/p}\, \cE_1(v_1,b_1)  \, \cE_2(v_2,b_2) \, \phipuu([v_1b_1]_+,[v_2b_2]_+),
\end{aligned}
$$
where $\cE_1(v_1,b_1) := \exp[E_1(v_1,b_1)H_0^{(1)}]$ and $\cE_2(v_2,b_2) := \exp[E_2(v_2,b_2)H_0^{(2)}]$. 
Also, we write 
$$
\cE_1(v_1,b_1) \, \cE_2(v_2,b_2) 
= \big[\cE_1(v_1,b_1)-1\big]  \, \big[\cE_2(v_2,b_2)-1\big] 
  + \big[\cE_1(v_1,b_1)-1\big]  
  + \big[\cE_2(v_2,b_2)-1\big] 
  +1.
$$
Correspondingly, $\cD^{1/p}\kuu\chi^{+,+}$ may be written as the sum of four terms.  The $\lu{\OBN;\lu{\BA}}$
norm of the first may be estimated as follows
$$
\begin{aligned}
& \int_{\OBN}\int_{\BA^+} P(v_1)^{2/p} P(v_2)^{2/p}\, \big[\cE_1(v_1,b_1)-1\big]  \, \big[\cE_2(v_2,b_2)-1\big] 
	 \bigmod{\phipuu([v_1b_1]_+,[v_2b_2]_+)} \wrt v_1\wrt v_2 \dtt{b_1} \dtt{b_2} \\
& \leq C  \bignorm{\phipuu}{\ly{\BA^+}} \int_{\OBN}\int_{\BA^+} P(v_1)^{2/p} P(v_2)^{2/p}\,  (b_1b_2)^{-2} 
	 \wrt v_1\wrt v_2 \dtt{b_1} \dtt{b_2} \\
& \leq C \bignorm{m}{\Mar{T_p}{N}}.   
\end{aligned}
$$
Next we consider the second term 
$\chi^{+,+}(b) \, \big[\cE_1(v_1,b_1)-1\big] \, P(v_1)^{2/p} P(v_2)^{2/p}\, \phipuu([v_1b_1]_+,[v_2b_2]_+)$,
which we write as the sum of 
$$
T_1(vb)
:= \chi^{+,+}(b) \,\big[\cE_1(v_1,b_1)-1\big] \, P(v_1)^{2/p} P(v_2)^{2/p}\, 
\big[\phipuu([v_1b_1]_+,[v_2b_2]_+)-\phipuu([v_1b_1]_+,b_2)\big]
$$ 
and 
$
T_2(vb)
:= \chi^{+,+}(b) \,\big[\cE_1(v_1,b_1)-1\big] \, P(v_1)^{2/p} P(v_2)^{2/p}\, \phipuu([v_1b_1]_+,b_2).
$  
We \emph{claim} that 
\begin{equation*} 
\bignorm{\chi_{\BA_2^+}\phipuu([v_1b_1]_+,\cdot)}{\Cvp{\BA_2}}
\leq C \bignorm{m_B}{\Mar{T_p}{N}}.  
\end{equation*}
Indeed, it is straightforward to check that we may write 
$$
\phipuu(a_1,b_2)
= [1-\Phi_2(b_2)] \, b_2^{\de(p)\rho_2} \int_{\fra_2^*} b_2^{i\la_2} \, 
	[\check\bc_2^{-1}\cU(a_1; \cdot)](\la_2) \wrt \la_2,
$$
where 
$$
\cU(a_1;\la_2)
:= [1-\Phi_1(a_1)] \, a_1^{\de(p)\rho_1} \, \int_{\fra_1^*} a_1^{i\la_1}\, 
   [\check\bc_1^{-1}m_B(\cdot,\la_2)](\la_1) \wrt \la_1.   
$$
By Lemma~\ref{l: phip basic},  
$$
\bignorm{\chi_{\BA_2^+}\phipuu(a_1,\cdot)}{\Cvp{\BA_2}} 
\leq C \bignorm{\cU(a_1;\cdot)}{\Horm{T_p^{(2)}}{N_2}}    
$$
Thus, we need to estimate the right hand side.  Notice that 
$$
\partial_{\la_2}^{j_2} \cU(a_1;\la_2)
:= [1-\Phi_1(a_1)] \, a_1^{\de(p)\rho_1} \, \int_{\fra_1^*} a_1^{i\la_1}\, 
   [\check\bc_1^{-1}\partial_{\la_2}^{j_2} m_B(\cdot,\la_2)](\la_1) \wrt \la_1.   
$$
Another application of Lemma~\ref{l: phip basic} (with $\partial_{\la_2}^{j_2} m_B(\cdot;\la_2)$ in place of $m$) gives 
$$
\begin{aligned}
\bigmod{\partial_{\la_2}^{j_2} \cU(a_1;\la_2)} 
& \leq C \bignorm{\partial_{\la_2}^{j_2} m_B(\cdot;\la_2)}{\Horm{T_p^{(1)}}{N_1}}  \, \frac{1-\Phi_1(a_1)}{\log a_1} \\
& \leq C \bignorm{m_B}{\Mar{T_p}{N}} \, \Theta_p^{(2)}(\la_2)^{-j_2} \, \frac{1-\Phi_1(a_1)}{\log a_1} \\
\end{aligned}
$$
for every $a_1$ in $\BA_1^+$.  This implies 
$\ds\bignorm{\cU(a_1;\cdot)}{\Horm{T_p^{(2)}}{N_2}} \leq C \bignorm{m_B}{\Mar{T_p}{N}} \,\frac{1-\Phi_1(a_1)}{\log a_1}$,
thereby proving the claim.  

The claim implies that 
$$
\begin{aligned}
\bignorm{T_2(v_1b_1,\cdot)}{\Cvp{\OBN_2\BA_2}} 
& \leq \bignorm{T_2(v_1b_1,\cdot)}{\lu{\OBN_2;\Cvp{\BA_2}}} \\
& \leq C \bignorm{m}{\Mar{T_p}{N}} \, \frac{\chi_{\BA_1^+}(b_1)}{b_1^2} 
	P(v_1)^{2/p} \int_{\OBN_2} P(v_2)^{2/p}\wrt v_2,
\end{aligned}
$$
whence 
$$
\begin{aligned}
\bignorm{T_2}{\Cvp{\OBN\BA}} 
& \leq \bignorm{T_2}{\lu{\OBN_1\BA_1;\Cvp{\OBN_2\BA_2}}} \\
& \leq C \bignorm{m}{\Mar{T_p}{N}} \, \int_{\BA_1^+} b_1^{-2} \wrt b_1 \int_{\OBN_1} P(v_1)^{2/p}\wrt v_1 \\
& \leq C \bignorm{m}{\Mar{T_p}{N}}.  
\end{aligned}
$$
Furthermore, 
$$
\begin{aligned}
\bigmod{\phipuu([v_1b_1]_+,[v_2b_2]_+)-\phipuu([v_1b_1]_+,b_2)}
& \leq \int_{b_2}^{[v_2b_2]_+} a_2\bigmod{\partial_{a_2}\phipuu([v_1b_1]_+,a_2)} \dtt{a_2} \\
\hbox{(by Lemma~\ref{l: multipliers glob glob}~\rmi)} \qquad
	& \leq C \bignorm{m}{\Mar{\TWp}{N}}\, \frac{1-\Phi_1([v_1b_1]_+)}{\log [v_1b_1]_+}  \, \int_{b_2}^{[v_2b_2]_+} 
	\frac{1-\Phi_2(a_2)}{\log^2 a_2} \dtt{a_2} \\ 
& \leq C \bignorm{m}{\Mar{\TWp}{N}}\,  \frac{1+\mod{\al_2(H(v_2))}}{1+ \log^2 b_2},
\end{aligned}
$$ 
where $C$ does not depend on $v$ and $b$.  Therefore 
$$
\begin{aligned}
 \bignorm{T_1}{\lu{\OBN\BA}} 
& \leq C \bignorm{m}{\Mar{T_p}{N}} \, \int_{\BA^+} 
     [1+ \log^2 b_2]^{-2} \frac{\wrt b_1}{b_1^2} \dtt{b_2} \int_{\OBN} P(v_1)^{2/p}[1+\mod{\al_2(H(v_2))}] P(v_2)^{2/p}\wrt v_1 
	\wrt v_2 \\
& \leq C \bignorm{m}{\Mar{T_p}{N}},  
\end{aligned}
$$
which obviously implies the estimate $\bignorm{T_1}{\Cvp{\OBN\BA}} \leq C \bignorm{m}{\Mar{T_p}{N}}$.  This
completes the estimate of the second term.  

The third term $\big[\cE_2(v_2,b_2)-1\big] \, P(v_1)^{2/p} P(v_2)^{2/p}\, \phipuu([v_1b_1]_+,[v_2b_2]_+)$
may be estimated much as the second term.  We omit the details. 

Finally, we consider $P(v_1)^{2/p} P(v_2)^{2/p}\, \phipuu([v_1b_1]_+,[v_2b_2]_+)$.  We write
$$
\begin{aligned}
\phipuu([v_1b_1]_+,[v_2b_2]_+) 
& = \big\{\phipuu([v_1b_1]_+,[v_2b_2]_+) - \phipuu(b_1,[v_2b_2]_+) + \phipuu([v_1b_1]_+,b_2) - \phipuu(b_1,b_2)\big\} \\
& \quad- \big\{\phipuu([v_1b_1]_+,b_2) - \phipuu(b_1,b_2)\big\} \\
& \quad	+ \big\{\phipuu(b_1,[v_2b_2]_+) - \phipuu(b_1,b_2) \big\}  \\
& \quad + \phipuu(b_1,b_2);
\end{aligned}
$$  
correspondingly, we write $P(v_1)^{2/p} P(v_2)^{2/p}\, \phipuu([v_1b_1]_+,[v_2b_2]_+)$ as the sum of the following four  
terms
$$
\begin{aligned}
\taupuuu (v_1b_1,v_2b_2)
& = \chi^{+,+}(b_1,b_2) \, P(v_1)^{2/p} \, P(v_2)^{2/p} \, 
    \big\{\phipuu([v_1b_1]_+,[v_2b_2]_+) - \phipuu(b_1,[v_2b_2]_+)\\
& \qquad + \phipuu([v_1b_1]_+,b_2) - \phipuu(b_1,b_2)\big\} \\
\taupduu (v_1b_1,v_2b_2)
& = \chi^{+,+}(b_1,b_2) \, P(v_1)^{2/p} \, P(v_2)^{2/p} \, \big\{ \phipuu(b_1,b_2) -\phipuu([v_1b_1]_+,b_2) \big\} \\
\tauptuu (v_1b_1,v_2b_2)
& = \chi^{+,+}(b_1,b_2) \, P(v_1)^{2/p} \, P(v_2)^{2/p} \, 
    \big\{\phipuu(b_1,[v_2b_2]_+) - \phipuu(b_1,b_2)\big\} \\
\taupquu (v_1b_1,v_2b_2) 
& = \chi^{+,+}(b_1,b_2) \, P(v_1)^{2/p} \, P(v_2)^{2/p} \,\phipuu(b_1,b_2). 
\end{aligned}
$$
Clearly, it suffices to prove that the convolution norm of each of the four terms above is bounded by $C \bignorm{m}{\Mar{T_p}{N}}$.
First we consider $\taupuuu$.  Observe that 
$$
\begin{aligned}
& \phipuu([v_1b_1]_+,[v_2b_2]_+) - \phipuu(b_1,[v_2b_2]_+) + \phipuu([v_1b_1]_+,b_2) - \phipuu(b_1,b_2) \\
& = \int_{b_1}^{[v_1b_1]_+}\! \int_{b_2}^{[v_2b_2]_+} \partial_{a_1a_2}^2 \phipuu(a_1,a_2) \wrt{a_1}\! \wrt{a_2}.
\end{aligned}
$$
By the estimate for $\bigmod{\partial_{a_1a_2} \phipuu(a_1,a_2)}$ proved in 
Lemma~\ref{l: multipliers glob glob}~\rmi, 
$$
\begin{aligned}
& \bignorm{\taupuuu}{\Cvp{\OBN\BA}} \\
& \leq C \bignorm{m}{\Mar{T_p}{N}} \int_{\OBN_1}\! \int_{\OBN_2} P(v_1)^{2/p} \, P(v_2)^{2/p}  \wrt v_1\! \wrt v_2 
     \int_{b_1}^{[v_1b_1]_+}\! \int_{b_2}^{[v_2b_2]_+} \frac{1-\Phi_1(a_1)}{\log^2 a_1}  \, 
     \frac{1-\Phi_2(a_2)}{\log^2 a_2} \dtt{a_1}\! \dtt{a_2}.  
\end{aligned}  
$$
Notice that 
$$
\begin{aligned}
\int_{b_1}^{[v_1b_1]_+}\! \int_{b_2}^{[v_2b_2]_+} \frac{1-\Phi_1(a_1)}{\log^2 a_1}  \, 
\frac{1-\Phi_2(a_2)}{\log^2 a_2} \dtt{a_1}\! \dtt{a_2} 
& \leq C \int_{b_1}^{[v_1b_1]_+}\! \int_{b_2}^{[v_2b_2]_+} \frac{1}{1+\log^2 a_1}  \, 
     \frac{1}{1+\log^2 a_2} \dtt{a_1}\! \dtt{a_2}\\
& \leq C \, \frac{\log([v_1b_1]_+)-\log{b_1}}{1+\log^2 b_1}\,\, \frac{\log([v_2b_2]_+)-\log{b_2}}{1+\log^2 b_2}.
\end{aligned}
$$
Now, we apply formula \eqref{f: decom Iwas} to the symmetric spaces $\BX_1$~and~$\BX_2$,
and obtain that $\log([v_1b_1]_+)-\log{b_1} \leq C\, [\mod{\al_1(H(v_1))} + 1]$ and 
$\log([v_2b_2]_+)-\log{b_2} \leq C\, [\mod{\al_2(H(v_2))} + 1]$.  Here we have used the fact
that $b_1\in \BA_1^+$ and $b_2\in \BA_2^+$.   By combining these estimates, we see that 
$$
\begin{aligned}
\bignorm{\taupuuu}{\Cvp{\OBN\BA}}
& \leq C \bignorm{m}{\Mar{T_p}{N}} \int_{\OBN_1}\! [\mod{\al_1(H(v_1))} + 1]\, P(v_1)^{2/p} \wrt v_1 
   \int_{\OBN_2} [\mod{\al_2(H(v_2))} + 1] \,  P(v_2)^{2/p} \wrt v_2 \\ 
& \leq C \bignorm{m}{\Mar{T_p}{N}}.  
\end{aligned}
$$

Next we consider $\taupduu$.  By Lemma~\ref{l: multipliers glob glob}~\rmiv, 
$$
\bignorm{\taupduu(v_1b_1,v_2\cdot)}{\Cvp{\BA_2}} 
\leq C  \bignorm{m_B}{\Mar{\TWp}{N}}\, \chi_{\BA_1^+}(b_1)\, P(v_1)^{2/p} P(v_2)^{2/p}  \frac{\mod{\al_1(H(v_1))}+1}{1+\log^2b_1}.
$$
Hence
$$
\begin{aligned}
& \bignorm{\taupduu(v_1b_1,v_2\cdot)}{\Cvp{\BX}} \\ 
& \leq \int_{\OBN} \int_{\BA_1} \bignorm{\taupduu(v_1b_1,v_2\cdot)}{\Cvp{\BX}} \dtt{b_1} \wrt v_1\wrt v_2\\
& \leq C  \bignorm{m_B}{\Mar{\TWp}{N}}\, \int_{\BA_1^+}\frac{\wrt b_1}{1+\log^2b_1} \, 
	\int_{\OBN_1}P(v_1)^{2/p} \wrt v_1\int_{\OBN_2}[\mod{\al_1(H(v_1))}+1] \, P(v_2)^{2/p} \wrt v_2 \\
& \leq C  \bignorm{m_B}{\Mar{\TWp}{N}},
\end{aligned}
$$
as required. 

By Lemma~\ref{l: multipliers glob glob}~\rmi, $\phipuu\chi^{+,+}$ is in $\Cvp{\BA}$, 
with norm controlled by $C \bignorm{m}{\Mar{T_p}{N}}$.  Thus, by Theorem~\ref{t: Transference principle},  
$$
\begin{aligned}
\bignorm{\taupquu}{\Cvp{\OBN\BA}} 
& \leq \int_{\OBN} \bignorm{(\taupquu)_v}{\Cvp{\BA}} \wrt v \\
& \leq C \bignorm{m}{\Mar{T_p}{N}} \int_{\OBN_1} \int_{\OBN_2} P(v_1)^{2/p} \, P(v_2)^{2/p}  \wrt v_1 \wrt v_2 \\
& \leq C \bignorm{m}{\Mar{T_p}{N}}.  
\end{aligned}
$$
This concludes the proof of \rmiii, and of the proposition.  
\end{proof}

\subsection*{Acknowledgments}
Both authors were supported by the Italian PRIN 2011 and PRIN 2015 projects 
\emph{Real and complex manifolds: geometry, topology and harmonic analysis}.
The research of the second author was  also supported by National Science Centre, Poland (NCN), research project 2018\slash 31\slash B\slash ST1\slash 00204.
The research was initiated while the second author was an Assegnista di ricerca at the Universit\`a
di Milano-Bicocca and continued when he was an HCM Postdoc at Universit\"{a}t Bonn.


\begin{thebibliography}{ccccccc}

\bibitem[A1]{A1} J.-Ph. Anker,
\textit{$L_p$ Fourier multipliers on Riemannian symmetric spaces
of the noncompact type}, Ann. of Math. \textbf{132} (1990),
597--628.

\bibitem[A2]{A2} J.-Ph. Anker,
\textit{Sharp estimates for some functions of the Laplacian
on noncompact symmetric spaces}, Duke Math. J. \textbf{65} (1992), 257--297.

\bibitem[AJ]{AJ} J.-Ph. Anker and L. Ji, \textit{Heat kernel and Green
function estimates on noncompact symmetric spaces I},
Geom. Funct. Anal. \textbf{9} (1999), 1035--1091.

\bibitem[AL]{AL} J.-Ph. Anker and N. Lohou\'e,
\textit{Moltiplicateurs sur certain espaces sym\'etriques},
Amer. J. Math \textbf{108} (1986), 1303--1354.


\bibitem[CMW]{CMW} D. Celotto, S. Meda and B. Wr\'obel, \textit{$L^p$ spherical multipliers on 
homogeneous trees},
Studia Math. {\bf 247} (2019), 175--190.


\bibitem[CS]{CS} J.-L. Clerc and E.M. Stein,
\textit{$L^p$ multipliers for noncompact symmetric spaces},
Proc. Nat. Acad. Sci. U. S. A. \textbf{71} (1974), 3911--3912.


\bibitem[CW]{CW} R.R. Coifman and G. Weiss, \textit{Transference methods in analysis},
Conference Board of the Mathematical Sciences Regional Conference Series in Mathematics,
No. \textbf{31}, American Mathematical Society, Providence, R.I., 1976.

\bibitem[Co]{Co} M.G. Cowling, 
\textit{Some applications of Grothendieck's theory of topological tensor products in 
Harmonic Analysis},
Math. Ann. \textbf{232} (1978), 273--285.

\bibitem[CGM1]{CGM1} M.G. Cowling, S. Giulini and S. Meda,
\textit{Estimates for functions of the Laplace--Beltrami operator on
noncompact symmetric spaces. II},
J. Lie Th. \textbf{5} (1995), 1--14.


\bibitem[GV]{GV} R. Gangolli and V.S. Varadarajan,
Harmonic Analysis of Spherical Functions on
Real Reductive Groups, Springer-Verlag, 1988.

\bibitem[GMM]{GMM} S. Giulini, G. Mauceri and S. Meda,
$L^p$ multipliers on noncompact symmetric spaces, \emph{J.
reine angew. Math.} \textbf{482} (1997), 151--175.



\bibitem[H1]{H1} S. Helgason,
Groups and Geometric Analysis,
Academic Press, New York, 1984.

\bibitem[H2]{H2} S. Helgason,
Differential Geometry, Lie groups, and Symmetric Spaces,
Academic Press, New York, 1978.

\bibitem[H3]{H3} S.~Helgason,
Geometric analysis on symmetric spaces,
Math. Surveys \& Monographs \textbf{39}, Amer. Math. Soc., 1994.

\bibitem[HR]{HR} E. Hewitt and K.A. Ross,
Abstract Harmonic Analysis , A Series of
Comprehensive Studies in Mathematics 1 (1979), n. 115, Springer- Verlag.

\bibitem[Ho]{Ho} L. H\"ormander,
\textit{Estimates for translation invariant operators in $L^p$ spaces},
Acta Math. \textbf{104} (1960), 93--140.

\bibitem[I1]{I1} A.D. Ionescu, \textit{Fourier integral operators on noncompact
symmetric spaces of real rank one}, J. Funct. Anal.
\textbf{174} (2000), 274--300.

\bibitem[I2]{I2} A.D. Ionescu, \textit{Singular integrals on symmetric spaces
of real rank one}, Duke Math. J.
\textbf{114} (2002), 101--122.

\bibitem[I3]{I3} A.D. Ionescu, \textit{Singular integrals
on symmetric spaces, II}, Trans. Amer. Math. Soc.
\textbf{335} (2003), 3359--3378.

\bibitem[L]{L} N.N. Lebedev, Special Functions and Their Applications, Dover, New York, 1972.

\bibitem[MV]{MV} S. Meda and M. Vallarino, \textit{Weak type estimates for spherical
multipliers on noncompact symmetric spaces},
\emph{Trans. Amer. Math. Soc.} \textbf{362} (2010), no. 6, 2993--3026.

\bibitem[ST]{ST} R.J.~Stanton, P.A.~Tomas,
\textit{Expansions for spherical functions on noncompact symmetric spaces},
Acta Math. \textbf{140} (1978), 251--276.

\bibitem[St1]{St1} E.M. Stein, \textit{Harmonic Analysis. Real variable
methods, orthogonality and oscillatory integrals}, Princeton Math. Series
No. \textbf{43}, Princeton N. J., 1993.

\bibitem[Str]{Str} J.-O.~Str\"omberg,
\textit{Weak type $L^1$ estimates for maximal functions on non-compact
symmetric spaces}, Ann. of Math.
\textbf{114} (1981), 115--126.


\bibitem[TV]{TV} P.C. Trombi and V.S. Varadarajan,
\textit{Spherical transforms on semisimple Lie groups},
Ann.\break of Math. \textbf{94} (1971), 246--303.

\bibitem[W]{W} G.N. Watson, A treatise on the theory of Bessel functions, Cambridge Univ. Press
Cambridge, 2nd edition, 1944.

\bibitem[Wr] {WroJSMMSO} B.\ Wr\'obel,  \textit{ Joint spectral multipliers for mixed systems of operators}, 
J. Fourier Anal. Appl. (2) {\bf 23} (2017),  245--287.
\end{thebibliography}
\end{document}